\documentclass[12pt]{article}
\usepackage{amssymb,amsmath,amsthm,mathrsfs,amscd}
\usepackage[usenames]{color}
\usepackage{hyperref}
\usepackage[all]{xy}
\usepackage{xypic}
\usepackage{graphicx}
\usepackage{amscd}
\usepackage{enumerate}
\usepackage{verbatim}
\usepackage[shortlabels]{enumitem}

\newcommand{\nir}[1]{{\color{red}{Nir: #1}}}

\newtheorem{theorem}{Theorem}[section]
\newtheorem{proposition}[theorem]{Proposition}
\newtheorem{corollary}[theorem]{Corollary}
\newtheorem{lemma}[theorem]{Lemma}
\newtheorem{assumption}[theorem]{Assumption}
\newtheorem{setting}[theorem]{Setting}
\newtheorem*{lemma*}{Lemma}
\newtheorem{definition}[theorem]{Definition}
\newtheorem{notation}[theorem]{Notation}

\newtheorem{remark}[theorem]{Remark}
\newtheorem{conjecture}[theorem]{Conjecture}

\newtheorem{claim}[theorem]{Claim}

\def\N{\Bbb N}
\def\Q{\Bbb Q}
\def\Z{\Bbb Z}

\def\Gr{G}

\DeclareMathOperator{\Aut}{Aut}
\DeclareMathOperator{\Cent}{Cent}

\DeclareMathOperator{\disc}{disc}
\DeclareMathOperator{\dist}{dist}

\DeclareMathOperator{\ccl}{gcl}
\DeclareMathOperator{\gcl}{gcl}
\DeclareMathOperator{\id}{id}
\DeclareMathOperator{\Mat}{Mat}

\DeclareMathOperator{\SL}{SL}
\DeclareMathOperator{\SO}{SO}

\DeclareMathOperator{\Span}{Span}
\DeclareMathOperator{\Spec}{Spec}
\DeclareMathOperator{\Spin}{Spin}

\DeclareMathOperator{\val}{val}
\DeclareMathOperator{\Th}{Th}

\DeclareMathOperator{\supp}{supp}
\DeclareMathOperator{\Ad}{Ad}
\DeclareMathOperator{\Id}{Id}
\DeclareMathOperator{\GL}{GL}

\DeclareMathOperator{\Stab}{Stab}

\DeclareMathOperator{\rank}{rank}
\DeclareMathOperator{\PSL}{PSL}
\DeclareMathOperator{\Lie}{Lie}

\newcounter{my_counter}

\title{On the model theory of higher rank arithmetic groups}
\author{Nir Avni and Chen Meiri}
\date{\today}

\begin{document}

\maketitle
\begin{abstract} 
Let $\Gamma$ be a centerless irreducible higher rank arithmetic lattice in characteristic zero. We prove that if $\Gamma$ is either non-uniform or is uniform of orthogonal type and dimension at least 9, then $\Gamma$ is bi-interpretable with the ring $\mathbb{Z}$ of integers. It follows that the first order theory of $\Gamma$ is undecidable, that all finitely generated subgroups of $\Gamma$ are definable, and that $\Gamma$ is characterized by a single first order sentence among all finitely generated groups.  
 \end{abstract}

\section{Introduction}\label{sec:intro}

In this paper, we continue the study, initiated in \cite{ALM}, of the model theory of higher rank arithmetic groups. One of the  central themes in the study of arithmetic groups is the contrast between arithmetic groups of $S$-rank 1 and arithmetic groups of $S$-rank bigger than 1. Examples of such dichotomy are Margulis's Normal Subgroup Theorem, the Congruence Subgroup Property, and, to lesser extent, superrigidity. The results of this paper, together with results of Sela (\cite{Sel,Sel2}) and Kharlampovich--Myasnikov (\cite{KM,KM2}) for hyperbolic groups (which include, in particular, free groups and uniform rank one orthogonal groups), show that there is also a sharp contrast between the model theories of rank one arithmetic groups and higher rank arithmetic groups.

This paper focuses on higher rank non-uniform arithmetic groups and uniform arithmetic groups of orthogonal type. We show that the model theories of these groups are closely related to the model theory of the ring of integers. More precisely, we show that such groups are bi-interpretable with the ring $\mathbb{Z}$. The notion of bi-interpretability is defined in \S\ref{sec:logic}; in category-theoretic terminology, two structures are bi-interpretable when their categories of imaginaries are equivalent as categories over the category of sets, see Remark \ref{remark:nir} for the precise statement.\\

The setup is the following:

\begin{setting} \label{nota:Gamma} $K$ is a number field, $S$ is a finite set of places containing all archimedean ones, $A$ is the ring of $S$-integers in $K$, $G$ is a connected group scheme over $A$ such that $G_K$ is connected, simply connected, absolutely simple, $\rank_S G \geq 2$ and $G$ satisfies one of the following:  \begin{enumerate}
\item $G_K$ is isotropic.
\item $G=\Spin_q$, where $q:A^n \rightarrow A$ is a regular quadratic form on $A^n$ and there is a place $w\in S$ such that $\rank_w G \geq 2$ (note that this rank is the Witt index, denoted by $i_q(K_w^n)$, of $q$ on $K_w^n$).
\end{enumerate} 
Finally, $\Gamma$ is a centerless congruence subgroup of $G(A)$. 
\end{setting}

For a specific example of a uniform lattice that satisfies the assumptions in Setting \ref{nota:Gamma}, let $D$ to be a square-free positive integer, let $A$ to be the ring of integers of $\Q(\sqrt{D})$, let $q$ be the quadratic from $q(\vec{x}) = x_1^2+\ldots+x_7^n-\sqrt{D}x_8^2-\sqrt{D}x_9^2$, and let $\Gamma$ be the $\mathfrak{p}$th congruence subgroup of $\SO_q(A)$, for some non-dyadic $\mathfrak{p} \triangleleft A$.

The main result of this paper is:

\begin{theorem} \label{thm:main.biint} Under Setting \ref{nota:Gamma}, assume further that $n\ge 9$ if $\Gr=\Spin_q$. The group ${\Gamma}$ is bi-interpretable with the ring $\mathbb{Z}$.
\end{theorem} 

Theorem \ref{thm:main.biint} together with the results of Sela and Kharlampovich--Myasnikov on hyperbolic groups lead to following conjecture: 
\begin{conjecture} Let $\Delta$ be an irreducible lattice in a semisimple group $\prod_{v\in S}G(K_v)$. Then $\Delta$ is bi-interpretable with the ring of integers if and only if $\rank_SG \geq 2$.
\end{conjecture}

Bi-interpretability with the integers has many consequences, we discuss a few of them. 

\begin{corollary} \label{cor:undecidable} Under Setting \ref{nota:Gamma}, assume that $n\ge 9$ if $\Gr=\Spin_q$. The first order theory of $\Gamma$ is undecidable. 	
\end{corollary}

Indeed,  $\Gamma$  interprets the ring of integers and the first order theory of every structure that interprets the ring of integers is undecidable. On the other hand, Kharlampovich--Myasnikov \cite{KM} proved that the first order theory of non-abelian free groups in decidable. In the preprint \cite{KM3}, they extended this result to all torsion free hyperbolic groups. 

\begin{corollary}\label{cor:def_fg_sub} Under Setting \ref{nota:Gamma}, assume that $n\ge 9$ if $\Gr=\Spin_q$. All the finitely generated subgroups of $\Gamma$ are definable.
\end{corollary}

 Corollary \ref{cor:def_fg_sub} is an immediate consequence of Theorem \ref{thm:main.biint} and  a theorem of G\"{o}del which states that any recursively enumerable set in $\mathbb{Z} ^n$ is definable. In contrast, Sela \cite{Sel2} and Kharlampovich--Myasnikov \cite{KM2} proved that the only definable  non-trivial proper subgroups of a torsion free hyperbolic group are its cyclic subgroups.

%\begin{theorem} \label{thm:main.sfor} Under Assumption \ref{nota:Gamma}, assume that $n\ge 9$ if $\Gr=\Spin_q$. Let $\Delta$ be a group that is abstractly commensurable to $\Gamma$. There is a first order sentence $\varphi$ such that, if $\Lambda$ is a finitely generated group, then $\Lambda$ satisfies $\varphi$ if and only if $\Lambda\cong \Delta$.
%\end{theorem} 

\begin{corollary} \label{cor:main.sfor} Under Setting \ref{nota:Gamma}, assume that $n\ge 9$ if $\Gr=\Spin_q$. Then, in the class of finitely generated groups, $\Gamma$ is determined by a single first order sentence, in the following sense: there is a first order sentence $\varphi$ such that, if $\Lambda$ is a finitely generated group, then $\Lambda$ satisfies $\varphi$ if and only if $\Lambda \cong \Gamma$.
\end{corollary} 

Corollary \ref{cor:main.sfor} follows from a result of Khelif \cite[Lemma 1]{Kh}. 
This corollary should be compared with the following result of Sela (\cite{Sel2}): for any torsion-free hyperbolic group $\Gamma$, any first order sentence that holds in $\Gamma$, also holds in $\Gamma * F_n$ (and vice versa).

The property that a group is determined by a single first order sentence (in the class of all finitely generated groups) was first studied by Nies in a more general setting (see \cite{Nie} and \cite{Kh, Las0, Las, Nie1.5, Nie2, OS} for more recent work). Nies called this property {\em quasi-finitely axiomatizable}. 
%Khelif proved in \cite{Kh} that bi-interpretability with the ring of integers implies quasi-finite axiomatizability. Note however, that we cannot apply Khelif's result directly since we allow the group $\Delta$ to have normal abelian subgroups and thus we do not know that it is bi-interpretable with $\mathbb{Z}$. 
There are groups which are quasi-finitely axiomatizable but are not bi-interpretable with the integers. For example,  Khelif and Nies (see \cite{Kh} and \cite{Nie2}) proved that the Heisenberg group $\mathrm{U}_3(\Z)$ is quasi-finitely axiomatizable but not  bi-interpretable with the integers. 

We now discuss the assumption about the triviality of the center. It is possible to extend Corollary \ref{cor:main.sfor} and prove that any central extension of $\Gamma$ by a finite group is quasi-finitely axiomatizable. It turns out that generalizing  Theorem \ref{thm:main.biint} to central extensions is related to word width in $\Gamma$. Recall that every group $\Gamma$ as in Setting  \ref{nota:Gamma} is finitely presented (see \cite[Theroem 5.11]{PR} and the reference therein).

\begin{theorem}\label{thm:width}
	Let $\Gamma$ be a finitely presented group which is bi-interpretable with the ring $\Z$. Let $\Delta$ be a central extension of $\Gamma$ by a group of size $d$. Then $\Delta$ is bi-interpretable with $\Z$ if and only if the word $x^d[y,z]$ has finite width in $\Gamma$. 
\end{theorem}

We find the connection between bi-interpretability and width of words exciting because, even though width of words has been extensively studied (see, for example, \cite{Seg} and \cite{Sha} and the reference therein), to the best of our knowledge, there is not even one example of a non-silly\footnote{$w\in F_n$ is called silly if its image in the abelianization $\mathbb{Z} ^n$ of $F_n$ is primitive. If $w$ is silly, then $w(\Gamma)=\Gamma$, for any group $\Gamma$.} word $w$ and a higher rank uniform lattice $\Gamma$ for which it is known whether the width of $w$ in $\Gamma$ is finite or not. We note that more is known about widths of words in non-uniform lattices, see, for example, \cite{AM} for widths of words in $\SL_n(\mathbb{Z})$. We conjecture that higher rank lattices have finite word widths:

\begin{conjecture} Let $\Delta$ be an irreducible lattice in a semisimple group $\prod_{v\in S} G(O_S)$ and let $w\in F_n$ be a word. Then the width of $w$ in $\Delta$ is finite if and only if $\rank_S G \geq 2$.
\end{conjecture} 

For the only if direction, see \cite{BBF}.

In the proof of Theorem \ref{thm:main.biint}, we prove an effective version of a theorem of Kneser which might be interesting on its own. In order to state it, we need a couple of definitions. 
 
%The proof of Theorem \ref{thm:main.biint} is divided into three steps. The first step is to show that the collection of principal congruence subgroups of $\Gamma$ is uniformly definable. We use ideas from the proof of the congruence subgroup property for such groups. If $G_K$ is isotropic, we the argument is due to Raghunathan, see Theorem \ref{} below. In the case where $\Gr=\Spin$ we prove an effective  version of a theorem of Kneser \cite{}.

\begin{definition} Under Setting \ref{nota:Gamma}, assume that $\Gr=\Spin_q$ and let $S_{def}$ be the set of real places $v \in S$, for which $\Spin_q(K_v)$ is compact. 
\begin{enumerate}
\item For every $v \in S_{def}$, let $||\cdot||_v$ be the norm on $K_v^n$ defined by $||a||_v:=\sqrt{q(a)}$ for every $a \in K_v^n$.
\item For every $v \in S_{def}$, let $\dist_v(\cdot,\cdot)$ be the   bi-invariant metric on $\SO_q(K_v)$ defined  by for all $\alpha,\beta \in \SO_q(K_v)$,
$$\dist_v(\alpha,\beta):=\sup\{||(\alpha-\beta) a||_v \mid a \in  K_v^n \text{ and } ||a||_v=1 \}.$$ 
\item An element $\alpha \in \SO_q(K)$ is called $\epsilon$-separated if for every $v \in S_{def}$, $\dist_v(\alpha,Z(\SO_q(K_v))) \geq \epsilon$.
\item An element $\alpha \in \Gamma$ is called $\epsilon$-separated if its image in  $\SO_q(K)$ is $\epsilon$-separated.
\end{enumerate}
\end{definition}

\begin{definition}
For every element $\alpha $ in a group $\Gamma$ denote $$\ \gcl_\Gamma(\alpha):=\{\beta \alpha \beta^{-1}, \beta \alpha^{-1} \beta^{-1},\id \mid \beta \in \Gamma \}.$$ Note that $\gcl_\Gamma(\alpha)$ is a symmetric and normal set.\end{definition}

The following is an effective version of Theorem 6.1 of \cite{Kne}:

\begin{theorem}\label{thm:effective_kneser} 
Under Setting \ref{nota:Gamma}, assume that $\Gr=\Spin_q$ and $n \ge 7$. For every  $\epsilon>0$ there exists $N=N(n,\epsilon)$ such that for every  $\epsilon$-separated element $\alpha \in \Gamma$ and every non-isotropic $a \in A^n$, $\gcl_\Gamma(\alpha)^Na$ contains an $S$-adelic open neighborhood of $\Gamma a$. \end{theorem}

\begin{remark} We make the following convention: when we write $N=N(X,Y,Z)$, it means that the constant $N$ depends only on $X$, $Y$ and $Z$. For example, the constant in Theorem \ref{thm:effective_kneser} does not depend on the quadratic from $q$ nor on $\Gamma$. 
\end{remark}

%Once we show that congruence subgroups are definable we can use techniques developed by Robinson \cite{},  Romanovski  \cite{} and Noskov \cite{} to show that there exists infinite order $\alpha \in \Gamma$ such that $\langle \alpha\rangle$ is a definable  subset of $\Gamma$ and the map $(\alpha^r,\alpha^s)\mapsto \alpha^{rs}$ is definable. 

%At the final step we have an interpretation $h:H\rightarrow \Gamma$ of $\Gamma$ in itself and we want to show that it is definable. By construction, $h^{-1}$ is definable on $\langle\alpha\rangle$. We then use a theorem  of  Pyber and Szab\'{o} and Breuillad, Green and Tao on expansion in finite simple groups to show that $h^{-1}$ is deniable on a subset which projects onto infinitely many simple quotients of $\Gamma$. Since two elements in $\Gamma$ are equal if and only if they have equal images in infinitely many congruence quotients, it follows that $h^{-1}$ is definable. 

The paper is organized as follows. In \S\ref{sec:logic}, we collect the model-theoretic definitions we use. In \S\ref{sec:local}, we study products of conjugacy classes in locally compact and arithmetic groups. In \S\ref{sec:cong} we show that the collection of congruence subgroups of $\Gamma$ is uniformly definable. In \S\ref{sec:inter}, we show that $\Gamma$ interprets the ring $\mathbb{Z}$. In \S\ref{sec:bi}, we complete the proof of Theorem \ref{thm:main.biint} and show that $\Gamma$ is bi-interpretable with $\mathbb{Z}$. In \S\ref{sec:width}, we prove Theorem \ref{thm:width}, and, in \S\ref{sec:orbit}, we prove Theorem \ref{thm:effective_kneser}.

\begin{remark} In the final stages of writing this paper, we became aware of the preprint \cite{ST} which proves that a higher rank Chevalley group $G(R)$ is bi-interpretable with $R$, under some conditions on $R$ (which include the case $R=O_S$); see also \cite{MM} for the case $\SL_n(\mathcal{O}_S)$. The proofs in these two papers go along the proof of the bi-interpretability of a field $F$ and $\SL_n(F)$, although the analysis in the case of Chevalley groups is harder. In particular, the proof uses in a crucial way information about rational tori and root subgroups, as well as bounded generation. Our goal here is to develop techniques that could be applied to general higher rank lattices. Even in the case of $\SL_n$, the techniques in this paper can be used to prove stronger results: Theorem \ref{thm:app} implies that, for every $n \geq 3$ and every integral domain $R$ with trivial Jacobson radical and finite Krull dimension, $\PSL_n(R)$ and $R$ are bi-interpretable. {In particular, there is no assumption on the stable range of $A$. For example, Theorem \ref{thm:app} implies that, for every $m \ge 1$, $\PSL_3(\Z[X_1,\ldots,X_m])$ is bi-interpretable with $\Z[X_1,\ldots,X_m]$ (and thus also with $\Z$).  }
\end{remark} 

{\noindent \bf Acknowledgement} We thank Alex Lubotzky and Peter Sarnak for helpful discussions. N.A. was partially supported by NSF grant DMS-1902041, BSF grant 2018201, and a Simons Fellowship. C.M. was partially supported by ISF grant 1226/19 and BSF grant 2014099.

\section{Model Theory}\label{sec:logic}  

In this section, we collect the definitions of definable sets, imaginaries, uniformly definable collections, and interpretations. 

\begin{enumerate}

\item For a first-order language $L$ and an $L$-structure $M$, we let $L_M$ be the language $L$ together with a constant symbol for every element of $M$. We denote the $L_M$-theory of $M$ by $\Th_M$.
\item Let $M$ be an $L$-structure. For every $L_M$-formula $F(x_1,\ldots,x_n)$, denote $$F(M):=\left\{ (a_1,\ldots,a_n) \in M^n \mid \text{$F(a_1,\ldots,a_n)$ holds in $M$} \right\}.$$ A subset of  $M^n$ of the form $F(M) $ is called a definable set (in $M$). Note that we say that $F$ is definable in $M$ although it is a subset of some power of $M$. A function  between two definable sets is called definable if its graph is definable. 

\item\label{item:def_col}  Let $Y$ and $Z$ be two definable subsets in $M$. For every definable subset $X \subseteq Y \times Z$ and every $y \in Y$, denote $X_y:=\{z \in Z\mid (y,z)\in X\}$.  A collection $\mathcal{Z}$ of subsets  of a definable set $Z$ is called uniformly definable by a parameter set $Y$ if there exists a definable subset $X \subseteq Y \times Z$ such that $\mathcal{Z}=\{X_y \mid y \in Y\}$. %If the  definable set $X$ itself  and the induced map $y \mapsto X_y$ are not important, we just say that $\mathcal{Z}$ is uniformly definable by a parameter set $Y$ and write $Z[y]$ instead of $X_y$.

\item Given a definable set  $X$ and a definable equivalence relation $E \subseteq X \times X$, the set of $E$-equivalence classes is called an imaginary.  Note that any definable set is also an imaginary. The notions of subset, cartesian product, relation, and function are generalized in the obvious way to  imaginaries. 

\item Suppose that $L_1,L_2$ are two (possibly different) first-order languages, and that, for $i=1,2$, $M_i$ is a structure of $L_i$. An interpretation of $M_2$ in $M_1$ is a pair $\mathscr{F}=(F,f)$, where $F$ is an imaginary in $M_1$ and $f$ is a bijection between the sets $F$ and $M_2$ such that \begin{enumerate}
\item For each $n$-ary relation symbol $r$ of $L_2$, the imaginary $f ^{-1} (r^{M_2}) \subset F^n$ is definable.
\item For every function symbol $g$ of $L_2$, say of arity $(r,s)$, the function $f ^{-1} \circ g^{M_2} \circ f:F^r \rightarrow F^s$ is definable.
\end{enumerate}

\item Suppose that $\mathscr{F}=(F ,f)$ is an interpretation of $M_2$ in $M_1$. By induction on the length of a defining formula, we can define, for each imaginary $X$ in $M_2$, an imaginary $\mathscr{F}^*X$ of $M_1$ and a bijection $f_X:\mathscr{F}^*{X}\rightarrow X$. Similarly, if $X$ and $Y$ are imaginaries in $M_2$ and $g:X\rightarrow Y$ is a definable function,  then there exists a definable  function $\mathscr{F}^*g:\mathscr{F}^*X \rightarrow \mathscr{F}^*Y$ in $M_2$.

\item Suppose that, for $i=1,2,3$, $L_i$ is a first order language and $M_i$ is a structure of $L_i$. Suppose that $\mathscr{F}=(F,f)$ is an interpretation of $M_2$ in $M_1$ and that $\mathscr{H}=(H,h)$ is an interpretation of $M_3$ in $M_2$. The composition $\mathscr{H} \circ \mathscr{F}$ of $\mathscr{F}$ and $\mathscr{H}$ is the interpretation of $M_3$ in $M_1$ given by$(\mathscr{F}^*H,h\circ f_H)$.

\item If $M$ is a structure of a language $L$ and $\mathscr{F}=(\mathcal{F},f)$ is a self interpretation of $M$, we say that $\mathscr{F}$ is trivial if $f$ is definable.

\item Let $L_1,L_2$ be first order languages, and, for $i=1,2$, let $M_i$ be a structure of $L_i$. Given an interpretation $\mathscr{F}_{1,2}$ of $M_2$ in $M_1$ and an interpretation $\mathscr{F}_{2,1}$ of $M_1$ in $M_2$, we say that $(\mathscr{F}_{1,2},\mathscr{F}_{2,1})$ is a bi-interpretation if both compositions $\mathscr{F}_{1,2} \circ \mathscr{F}_{2,3}$ and $\mathscr{F}_{2,3} \circ \mathscr{F}_{1,2}$ are trivial. If there is a bi-interpretation between $M_1$ and $M_2$, we say that they are bi-interpretable.

\end{enumerate}

\begin{remark}\label{remark:nir}
	We can reformulate the notions of interpretations and bi-interpretations in a categorical language. Let $Cat$ be the 2-category of categories, and let $Sets$ be the category of sets. Given a first order language $L$ and an $L$-structure $M$, denote by $Def_M$ the category whose objects are imaginary sets in $M$ and whose morphisms are definable maps between them. We have a functor $M_{pts}:Def_M \rightarrow Sets$ sending a definable set $X$ to $X(M)$. The pair $(Def_M,M_{pts})$ is an object in the slice 2-category $Cat_{/Sets}$. Under these definitions, an interpretation between the structures $M,N$ is a 1-morphism between the objects $(M,M_{pts})$ and $(N,N_{pts})$; a bi-interpretation is an equivalence of these objects.
\end{remark}

\section{Uniform density of conjugacy class products}\label{sec:local}

\subsection{Statement of the results}

In this section, our setting is more general than Setting \ref{nota:Gamma}. Instead, we use the following:

\begin{setting} \label{ass:gcc} Assume that \begin{itemize} 
\item $K$ is a number field.
\item $S$ is a finite set of places of $K$ containing all archimedean places. Denote the ring of $S$-integers of $K$ by $A$.
\item $w$ is a place in $S$.
\item $D$ is a natural number and $\underline{G} \subset (\GL_D)_A$ is an algebraic group scheme defined over $A$ such that $\underline{G}_K$ is connected, simply connected, and simple, and such that $\underline{G}(K_w)$ is non-compact.
\item $\Gamma$ is a subgroup of finite index in $\underline{G}(A)$. If $v$ is a place of $K$, denote the closure of $\Gamma$ in $\underline{G}(K_v)$ by $\Gamma_v$.
\end{itemize} 
\end{setting} 

\begin{definition} \label{defn:good.reduction.big.char} Let $F$ be a local field with ring of integers $R$ and maximal ideal $\mathfrak{m}$. Let $D$ be a natural number and let $\underline{H} \subseteq \GL_D$ be a semisimple algebraic group scheme over $R$. Denote the Lie ring of $\underline{H}$ by $\mathfrak{h}$. We say that $\underline{H}$ is good if the following conditions hold: \begin{enumerate} 
\item \label{cond:good.reduction.smooth} $\underline{H}$ is smooth over $\Spec R$.
\item \label{cond:good.reduction.Z.smooth} The reduction map $Z(\underline{H}(R)) \rightarrow Z(\underline{H}(R/\mathfrak{m}))$ is onto.
\item \label{cond:good.reduction.irr} $\underline{H}(R/\mathfrak{m})$ acts irreducibly on $\mathfrak{h}(R/\mathfrak{m})$.
\item \label{cond:good.reduction.Ad} The kernel of $\Ad:\underline{H}(R/\mathfrak{m}) \rightarrow \Aut(\mathfrak{h}(R/\mathfrak{m}))$ is $Z \left( \underline{H}(R/\mathfrak{m}) \right)$.
\item \label{cond:good.reduction.char} The characteristic of $R/\mathfrak{m}$ is greater than $|Z(G)| + (4D\dim \underline{H})^4$.
\end{enumerate} 

Note that, by the classification of simple algebraic groups over finite fields, we get that $\underline{H}(\mathbb{F}_q)$ acts irreducibly on its Lie algebra.
\end{definition} 

\begin{definition}\label{def:T_Gamma} In Setting \ref{ass:gcc}, define $T_\Gamma$ to be the set of all places $v\notin S$ such that $\underline{G}_{A_v}$ is good and $\Gamma$ is dense in $\underline{G}(A_v)$.
\end{definition} 

\begin{remark} In Setting \ref{ass:gcc}, $T_\Gamma$ is always finite. Indeed, Condition \ref{cond:good.reduction.smooth} follows from generic smoothness, Condition \ref{cond:good.reduction.Z.smooth} follows from generic smoothness of the group scheme $Z(\underline{G})$ and Hensel's lemma, Condition \ref{cond:good.reduction.irr} follows because $\underline{G}(K)$ acts irreducibly on $\mathfrak{g}(K)$ and this is a Zariski open condition, and Condition \ref{cond:good.reduction.Ad} follows because it holds over $\overline{K}$. 
\end{remark}

\begin{definition} \label{def:std.metric} In Setting \ref{ass:gcc}, for every real place $v$ of $K$ such that $\underline{G}(K_v)$ is compact, we define the standard metric on $\underline{G}(K_v)$ as follows: \begin{enumerate} 
\item If $\underline{G}_{K_v}=\Spin_f$, for some (positive-definite) quadratic form $f$ on $K_v ^n$, let $d_v$ be the metric induced by the norm $f$: 
\[
d_v(g_1,g_2)=\max \left\{ \sqrt{f(g_1x-g_2 x)} \mid x\in K_v^n, f(x)=1 \right\}.
\]
\item In all other cases, let $d_v$ be the translation-invariant Riemannian metric on $\underline{G}(K_v)$ whose restriction to the Lie algebra of $\underline{G}(K_v)$ is the Killing form, and normalized such that the diameter of $d_v$ is one.
\end{enumerate} 
\end{definition} 

\begin{remark} The reason for the different definition for spin groups is that it simplifies the notations in the proof of Theorem \ref{thm:effective_kneser} in \S\ref{sec:orbit}. Since the norm metric and the Killing metric are bi-Lipschitz, one can use the Killing metric in both cases after changing some constants in \S\ref{sec:orbit}.
\end{remark} 

\begin{definition} In Setting \ref{ass:gcc}, given $g\in \Gamma$ and $\epsilon >0$, we say that $g$ is $\epsilon$-separated if, for any real valuation $v$ of $K$ such that $\underline{G}(K_v)$ is compact, we have $d_v(g,Z(\underline{G}(K_v)))>\epsilon$, where $d_v$ is the standard metric on $\underline{G}(K_v)$ from Definition \ref{def:std.metric}.
\end{definition}

The main results of this section are the following three claims:

\begin{proposition} \label{cor:bounded_prod} In Setting \ref{ass:gcc}, for every $\epsilon>0$, there is a natural number $N=N(K,S,D,\epsilon)$ such that the following holds:

If $\underline{G}(K_w)$ is non-compact, and $g\in \Gamma$ is $\epsilon$-separated, then there is a neighborhood $W$ of the identity in $\prod_{v\in T_\Gamma} \underline{G}(K_v)$ such that the set $\gcl_\Gamma(g)^N$ contains a dense subset of $W \times \prod_{v\notin T_\Gamma \cup \left\{ w \right\}} \langle \gcl_{\Gamma_v} (g) \rangle$.
\end{proposition} 

\begin{remark} Under the more restrictive Setting \ref{nota:Gamma}, for every non-archimedean $v\in S$, the group $\underline{G}(K_v)$ is non-compact. In that case, we can prove Proposition \ref{cor:bounded_prod} without using Lemma \ref{lem:conj.compact.p-adic.open} below and conclude that (under Setting \ref{nota:Gamma}) the constant $N$ depends only on $D$ and $\epsilon$.
\end{remark} 

\begin{proposition}\label{cor:bound_width_cong} In Setting \ref{ass:gcc}, for every $\epsilon>0$  there exists a constant $N=N(\Gamma,\epsilon)$ such that, for every $\epsilon$-separated element $g \in \Gamma$, $\gcl_\Gamma(g)^N$ is dense in $\langle\gcl_\Gamma(\alpha)\rangle$ with respect to the topology induced by $G \left( \mathbb{A}_K ^{\left\{ w \right\}} \right)$, where $\mathbb{A}_K ^{\left\{ w \right\}}=\prod_{v\neq w}' K_v$ is the ring of $w$-adeles. In particular, it is dense in the congruence topology.
\end{proposition}

Note that the constant $N$ in Proposition \ref{cor:bounded_prod} depends only on $K,S,D,\epsilon$ and not on $\underline{G}$ or $\Gamma$. In contrast, the constant $N$ in Proposition \ref{cor:bound_width_cong} does depend on $\Gamma$.

\begin{proposition}\label{cor:bound_prod} In Setting \ref{ass:gcc}, there exists a constant $N=N(\Gamma)$ such that, for every principal congruence subgroup $\Delta$ contained in $\Gamma$, there are $\alpha_1,\ldots,\alpha_N \in \Delta$ such that $\prod_{1 \leq i \leq N}\gcl_{\Gamma}(\alpha_i)$ is a dense subset of $\Delta$, with respect to the congruence topology. 
\end{proposition}

\subsection{Finite}

%\nir{do we need it? move there}
%\begin{lemma} \label{lem:effective.generation} Let $N$ be a positive integer. If $G$ is a group, $H \le G$ is a subgroup of index at most $N$, and $1 \in X \subset G$ is a symmetric generating set, then the set $X^{2N+1} \cap H$ contains a generating set for $H$.
%\end{lemma} 
%\begin{proof}
%Let $T$ be a set of representative of $H \backslash G$ \chen{which contains the identity}. For $g \in G$ let $\bar{g}$ denote the unique element of $T$ for which $g \in G \bar{g}$. Then $H$ is generated by $\{tx\overline{tx}^{-1} \mid x \in X, t \in T\}$. The result follows since $X^{N}$ contains a representative set of   $H \backslash G$.
 
%\end{proof}

%The following Lemma follows from \cite{}:

%\begin{lemma} For every $r$ there is a constant $N=N(r)$ such that if $\underline{H}$ is a connected and simple algebraic group of rank at most $r$ over a finite field $\mathbb{F}_q$ and $g\in \underline{H}(\mathbb{F}_q)$, then $\gcl_{\underline{H}(\mathbb{F}_q)}(g)^N=\langle \gcl_{\underline{H}(\mathbb{F}_q)}(g) \rangle$.
%\end{lemma} 

\begin{lemma} \label{lem:irreduciblilty} Let $\underline{H} \subseteq \GL_D$ be a connected and simple algebraic group defined over a finite field $\mathbb{F}_q$ of characteristic $p>(4D\dim \underline{H})^4$. Denote the Lie algebra of $\underline{H}$ by ${\mathfrak{h}}$. Denoting the simply connected cover of $\underline{H}$ by $\underline{H}^{sc}$, assume that $\underline{H}^{sc}(\mathbb{F}_q)$ acts irreducibly on $\mathfrak{h}(\mathbb{F}_q)$. Then, for every non-zero $X\in \mathfrak{h}(\mathbb{F}_q)$, we have
\[
\underbrace{\Ad(\underline{H}(\mathbb{F}_q))X+\cdots+\Ad(\underline{H}(\mathbb{F}_q))X}_{\text{$4\dim \underline{H}$ times}}=\mathfrak{h}(\mathbb{F}_q).
\]
\end{lemma}

\begin{proof} Denote $d=\dim \underline{H}$. It is well-known that $\underline{H}^{sc}(\mathbb{F}_q)$ is generated by its $p$-elements. Since $\underline{H}^{sc}(\mathbb{F}_q)$ acts irreducibly on $\mathfrak{h}(\mathbb{F}_q)$ and the action factors through $\underline{H}(\mathbb{F}_q)$, it follows that there is no subspace of $\mathfrak{h}(\mathbb{F}_q)$ which is invariant under all $p$-elements of $\underline{H}(\mathbb{F}_q)$.

Since $D<p$, the logarithm map is defined on the set of $p$ elements of $\underline{H}(\mathbb{F}_q)$, and $\log(u)\in \mathfrak{h}(\mathbb{F}_q)$, for every such $u$. Since $\Ad(u)=\exp(ad(\log(u)))$, there is no non-trivial subspace invariant under all elements of the set
\[
\left\{ ad(\log(u))\mid \text{$u\in \underline{H}(\mathbb{F}_q)$ is a $p$-element} \right\}.
\]
It follows that there are $u_1,\ldots,u_d\in \underline{H}(\mathbb{F}_q)$ such that $\left\{ ad(\log(u_i))X \right\}$ is a basis for $\mathfrak{h}(\mathbb{F}_q)$.

Denote $u^t=\exp(t\log(u))$ and define a map $F:\mathbb{A} ^d \rightarrow \mathfrak{h}$ by
\[
F(t_1,\ldots,t_d)=\left( \Ad(u_1^{t_1}) \circ \Ad(u_2^{t_2}) \circ \cdots \circ \Ad(u_d^{t_d})\right) (X).
\]
$F$ is a polynomial map of degree $2dD$ and its derivative at $(0,\ldots,0)$ is the map
\[
dF(0,\ldots,0)(t_1,\ldots,t_d)=\sum t_i ad(\log(u_i))X.
\]
Since $\left\{ ad(\log(u_i))X \right\}$ is a basis, $dF(0,\ldots,0)$ is onto. In particular, $F$ is a dominant map. Let $\mu$ be the measure on $\mathfrak{h}(\mathbb{F}_q)$ given by $\mu=\sum_{a\in \mathbb{F}_q^d} \delta_{F(a)}$, where $\delta_a$ is the delta measure at $a$. For every $t\in \mathbb{F}_q$, $u_i^t \in \underline{H}(\mathbb{F}_q)$, so $\supp(\mu) \subset \Ad(\underline{H}(\mathbb{F}_q))X$.

Fix a non-trivial additive character $\psi$ of $\mathbb{F}_q$. Let $\chi:\mathfrak{h}(\mathbb{F}_q) \rightarrow \mathbb{C} ^ \times $ be an additive character. Then $\chi = \psi \circ \varphi$, where $\varphi: \mathfrak{h} (\mathbb{F}_q) \rightarrow \mathbb{F}_q$ is a $\mathbb{F}_q$-linear map. Denoting the Fourier transform of $\mu$ by $\widehat{\mu}$, we have
\[
\widehat{\mu}(\chi)=\sum_{a\in \mathbb{F}_q^d} \psi(-(\varphi \circ F)(a)).
\]
The polynomial $\varphi \circ F$ has degree $2dD < p$. The Weil bounds (see \cite[Proposition 3.8]{SGA}) give
\[
\left| \widehat{\mu}(\chi)\right| < (2dD)q^{d-\frac{1}{2}},
\]
for all non-trivial characters $\chi$. We have
\[
\left| \sum_{\text{$\chi$ non-trivial}} \widehat{\mu^{*4d}}(\chi) \right| \leq  \sum_{\text{$\chi$ non-trivial}} \left| \widehat{\mu(\chi)} \right|^{4d} < q^d(2dD)^{4d}q^{4d^2-2d} < q^{4d^2} = \widehat{\mu^{*4d}}(1).
\]
By Placherel inversion theorem,
\[
\left| \mu ^{*4d}(g) \right|=\left| \sum_\chi \widehat{\mu^{*4d}}(\chi) \chi(g) \right| \geq \widehat{\mu^{*4d}}(1)- \left| \sum_{\text{$\chi$ non-trivial}} \widehat{\mu^{*4d}}(\chi) \right| >0,
\]
so
\[
\mathfrak{h}(\mathbb{F}_q)=\supp(\mu ^{*4d}) \subseteq \underbrace{\Ad({\underline{H}})X+\cdots+\Ad({\underline{H}})X}_{\text{$4d$ times}}.
\]
\end{proof}

\subsection{Local}

In this section, we prove several local versions of Propositions \ref{cor:bounded_prod}, \ref{cor:bound_width_cong}, and \ref{cor:bound_prod}. The versions we prove are for compact Lie groups (Lemma \ref{lem:conj.open.compact.Lie}), Non-compact groups (Lemma \ref{lem:conj.open.noncompact.Lie}), compact p-adic groups (Lemma \ref{lem:conj.compact.p-adic.open}), and another version for compact p-adic groups (Lemma \ref{lem:dense.p-adic}) which works only for good compact p-adic groups, but gives a uniform bound on the exponent $N$.

We will use the following quantitative version of the open mapping theorem:

\begin{lemma} \label{lem:open.mapping.quantitative} Let $R$ be the ring of integers of a non-archimedean local field, let $\mathfrak{m} \subseteq R$ be the maximal ideal, and let $X,Y \subset R^d$ be $p$-adic manifolds. There is a constant $C$ such that, for every function $f:X \rightarrow Y$ which is given by a convergent power series with coefficients in $R$, every $x_0\in X$, and every natural number $n$ such that $df_{x_0}(T_{x_0}X\cap R^d) \supseteq \mathfrak{m} ^n (T_{f(x_0)}Y \cap R^d)$, we have $f(X) \supseteq Y \cap \left( f(x_0)+\mathfrak{m} ^{n+C}R^d \right)$. Moreover, for every $k \geq 0$, we have $f\left( X \cap \left( x_0+\mathfrak{m} ^kR^d \right) \right) \supseteq Y \cap \left( f(x_0)+\mathfrak{m} ^{n+k+C}R^d \right)$.
\end{lemma}

The following is an immediate consequence of Lemma \ref{lem:open.mapping.quantitative}:

\begin{lemma} \label{lem:implicit.function.theorem} Suppose that either $X,Y,Z$ are real manifolds or that they are $p$-adic manifolds. Let $f:X \times Y \rightarrow Z$ be a continuously differentiable function. For each $y\in Y$, let $f_y:X \rightarrow Z$ be the function $f_{y}(x)=f(x,y)$. Assume that there is a point $(x_0,y_0)$ for which $d f_{y_0}(x_0):T_{x_0}X \rightarrow T_{f(x_0,y_0)}Z$ is onto. Then there are open sets $U \subseteq Y$ and $V \subseteq Z$ such that $y_0\in U$, $f(x_0,y_0)\in V$, and $f_{y}(X) \supseteq V$, for every $y\in U$.
\end{lemma}

A compactness argument together with Lemma \ref{lem:implicit.function.theorem} yields:

\begin{corollary} \label{cor:open.mapping} Suppose that either $X,Y,Z$ are real manifolds or that they are $p$-adic manifolds, let $z_0\in Z$, and let $C \subseteq Y$ be a compact set. Let $f:X \times Y \rightarrow Z$ be a continuously differentiable function. Assume that, for each $y\in C$ there is $x\in X$ such that $f(x,y)=z_0$ and $df_{y}(x)$ is onto. Then there is an open set $z_0\in V\subseteq Z$ such that, for each $y\in C$, $V \subseteq f_y(X)$.
\end{corollary}

\begin{lemma} \label{lem:conj.open.general} Let $F$ be a local field, let $\underline{H}$ be a connected and almost simple algebraic group defined over $F$ and let $H=\underline{H}(F)$. Let $U \subset H$ be a neighborhood of 1 and let $C \subseteq H$ be a compact set disjoint from $Z(H)$. Then there is an identity neighborhood $V \subseteq H$ such that, for every $g\in C$, $V \subseteq \ccl_U(g) ^{\dim_F \underline{H}}$. 
\end{lemma}

\begin{proof} Let $d=\dim_F \underline{H}$
%, and choose an identity neighborhood $U_0$ such that $U_0 \cdot U_0 \subseteq U$. 
and define $\Phi: U^{2d} \times H \rightarrow H$ as the map
\[
\Phi(h_1,\ldots,h_d,x_1,\ldots,x_d,g)= \prod_{i=1}^d \left( h_i ^{-1} g ^{-1} x_i g x_i ^{-1} h_i \right).
\]
We claim that the conditions of Corollary \ref{cor:open.mapping} hold for $X=U^{2d}$, $Y=Z=H$, $z_0=1$, $C=C$, and $f=\Phi$. It then follows that there is an identity neighborhood $V \subseteq H$ such that, for all $g\in C$, $V \subseteq \Phi_g(U^{2d}) \subseteq \ccl_U(g)^d$.

Denote the Lie algebra of $H$ by $\mathfrak{g}$ and let $g\in C$. Since $g$ is not central, the subspace $W:=(\Ad(g)-\Id)(\mathfrak{g})$ is non-trivial. By assumption, the adjoint action of any open subgroup of $H$ on $\mathfrak{g}$ is irreducible, so there are $h_1,\ldots,h_d \in U$ such that $\Ad(h_1)W+\cdots+\Ad(h_d)W=\mathfrak{g}$. The linear map $d\Phi_g(h_1,\ldots,h_d,1,\ldots,1):\mathfrak{g}^{2d} \rightarrow \mathfrak{g}$ takes the vector $(0,\ldots,0,X_1,\ldots,X_d)$ to $\sum \Ad(h_i)(\Ad(g)X_i-X_i)$. Therefore, the conditions of Corollary \ref{cor:open.mapping} hold.
\end{proof}

\begin{lemma} \label{lem:conj.open.compact.Lie} Let $H$ be a compact connected almost simple real Lie group, and let $C \subseteq H$ be a compact set disjoint from $Z(H)$. There is an $N=N(H,C)$ such that, for every $g\in C$, $\ccl_H(g)^N=H$.
\end{lemma}

\begin{proof} Let $\rho$ be a bi-invariant Riemannian metric on $H$ with diameter 1. By Lemma \ref{lem:conj.open.general} (applied with $U=H$), there is an $\epsilon>0$ such that $\ccl(g)^{\dim H}$ contains the ball of radius $\epsilon$ around the identity, for all $g\in C$. Since the metric is bi-invariant and geodesic, the product of a ball of radius $a_1$ and a ball of radius $a_2$ is a ball of radius $\min \left\{ a_1+a_2,1 \right\}$. It follows that $\ccl_H(g)^{\lceil 1/ \epsilon \rceil \dim H}=H$.
\end{proof}

\begin{lemma} \label{lem:conj.open.noncompact.Lie}  Let $F$ be a local field, let $\underline{H}$ be connected semisimple algebraic group over $F$ such that $\underline{H}(F)$ is non compact, and let $H \subseteq \underline{H}(F)$ be an open subgroup. Then there is $N=N(F,\underline{H},H)$ such that, for every $g\in H$, we have $\ccl_H(g)^N=\langle \ccl_H(g) \rangle$
\end{lemma} 

\begin{proof} 

The claim is clear for $g\in Z(H)$ (taking $N=|Z(H)|$), so we may assume $g\notin Z(H)$. Fix a maximal split torus $A \subseteq \underline{H}(F)$, a maximal compact subgroup $K \subseteq \underline{H}(F)$, and a non-trivial unipotent $u_0$.

First assume that $F$ is non-archimedean. By Lemma \ref{lem:conj.open.general}, applied with $C=\left\{ u_0 \right\}$ and $U=H$, there is an open neighborhood of 1 that is contained in $\ccl_H(u_0)^{\dim_F \underline{H}}$. In particular, there is a natural number $M$ such that the set $\ccl_H(u_0)^{\dim_F \underline{H}}\cap K$ contains a subgroup of index at most $M$ in $K$. We will show that the claim of the lemma holds with $N=2(\dim_F\underline{H})^2+5 (\dim_F \underline{H}) (\dim_F A)+M^2$.

By Lemma \ref{lem:conj.open.general}, applied with $C=\left\{ g \right\}$ and $U=H$, $\ccl_H(g)^{\dim_F \underline{H}}$ contains an open neighborhood of 1. Since it is conjugation invariant, $\ccl_H(g)^{\dim_F \underline{H}}$ contains all unipotents. In particular, it contains $u_0$. It follows that $\ccl_H(g)^{(\dim_F \underline{H})^2}$ contains a subgroup of index at most $M$ in $K$.

For any root $\alpha$ of $A$, there is a homomorphism $\phi_\alpha :\SL_2(F) \rightarrow \underline{H}(F)$ whose image contains the root group. We claim that $\phi_\alpha ^{-1} (H)=\SL_2(F)$. Indeed, let $H_+=\left\{ x \in F \mid \phi_\alpha \left(\begin{matrix} 1 & x \\ & 1\end{matrix}\right) \in H \right\}$, $H_0=\left\{ x \in F^ \times  \mid \phi_\alpha \left(\begin{matrix} x &  \\  & x ^{-1}\end{matrix}\right) \in H \right\}$, and $H_-=\left\{ x \in F \mid \phi_\alpha \left(\begin{matrix} 1 &  \\ x &  1\end{matrix}\right) \in H \right\}$. Then $H_+,H_-$ are finite index subgroups of $F$ that are invariant under $H_0$, which is a finite index subgroup of $F^ \times$. Hence, $H_+=H_-=F$ and $\phi_\alpha ^{-1} (H)=\SL_2(F)$.

For every unipotent element $u\in \SL_2(F)$, we have $\ccl_{\SL_2(F)}(u)^{5}=\SL_2(F)$ (see, for example, \cite[Theorem 2.5]{VW}). Hence, $\ccl_H(g)^{5 \dim_F \underline{H}}$ contains the entire root subgroup of $\alpha$. Since the root subgroups of $A$ generate $A$, we get that $A \subseteq \ccl_H(g)^{5 (\dim_F \underline{H}) (\dim_F A)}$. By Cartan decomposition, we get that $\ccl_H(g)^{2(\dim_F\underline{H})^2+5 (\dim_F \underline{H}) (\dim_F A)}$ contains a subset of index at most $M^2$ in $\underline{H}(F)$, and the claim follows.

In the case $F$ is archimedean, the proof is similar, replacing the condition that $\ccl_H(u_0)^{\dim_F \underline{H}}\cap K$ contains a subgroup of index at most $M$ in $K$ by the condition that $\left( \ccl_H(u_0)^{\dim_F \underline{H}}\cap K \right) ^M=K$.
\end{proof} 

For the rest of this subsection, we will use the following setting:

\begin{setting} \label{set:n.a.local} $F$ is a non-archimedean local field with ring of integers $R$ and maximal ideal $\mathfrak{m}$. $D$ is a natural number, $\underline{H} \subseteq \GL_D$ is a simple algebraic group over $R$, and $H \subset \underline{H}(R)$ is a compact open subgroup. We denote the Lie ring of $\underline{H}$ by $\mathfrak{h}$, the $\mathfrak{m} ^k$-th congruence subgroup of $\GL_D(R)$ by $\GL_D(R;\mathfrak{m} ^k)$, and denote $H[\mathfrak{m} ^k]:=H \cap \GL_D(R;\mathfrak{m} ^k)$.
\end{setting} 

\begin{lemma} \label{lem:conj.compact.p-adic.open} In Setting \ref{set:n.a.local}, there are constants $c=c(F,D,\underline{H},H),N=N(F,D,\underline{H},H)$ such that \begin{enumerate} 
\item \label{item:conj.compact.1} For every $n$, if $g\in H \smallsetminus \left( Z(H) \cdot H[\mathfrak{m} ^n] \right)$, then $\gcl_H(g)^{|Z(H)| \cdot \dim \underline{H}} \supseteq H[\mathfrak{m} ^{n+c}]$.
\item \label{item:conj.compact.2} For any $g\in H$, $\gcl_H(g)^N = \langle \gcl_H(g) \rangle$.
\item \label{item:conj.compact.3} For any normal subgroup $L$ of $H$, there are $h_1,\ldots,h_N\in L$ such that $L=\gcl_H(h_1) \cdots \gcl_H(h_N)$.
\end{enumerate} 
\end{lemma} 

\begin{proof} Denote $d=\dim_F \underline{H}$, $q=|R/\mathfrak{m} |$, and $b=\val_\mathfrak{m} |Z(H)|$. 

\begin{enumerate} 
\item[\ref{item:conj.compact.1}.] Let $c_1$ be the constant from Lemma \ref{lem:open.mapping.quantitative} applied to $X=R^d$ and $Y=H$. By enlarging $c_1$, we can assume that \begin{enumerate} 
\item The series $\exp$ converges on $\mathfrak{m} ^{c_1} \mathfrak{g}$ and $\exp(\mathfrak{m} ^{c_1} \mathfrak{g})=H[\mathfrak{m} ^{c_1}]$ (and, hence, $\exp(\mathfrak{m} ^k \mathfrak{g})=H[\mathfrak{m} ^k]$, for every $k \geq c_1$.
\item $[\mathfrak{g},\mathfrak{g}] \supseteq \mathfrak{m} ^{c_1} \mathfrak{g}$.
\end{enumerate} 
By Lemma \ref{lem:conj.open.general}, there is $c_2$ such that, for every $g\in H \smallsetminus Z(H) \cdot H[\mathfrak{m} ^{c_1}]$, $\gcl_H(g)^d \supseteq H[\mathfrak{m} ^{c_2}]$. We will prove that the claim holds with $c=\max \left\{ 2c_1+b,c_2 \right\}$.

Suppose $g\in H \smallsetminus Z(H) \cdot H[\mathfrak{m} ^n]$, and let $a$ be the minimal number such that $g\in H \smallsetminus Z(H) \cdot H[\mathfrak{m} ^a]$. Then $1 \leq a \leq n$. There are two cases:

{\bf Case 1: $a \leq c_1$.} In this case, $g\in H\smallsetminus Z(H) \cdot H[\mathfrak{m} ^{c_1}]$ and, by the definition of $c_2$, we have $\gcl_H(g)^{d} \supseteq H[\mathfrak{m} ^{c_2}] \supseteq H[\mathfrak{m} ^{a+c}]$.

{\bf Case 2: $a \geq c_1+1$.} In this case, $g\in Z(H) \cdot H[\mathfrak{m} ^{a-1}] \subseteq Z(H) \cdot H[\mathfrak{m} ^{c_1}]$, so $g=\zeta \exp(X)$, where $\zeta \in Z(H)$ and $X\in \mathfrak{m} ^{a-1}\mathfrak{g} \smallsetminus \mathfrak{m} ^a \mathfrak{g}$. Denoting $Y=|Z(H)| X$, we get that $g^{|Z(H)|}=\exp(Y)$ and $Y\in \mathfrak{m} ^{a-1+b}\mathfrak{g} \smallsetminus \mathfrak{m} ^{a+b}\mathfrak{m}$. By the definition of $c_1$, there are $X_1,\ldots,X_d\in \mathfrak{m} ^{c_1} \mathfrak{g}$ such that the elements $[X_i,Y]$ are in $\mathfrak{m} ^{a-1+b+2c_1} \mathfrak{g}$ and their reduction modulo $\mathfrak{m} ^{a+b+2c_1}$ is a basis. Let $\Phi:R^d \rightarrow H$ be the function 
\[
\Phi(t_1,\ldots,t_d)=\left[\exp(-t_1X_1), g^{|Z(H)|}\right] \cdots \left[\exp(-t_dX_d), g^{|Z(H)|}\right]
\]
Then $\Phi(0,\ldots,0)=1$ and $d \Phi_{(0,\ldots,0)}(R^d) \supseteq \mathfrak{m} ^{a+b+2c_1-1}\mathfrak{g}$. By Lemma \ref{lem:open.mapping.quantitative}, $\gcl_H(g)^{|Z(H)|d} \supseteq \Phi(R^d) \supseteq H[\mathfrak{m} ^{a+b+2c_1+c}]$

\item[\ref{item:conj.compact.2}.] Let $c$ be the constant from Claim \ref{item:conj.compact.1}. We will show that Claim \ref{item:conj.compact.2} holds with $N=|Z(H)|\dim_F \underline{H}+q^{cD^2}$. If $g\in Z(H)$, then $\gcl_H(g)^{|Z(H)|}=\langle \gcl_H(g) \rangle$. Assume now that $g\notin Z(H)$ and let $n$ be the minimal natural number such that $g\in H \smallsetminus Z(H) \cdot H[\mathfrak{m} ^n]$. We have
\[
H[\mathfrak{m} ^{c+n}] \subseteq \gcl_H(g)^{|Z(H)|\dim_F \underline{H}} \subseteq \langle \gcl_H(g)\rangle \subseteq Z(H) \cdot H[\mathfrak{m} ^{n-1}].
\]
Since $|H[\mathfrak{m} ^{n-1}] /  H[\mathfrak{m} ^{c+n}]| < q^{(c+1)D^2}$, we get the result.

\item[\ref{item:conj.compact.3}.] Let $c$ be the constant from Claim \ref{item:conj.compact.1}. We show that Claim \ref{item:conj.compact.3} holds with $N=|Z(H)|q^{(c+1)D^2}$. If $L \subset Z(H)$ then the claim holds. Otherwise, let $n$ be the minimal natural number such that $L \smallsetminus Z(H) \cdot H[\mathfrak{m} ^n] \neq \emptyset$, and choose $h_1\in L \smallsetminus Z(H) \cdot H[\mathfrak{m} ^n]$. By definition of $c$, 
\[
H[\mathfrak{m} ^{n+c}] \subseteq \gcl_H(h_1)^{|Z(H)| \dim_F \underline{G}} \subseteq L \subseteq Z(H) \cdot H[\mathfrak{m} ^{n-1}].
\]
Since $|Z(H) \cdot H[\mathfrak{m} ^{n-1}] / H[\mathfrak{m} ^{n+c}]|<|Z(H)|q^{(c+1)D^2}$, the result follows.
\end{enumerate} 
\end{proof}

\begin{lemma} \label{lem:Ad-1.g} In Setting \ref{set:n.a.local}, assume that $\underline{H}$ is good. Let $k \geq 1$. Suppose $g\in \underline{H}(R)$ and $(\Ad(g)-\Id)\mathfrak{h} (R) \subseteq \mathfrak{m}^k \mathfrak{h}(R)$. Then $g\in Z(\underline{H}(R)) \cdot \underline{H}(R;\mathfrak{m} ^k)$.
\end{lemma}

\begin{proof} For $k=1$, this follows from the assumption that the action of $\underline{H}(R/\mathfrak{m})$ on $\mathfrak{h}(R/\mathfrak{m})$ is faithful.

Assume now that $k > 1$. By the case $k=1$, we know that $g\in \underline{H}(R;\mathfrak{m})$. By assumption, $g=\exp(Y)$, for some $Y\in \mathfrak{m}\mathfrak{h}(R)$. Since $\Ad(g)=\exp(ad(Y))$, we get that $[Y,\mathfrak{h}(R)] \subseteq \mathfrak{m}^k \mathfrak{h}(R)$. Since $\mathfrak{h}(R/\mathfrak{m})$ has no center, we get by induction on $k$ that $Y\in \mathfrak{m}^k \mathfrak{h}(R)$, so $g=\exp(Y)\in \underline{H}(R;\mathfrak{m} ^k)$.
\end{proof}

\begin{lemma} \label{lem:dense.p-adic} In Setting \ref{set:n.a.local}, assume that $\underline{H}$ is good. For every $g\in \underline{H}(R)$, $\gcl_{\underline{H}(R)}(g)^{5|Z(\underline{H}(R))|\dim \underline{H}}=\langle \gcl_{\underline{H}(R)}(g) \rangle$. If $g\in \underline{H}(R) \smallsetminus Z(\underline{H}(R)) \cdot \underline{H}(R)[\mathfrak{m}]$, then $\langle \ccl_{\underline{H}(R)}(g) \rangle =\underline{H}(R)$.
\end{lemma}

\begin{proof} Denote $d=\dim_F \underline{H}$. If $g\in Z(\underline{H}(R))$, the claim is clear. Assume now that $g\in \underline{H}(R) \smallsetminus Z(\underline{H}(R)) \cdot \underline{H}(R)[\mathfrak{m}]$. By the smoothness assumption, $\mathfrak{h}(R)$ is a free $R$-module of rank $d$. By Lemma \ref{lem:Ad-1.g}, the submodule $V:=(\Ad(g)-\Id)(\mathfrak{h}(R)) \subseteq \mathfrak{h}(R)$ is not contained in $\mathfrak{m} \mathfrak{h}(R)$. By the irreducibility of the action of $\underline{H}(R/\mathfrak{m})$ and by Nakayama's lemma, there are $h_1,\ldots,h_d \in \underline{H}(R)$ such that $\Ad(h_1)V+\ldots+\Ad(h_d)V=\mathfrak{h}(R)$. Define $\Psi:\underline{H}(R)^{d} \rightarrow \underline{H}(R)$ by
\[
\Psi(x_1,\ldots,x_d)=\prod_{i=1}^d \left( h_i ^{-1} g ^{-1} x_i g x_i ^{-1} h_i \right).
\]
We get that $\mathfrak{h}(R) = d\Psi|_{(1,\ldots,1)} \left( \mathfrak{h}(R)^d\right)$, and, by Lemma \ref{lem:open.mapping.quantitative}, we get that $\underline{H}(R;\mathfrak{m}) \subseteq \Psi \left( \underline{H}(R) ^d \right) \subseteq \ccl_{\underline{H}(R)}(g)^{d}$. Since $H \subseteq \ccl_{\underline{H}(R)}(g)^d \cdot \underline{H}(R;\mathfrak{m})$, the result follows.

Finally, assume $g\in Z(\underline{H}(R)) \cdot \underline{H}(R)[\mathfrak{m}] \smallsetminus Z(\underline{H}(R))$. Let $k \geq 1$ be the number such that $g\in Z(\underline{H}(R)) \cdot \underline{H}(R;\mathfrak{m} ^k)\smallsetminus Z(\underline{H}(R)) \cdot \underline{H}(R;\mathfrak{m} ^{k+1})$. Since $\underline{H}$ is good, $g^{|Z(\underline{H}(R)|} \in \underline{H}(R;\mathfrak{m} ^k) \smallsetminus \underline{H}(R;\mathfrak{m} ^{k+1})$. The same arguments as above imply that $\underline{H}(R;\mathfrak{m} ^{k+1}) \subseteq \ccl_{\underline{H}(R)}(g^{|Z(\underline{H}(R)|})^d$.

Denote the image of $g^{|Z(\underline{H}(R)|}$ in $\underline{H}(R;\mathfrak{m} ^k)/\underline{H}(R;\mathfrak{m} ^{k+1})=\mathfrak{h}(R/\mathfrak{m})$ by $X$. By Lemma \ref{lem:irreduciblilty} 
\[
\underbrace{\Ad(\underline{H}(R))X+\cdots+\Ad(\underline{H}(R))X}_{\text{$4d$ times}}=\mathfrak{h}(R/\mathfrak{m}).
\]
Taking exponents, we get that
\[
\left( \gcl_{\underline{H}(R)}(g^{|Z(\underline{H}(R)|}) \right) ^{4d} \cdot \underline{H}(R;\mathfrak{m} ^{k+1})=\underline{H}(R;\mathfrak{m} ^{k}),
\]
so $\gcl_{\underline{H}(R)}(g^{|Z(\underline{H}(R)|})^{5d} \supseteq \gcl_{\underline{H}(R)}(g^{|Z(\underline{H}(R)|})^{4d} \gcl_{\underline{H}(R)}(g^{|Z(\underline{H}(R)|})^{d} \supseteq \underline{H}(R; \mathfrak{m} ^{k})$, from which the result follows.
\end{proof}

\subsection{Proofs of Propositions \ref{cor:bounded_prod}, \ref{cor:bound_width_cong}, and \ref{cor:bound_prod}}

\begin{proof}[{Proof of Proposition \ref{cor:bounded_prod}}] By \cite[Theorem 6.16]{PR}, for any local field $F$ and integer $D$, there are only finitely many connected semisimple algebraic subgroups of $\GL_D$ up to isomorphism. Therefore, given $K,S,D$, there are finitely many locally compact groups $L_1,\ldots,L_M$ such that, if $\underline{G} \subseteq \GL_D$ is a connected, simply connected semisimple algebraic group defined over $K$ and $v\in S$, then $\underline{G}(K_v)$ is isomorphic to one of the $L_j$s. Applying Lemmas \ref{lem:conj.open.compact.Lie} (for compact Lie groups), \ref{lem:conj.open.noncompact.Lie} (for non-compact groups), and \ref{lem:conj.compact.p-adic.open} (for compact totally disconnected groups) to the groups $L_i$, there is $N_1$, depending only on $K,S,D,\epsilon$ such that 
\begin{equation} \label{eq:bounded_prod.S}\left( \forall v\in S \right) \quad \gcl_{\underline{G}(K_v)}(g)^{N_1}=\langle \gcl_{\underline{G}(K_v)}(g) \rangle.
\end{equation} 
Let $N=\max \left\{ N_1,5D^3 \right\}$. We will show that the claim holds for $N$. Indeed, given $g$, applying Lemma \ref{lem:conj.open.general} for every $v\in T_\Gamma$, we get a neighborhood $W_v$ of 1 in $\underline{G}(K_v)$ such that
\begin{equation} \label{eq:bounded_prod.T_Gamma}\left( \forall v\in T_\Gamma\right) \quad \gcl_{\Gamma_v}(g)^{N} \supseteq W_v.
\end{equation} 
Finally, by Lemma \ref{lem:dense.p-adic}, we get
\begin{equation} \label{eq:bounded_prod.not.T_Gamma}\left( \forall v\notin T_\Gamma \cup S\right) \quad \gcl_{\Gamma_v}(g)^{N}=\langle \gcl_{\Gamma_v}(g) \rangle.
\end{equation} 
By the assumptions, $\underline{G}$ satisfies strong approximation, so $\Gamma$ is dense in $\prod_{v \neq w} \Gamma_v$.  It follows that the closure of $\gcl_\Gamma(g)$ in $\prod_{v \neq w} \Gamma_v$ is $\prod_{v \neq w} \gcl_{\Gamma_v}(g)$. Since $\Gamma_v=\underline{G}(K_v)$, for every $v\in S$, we get from \eqref{eq:bounded_prod.S}, \eqref{eq:bounded_prod.T_Gamma}, and \eqref{eq:bounded_prod.not.T_Gamma} that
\[
\prod_{v \neq w} \gcl_{\Gamma_v}(g) \supseteq \prod_{v\in T_\Gamma} W_v \times \prod_{v\notin T_{\Gamma}\cup \left\{ w \right\}} \langle \gcl_{\Gamma_v}(g) \rangle.
\]
\end{proof} 

\begin{proof}[Proof of Proposition \ref{cor:bound_width_cong}] The proof is almost identical to the proof of Proposition \ref{cor:bounded_prod}, except we apply Lemma \ref{lem:conj.compact.p-adic.open} for the groups $\Gamma_v$, for $v\in T_\Gamma$, and get that there is $N_2$ (this time, depending on $\Gamma$) such that
\begin{equation} \label{eq:bounded_prod.T_Gamma.instead} \left( \forall v \in T_\Gamma \right) \quad \gcl_{\Gamma_v}(g)^{N_2}=\langle \gcl_{\Gamma_v}(g) \rangle.
\end{equation} 
Taking $N=\max \left\{ N_1,N_2,5D^2\right\}$, the result follows when we use \eqref{eq:bounded_prod.T_Gamma.instead} instead of \eqref{eq:bounded_prod.T_Gamma}.
\end{proof} 

\begin{proof}[Proof of Proposition \ref{cor:bound_prod}] By the assumption that such $\Delta$ exists, we get that $\Gamma$ is a congruence subgroup. By replacing $\Gamma$ by a finite-index subgroup, we can assume that $\Gamma$ is a principal congruence subgroup. Let $N_1$ the constant from Proposition \ref{cor:bound_width_cong} applied with $\Gamma$ and $\epsilon=\frac{1}{2}$. For every $v\in T_\Gamma$, apply Lemma \ref{lem:conj.compact.p-adic.open} to $\Gamma_v$ to get a number $N_{2,v}$, and let $N_2=\max \left\{ N_{2,v} \mid v\in T_\Gamma \right\}$. We will show that the claim holds with $N=2N_1+N_2$.

Let $\Delta$ be a principal congruence subgroup. For a place $v$, denote the closure of $\Delta$ in $\Gamma_v$ by $\Delta_v$. 

By Strong Approximation, there is a non-central $\frac{1}{2}$-separated element $\alpha \in \Delta$. By the definition of $N_1$, $\gcl_\Gamma(\alpha)^{N_1}$ is dense in $\prod_{v\notin S} \langle \gcl_{\Gamma_v}(\alpha_v)\rangle$. For every $v\in T_\Gamma$, choose a natural number $k_v$ such that $\Gamma_v[\mathfrak{p}_v^{k_v}] \subseteq \langle \gcl_{\Gamma_v}(\alpha) \rangle$. Let $T$ be the finite set of places $v\notin T_\Gamma \cup S$ such that $\alpha \in Z(\Gamma_v) \cdot \Gamma_v[\mathfrak{p}_v]$. 

For $v\notin S\cup T \cup T_\Gamma$, $\gcl_{\Gamma_v}(\alpha)^{N_1}=\langle \gcl_{\Gamma_v}(\alpha) \rangle = \Gamma_v$. For every $v\in T$, there is a natural number $k_v$ such that $\Delta_v=\Gamma_v[\mathfrak{p}_v^{k_v}]$. By Strong Approximation, there is an element $\beta \in \Delta$ such that $\beta \in \Gamma_v[\mathfrak{p}_v^{k_v}] \smallsetminus Z(\Gamma_v) \cdot \Gamma_v[\mathfrak{p}_v^{k_v+1}]$. We have that $\gcl_{\Gamma_v}(\beta)^{N_1}=\langle \gcl_{\Gamma_v}(\beta) \rangle = \Delta_v$. For every $v\in T_\Gamma$, there are elements $\gamma_{v,1},\ldots,\gamma_{v,N}\in \Delta_v$ such that $\Delta_v=\gcl_{\Gamma_v}(\gamma_{v,1})\cdots \gcl_{\Gamma_v}(\gamma_{v,N})$. By Strong Approximation, choose elements $\gamma_1,\ldots,\gamma_{N_2}\in \Gamma$ such that $\gamma_i \cong \gamma_{v,i}\text{ (mod $\Gamma_v[\mathfrak{p}_v^{k_v}]$)}$, for all $i=1,\ldots,N_2$. For every $v\notin S$,
\[
\gcl_{\Gamma_v}(\alpha)^{N_1} \cdot \gcl_{\Gamma_v}(\beta)^{N_1} \cdot \gcl_{\Gamma_v}(\gamma_1) \cdots \gcl_{\Gamma_v}(\gamma_{N_2})=\Delta_v
\]
and the claim holds.
\end{proof}

\section{Definability of congruence subgroups}\label{sec:cong}

\begin{definition}
Under Setting \ref{nota:Gamma}, for every $q \lhd A$, let $G(A;\mathfrak{q})$ be the $\mathfrak{q}$th congruence subgroup of $G(A)$ and let  $G^*(A;\mathfrak{q})$ consists of the elements whose images in $G(A)/G(A;\mathfrak{q})$ are central. Then $\{\alpha G(A;\mathfrak{q}) \mid \mathfrak{q}\ne 0\}$ is a basis to the  congruence topology of $G(A)$ and  $\{\alpha G^*(A;\mathfrak{q}) \mid \mathfrak{q}\ne 0\}$ is a basis to a topology of $G(A)$ which we call the projective congruence topology. Finally,  denote $\Gamma[q]:=\Gamma \cap G(A;\mathfrak{q})$ and $\Gamma^*[q]:=\Gamma \cap G^*(A;\mathfrak{q})$. 
\end{definition}

\begin{theorem}\label{thm:principal.definable}
Let $\Gamma$ be as in Setting \ref{nota:Gamma}. If $G=\Spin_q$, assume further that $n \geq 9$. The exists a definable collection $\mathcal{F}$ of normal congruence subgroups of $\Gamma$ which contains $\{\Gamma^*[\mathfrak{q}] \mid  A \ne \mathfrak{q} \lhd A\}$.

\end{theorem}

\begin{proof}
	Proposition \ref{prop:unifrom_definable} below implies that there exists a definable collection $\mathcal{D}$ which is a basis of neighborhoods of identity under the projective congruence topology. Let $N$ be the constant given by Proposition \ref{cor:bound_prod}. Let $\mathcal{F}$ be the  collection of normal subgroups of $\Gamma$ which are of the form $\prod_{1 \le i \le N}\gcl_{\Gamma}(\alpha_i) \Lambda$ for some $\alpha_1,\ldots,\alpha_N \in \Gamma$  and some $\Lambda \in \mathcal{D}$. \end{proof}

\begin{proposition}\label{prop:unifrom_definable}
Let $\Gamma$ be as in Setting \ref{nota:Gamma}. If $G=\Spin_q$, assume further that $n \geq 9$. There exists a uniformly definable collection of subsets of $\Gamma$ which is a base of neighborhoods of identity under the projective congruence topology.  
\end{proposition}	

\begin{remark} By a base of neighborhoods of identity in Proposition \ref{prop:unifrom_definable} we mean a collection $\mathcal{A}$ of (not necessarily open) sets, that satisfies the following conditions: \begin{enumerate} 
\item For every $A\in \mathcal{A}$, the identity is in the interior of $A$.
\item For every open set $B$ containing the identity, there is an element $A\in \mathcal{A}$ such that $A \subseteq B$.
\end{enumerate} 
\end{remark} 	
		
\subsection{Proof of Proposition \ref{prop:unifrom_definable} for non-uniform $\Gamma$ }

\begin{lemma} \label{lem:root.system} Let $\Phi$ be a reduced and irreducible root system. Fix a lexicographic order on $\Phi$ and let $\Delta$ be the set of simple roots. For any $\alpha\in \Delta$, $\Span_\mathbb{Q} \left\{ r \in \Phi^+ \mid r= \sum_{\gamma \in \Delta} c_\gamma \gamma \text{ with } c_\alpha >0 \right\}=\Span_\mathbb{Q} \Phi$.
\end{lemma} 

\begin{proof} By the assumptions, the Dynkin diagram of $\Phi$ is connected. We claim that, for every $\beta \in \Delta$, $\beta$ is in the $\mathbb{Q}$-span of $\left\{ r \in \Phi^+ \mid r= \sum_{\gamma \in \Delta} c_\gamma \gamma \text{ with } c_\alpha >0 \right\}$. We show this by induction on the distance between $\alpha$ and $\beta$ on the Dynkin diagram. The basis case $\alpha=\beta$ is clear. Assume that $\beta \neq \alpha$, and let $\alpha=\alpha_0,\ldots,\alpha_n=\beta$ be a sequence of elements of $\Delta$ such that $\langle \alpha_i,\alpha_j \rangle\ne 0$ iff $i=j\pm 1$. By induction, we have that $\alpha_0,\ldots,\alpha_{n-1} \in \mathbb{Q}\left\{ r \in \Phi^+ \mid r= \sum_{\gamma \in \Delta} c_\gamma \gamma \text{ with } c_\alpha >0 \right\}$. Denoting the reflection in the root $\alpha_i$ by $s_{\alpha_i}$, the element $r=s_{\alpha_n} \circ \cdots \circ s_{\alpha_1}(\alpha_0)$ is of the form $\sum_{i=0}^n c_i \alpha_i$, with $c_0,c_n >0$, and the claim is proved.
\end{proof} 

\begin{lemma} \label{lem:L.acts.faithfully} Let $\underline{G}$ be a connected and simple algebraic group over $\mathbb{C}$, let $\underline{P} \subseteq \underline{G}$ be a maximal parabolic, and let $\underline{P}=\underline{L} \cdot \underline{U}$ be a Levi decomposition. Then, the kernel of the conjugation action map $\rho:\underline{L} \rightarrow \Aut(\underline{U})$ is $Z(\underline{G})$.
\end{lemma} 

\begin{proof} Let $\underline{T} \subseteq \underline{L}$ be a maximal torus. Let $\Phi$ be the root system corresponding to the action of $\underline{T}$ on the Lie algebra of $\underline{G}$, and, for $\chi \in \Phi$, denote the root space by $\mathfrak{g} ^ \chi$. There is a lexicographic order on $\Phi$, with corresponding set $\Phi ^+$ of positive roots and set $\Delta$ of simple roots, and a simple root $\alpha \in \Delta$ such that $\underline{P}$ is the parabolic attached to $\left\{ \alpha \right\}$. In particular, it follows that the Lie algebra of $\underline{U}$ is 
\[
\Lie \underline{U} = \bigoplus \left\{ \mathfrak{g} ^ \chi \mid \text{$\chi=\sum_{\gamma \in \Delta} c_\gamma \gamma$ with $c_\alpha >0$} \right\}. 
\]
By Lemma \ref{lem:root.system}, $\ker \rho \cap \underline{T}=\cap_{\beta \in \Delta} \ker \beta = Z(\underline{G})$. Since $\underline{L}$ is semisimple and $\ker \rho$ is normal in $\underline{L}$, if $\ker \rho \neq Z(\underline{G})$, then $\ker \rho \cap \underline{T} \neq Z(\underline{G})$, a contradiction.
\end{proof} 

\begin{lemma} \label{lem:centralizer.unipotent.alg.closed} Let $\underline{G}$ be a connected simple algebraic group defined over $\mathbb{C}$, let $\underline{P} \subseteq \underline{G}$ be a parabolic, and let $\underline{U}$ be the unipotent radical of $\underline{P}$. Then \begin{enumerate} 
\item $\Cent_{\underline{G}}(\underline{U}) \subset \underline{U} \cdot Z(\underline{G})$.
\item $\Cent_{\underline{G}}(\underline{U}) = Z(\underline{U}) \cdot Z(\underline{G})$.
\end{enumerate} 
\end{lemma} 

\begin{proof} \begin{enumerate} 
\item If $\underline{Q}$ is a parabolic containing $\underline{P}$ and $\underline{V}$ is the unipotent radical of $\underline{Q}$, then $\underline{V} \subseteq \underline{U}$ and $\Cent_{\underline{G}}(\underline{U}) \subseteq \Cent_{\underline{G}}(\underline{V})$. Hence, it is enough to prove the claim assuming $\underline{P}$ is maximal. Let $\underline{P}=\underline{L} \cdot \underline{U}$ be a Levi decomposition of $\underline{P}$. Since $\Cent_{\underline{G}}(\underline{U}) \subseteq N_{\underline{G}}(\underline{U})=\underline{P}$, the claim now follows by Lemma \ref{lem:L.acts.faithfully}.
\item This is an immediate consequence of the first claim.
\end{enumerate} 
\end{proof}

\begin{lemma} In Setting \ref{nota:Gamma}, assume that $P \subsetneq G$ is a maximal $K$-parabolic defined over $K$, and let $U \subseteq P$ be the unipotent radical of $P$. Then \begin{enumerate} 
\item $\Cent_{\Gamma}(U \cap \Gamma)=(Z(U) \cdot Z(G)) \cap \Gamma$.
\item $(Z(U) \cdot Z(G)) \cap \Gamma$ is definable.
\end{enumerate} 
\end{lemma} 

\begin{proof} \begin{enumerate} 
\item Since $P$ is defined over $K$, so is $U$. It follows that $U \cap \Gamma$ is Zariski-dense in $U$. By Lemma \ref{lem:centralizer.unipotent.alg.closed},
\[
\Cent_{\Gamma}(U \cap \Gamma)=\Cent_{G}(U \cap \Gamma) \cap \Gamma = \Cent_G(U)\cap \Gamma =( Z(U) \cdot Z(G)) \cap \Gamma.
\]
\item There are finitely many elements in $U \cap \Gamma$ that generate a Zariski dense subgroup of $U$. The result now follows from the first claim.
\end{enumerate} 
\end{proof} 

The following follows from \cite[Claim 2.11]{Rag} 

\begin{lemma} \label{lem:Rag.2.11} In Setting \ref{nota:Gamma}, assume that $P \subsetneq G$ is a maximal $K$-parabolic, and let $P=L \cdot U$ be a Levi decomposition defined over $K$. Denote the connected component of the Zariski closure of $L \cap \Gamma$ by $Z$. Then $Z$ acts non-trivially on $Z(U)$.
\end{lemma} 

%\begin{proof} Let $T \subset L$ be a maximal $K$-split torus. If $|S|>1$, then $A^ \times$ is infinite, and $T \cap \Gamma$ is Zariski dense in $T$. Since $T$ has a non-trivial root subgroup in $U$, we get the result. Assume now that $S=\left\{ v \right\}$, and choose a maximal $K_v$-split torus $T_v \supseteq T$.Then $\dim T_v=\rank_v G \geq 2$??
%\end{proof} 

\begin{remark} The assumption that the $S$-rank of $G$ is at least two is used in a crucial way in Lemma \ref{lem:Rag.2.11}. However, if $|S|>1$, then the claim is easier. Indeed, in this case, if $T \subseteq L$ is a maximal $K$-split torus, then $T\cap \Gamma$ is commensurable with $T(A)$, so it is Zariski dense in $T$. It is known that $T$ acts non-trivially on $Z(U)$, so the claim follows.
\end{remark} 

\begin{lemma} \label{lem:rag} In Setting \ref{nota:Gamma}, assume that $P \subsetneq G$ is a maximal $K$-parabolic, and denote the unipotent radical of $P$ by $U$. For every $v\notin S$ and every natural number $n$ there is a natural number $m$ such that, for any $g\in \Gamma \smallsetminus G^*(A;\mathfrak{p}_v^n)$, the set $\gcl_\Gamma(g)^{4\dim G}$ contains an element in $Z(U) \smallsetminus G(A;\mathfrak{p}_v^m)$.
\end{lemma} 

\begin{proof} Choose a Levi decomposition $P=L \cdot U$ and a maximal $K$-split torus $T \subseteq L$. There is a lexicographic ordering of the roots of $T$ acting on the Lie algebra of $G$ and a simple root $\alpha$ of $T$ such that $P$ is the parabolic corresponding to $\alpha$. By \cite[\S5]{BT65}, there is an element $w\in N_G(T)(K)$ that switches the positive and negative roots; the image of $w$ in the Weyl group has order 2. This implies that $P^{w^2}=P$ and, in particular, that $P\cap P^w$ is $w$-invariant. 

Let $\alpha'=-w(\alpha)$ (so $\alpha'$ is positive), let $P'$ be the maximal parabolic corresponding to $\alpha'$, and let $U'$ be its unipotent radical. By Bruhat decomposition, the map $\beta:U' \times P \rightarrow G$ given by $\beta(u,p)=uwp$ is a $K$-isomorphism between $U' \times P$ and a Zariski open set in $G$.

Denote the connected component of the identity in the Zariski closure of $P\cap P^w \cap \Gamma$ by $M$. Let $f:U' \times P \times M \rightarrow P$ be the function $f(u,p,x)=p ^{-1} w^{-1} x ^{-1} w pu x  u ^{-1}$. By \cite[Lemma 2.8]{Rag}, the subgroup generated by $f(U' \times P \times M)$ contains $\left( \overline{\Gamma \cap L}^Z \right) ^0$. By Lemma \ref{lem:Rag.2.11}, there is an element $u_0\in Z(U) \cap \Gamma$ such that $[f(U' \times P\times M),u_0]\neq 1$. Since $U'\times P \times M$ is connected and $[f(1,1,1),u_0]=1$, we get that the morphism $h:U' \times P \times M \rightarrow Z(U)$ given by $h(u,p,x) = [f(u,p,x),u_0]$ is not constant.

Given $v$ and $n$, by Lemma \ref{lem:conj.open.general} there is a constant $a$ such that, for every $g\in \Gamma \smallsetminus G^*(A_v;\mathfrak{p}_v^n)$, the set $\gcl_\Gamma(g)^{\dim G}$ is dense in $G(A_v;\mathfrak{p}_v^a)$. Since $h$ is non-constant but $h(u,p,1)=1$, for every $u,p$, there is a point $(u_1,p_1)\in \beta ^{-1} (G(A_v;\mathfrak{p}_v^a))$ such that the function $x\in M \mapsto h(u_1,p_1,x)$ is non-constant. Since $u_1 \in U(K)$, there is a natural number $b$ such that $[u_1,\Gamma \cap M(A_v; \mathfrak{p}_v^b)] \subseteq \Gamma$. It follows that there is a natural number $c$ such that $f(u_1,p_1,M(A_v;\mathfrak{p}_v^b)) \not \subseteq U(A_v;\mathfrak{p}_v^c)$. By continuity, there is a neighborhood $\mathcal{V} \subseteq \beta ^{-1} (G(A_v;\mathfrak{p}_v^a))$ of $(u_1,p_1)$ such that, for every $(u,p)\in \mathcal{V}$, $f(u,p,M(A_v;\mathfrak{p}_v^b)) \not \subseteq U(A_v;\mathfrak{p}_v^c)$ and $[u_1, \Gamma \cap M(A_v;\mathfrak{p}_v^b)] \subseteq \Gamma$. We will show that the claim of the lemma holds with $m=c$.

Indeed, suppose that $g\in \Gamma \smallsetminus G(A_v;\mathfrak{p}_v^n)$. Then, there is an element $g_1\in \gcl_\Gamma(g)^{\dim G} \cap \beta(\mathcal{V})$. Writing $g_1=uwp$, there is an element $x\in \Gamma \cap M(A_v;\mathfrak{p}_v^b)$ such that $h(u,p,x)\notin U(A_v;\mathfrak{p}_v^c)$. We have that
\[
xuwpx ^{-1}=xg_1x ^{-1} \in \gcl_\Gamma(g)^{\dim G}
\]
and, since $[x,u]\in \Gamma$, we get that
\[
xux ^{-1} wp u x u ^{-1} x ^{-1}=[x,u]g_1[x,u] ^{-1} \in \gcl_\Gamma(g)^{\dim G}.
\]
We get that
\[
x f(u,p,x) x ^{-1} =x p ^{-1} w ^{-1} x ^{-1} wp u x u ^{-1} x ^{-1} = \left( x g_1 ^{-1} x ^{-1}\right) \left( [x,u]g_1[x,u] ^{-1}\right) \in \gcl_\Gamma(g)^{2\dim G}.
\]
Therefore, $f(u,p,x)\in \gcl_\Gamma(g)^{2\dim G}$, so $h(u,p,x)\in \gcl_\Gamma(g)^{4\dim G}$. By construction, $h(u,p,x)\notin U(A_v;\mathfrak{p}_v^c)$, and the claim is proved.

\end{proof}

\begin{corollary} \label{cor:quick.escape.U} Under Setting \ref{nota:Gamma}, assume that $P \subsetneq G$ is a maximal $K$-parabolic, and denote the unipotent radical of $P$ by $U$. For every ideal $I \triangleleft A$, there is an ideal $J \triangleleft A$ such that, if $\gamma \notin \Gamma[I]$, then $\gcl_\Gamma(\gamma)^{4\dim G}$ contains an element in $Z(U)(A) \smallsetminus Z(U)(A;J)$.
\end{corollary} 

\begin{lemma} \label{lem:separate.ZG} There is an ideal $J_0$ such that $\Gamma[J_0] \cap \Cent_G(U) \subseteq Z(U)$.
\end{lemma} 

\begin{proof} It is known that there is an algebraic representation of $G$ and a vector $a$ such that $U=\Stab_G(a)$. It follows that there is a regular function $f$ on $G$ such that $f(xu)=f(x)$ for every $x\in G$ and $u\in Z(U)$, and such that $f(z)\neq f(1)$, for every $z\in Z(G)\smallsetminus \left\{ 1 \right\}$. We can also assume that $f$ is defined over $A$. Let $J_0$ be an ideal such that $f(z)\not\equiv f(1)\text{ (mod $J_0$)}$, for any $z\in Z(G)\smallsetminus \left\{ 1 \right\}$. If $x\in \Gamma[J_0] \cap \Cent_G(U)$, then $x=zu$, where $z\in Z(G)$ and $u\in Z(U)$ and also $f(z)=f(x) \equiv f(1) \text{ (mod $J_0$)}$, so $z=1$.
\end{proof} 

\begin{proof}[Proof of Proposition \ref{prop:unifrom_definable} for non-uniform $\Gamma$] By assumption, there is a proper, maximal $K$-parabolic $P \subsetneq G$. Let $U$ be its unipotent radical. We need to show that there is a uniformly-definable collection $X \subseteq (\Gamma \smallsetminus Z(\Gamma)) \times \Gamma$ such that \begin{enumerate} 
\item \label{item:uniform.definable.open} For each $\delta \in \Gamma \smallsetminus Z(\Gamma)$, $X_\delta$ is a symmetric, conjugation-invariant subset that contains some congruence subgroup.
\item \label{item:uniform.definable.small} For each ideal $I$, there is a $\delta \in \Gamma$ such that $X_\delta \subseteq \Gamma^*[I]$.
\end{enumerate} 
By \cite[Theorem 5.1]{ALM}, there is a constant $N_1$ such that, for any non-central element $\gamma \in \Gamma$, there is an ideal $I(\gamma)$ such that $\gcl_\Gamma(\gamma)^{N_1} \supseteq U(I(\gamma))$. Let $N=\max \left\{ N_1,4\dim G \right\}$ and let $X \subset \Gamma \times \Gamma$ be the definable set consisting of all pairs $(x,y)$ such that $\gcl_\Gamma(y)^N \cap \Cent_G(U) \subseteq \gcl_\Gamma(x)^N$. If $\delta \in \Gamma$ is non-central, then, by Lemma \ref{lem:separate.ZG}, $\Gamma[J_0I(\delta)] \subseteq X_{\delta}$, proving \eqref{item:uniform.definable.open}. On the other hand, given an ideal $I$ of $A$, let $J$ be the ideal obtained by applying Corollary \ref{cor:quick.escape.U} to $I$, and let $\delta \in \Gamma[J]\smallsetminus \left\{ 1\right\}$. By the definition of $J$, if $\gamma \notin \Gamma^*[I]$, then $\gcl_\Gamma(\gamma)^{N} \cap U \not\subseteq \Gamma[J]$, and, in particular, $\gamma \notin X_\delta$. This proves \eqref{item:uniform.definable.small}.

%By \ref{cor:bound_prod}, there is a natural number $N$ such that any principal congruence subgroup contains a dense set of the form $\gcl_\Gamma(\alpha_1) \cdots \gcl_\Gamma(\alpha_N)$, for some $\alpha_1,\ldots,\alpha_N \in \Gamma$. Let $Y \subset (\Gamma \smallsetminus Z(\Gamma)) \times \Gamma ^N$ be the set of all tuples $(\delta,\alpha_1,\ldots,\alpha_N)$ such that $X_\delta \cdot \gcl_\Gamma(\alpha_1) \cdots \gcl_\Gamma(\alpha_N)$ is a subgroup. Finally, let $Z \subset Y \times \Gamma$ be the set of all tuples $(\delta,\alpha_1,\ldots,\alpha_N,\beta)$ such that $\beta \in X_\delta \cdot \gcl_\Gamma(\alpha_1) \cdots \gcl_\Gamma(\alpha_N)$. The uniformly definable collection $Z_{\delta,\alpha_1,\ldots,\alpha_N}$ is a basis of neighborhoods of $1$ in the congruence topology, and from this collection, it is easy to construct a uniformly definable collection as required by the Proposition.
\end{proof} 

\subsection{Proof of Proposition \ref{prop:unifrom_definable} for $G=\Spin$ }

\begin{setting}\label{nota:subspace} Under Setting \ref{nota:Gamma}, assume that $G=\Spin_q$ and $n \ge 9$. Let $c_1,c_2,c_3 \in A^n$ be non-isotropic orthogonal vectors such that $i_q \left( \left( K_wc_0 + K_wc_1 \right) ^\perp \right) \geq 2$ and $i_q(C_w)=1$ where $C:=Kc_0+Kc_1+Kc_2$. Let $\Lambda$ be the subgroup of $\Gamma$ consisting of the elements which act on $C$ as $\pm 1$. 
\end{setting}

\begin{lemma}\label{lemma:def_stab} Under Setting \ref{nota:subspace}, $\Lambda$ is definable.
\end{lemma}
\begin{proof} Let $\Delta$ be the subgroup of $\Gamma$ consisting of the elements which act as the identity on $C^\perp$. Then $\Delta$ is a congruence subgroup of $\Spin_{q\restriction C}(K)$ so  $\Delta$ is finitely generated and the action of $\Delta$ on $C$ is absolutely irreducible. Thus, $\Cent_\Gamma(\Delta)=\Lambda$ is definable. 
\end{proof}

We will need the following lemma which follows from Theorem \ref{thm:effective_kneser}. The proof of will be given \S\ref{subsection:uff} below,

\begin{lemma}\label{lemma:uff}
Under Setting \ref{nota:subspace}, for every  $\epsilon >0$, there exists $N=N(\Gamma,\epsilon)$ such that the following holds:

If $\alpha \in \Gamma$ is $\epsilon$-separated, then $\gcl_\Gamma(\alpha)^{N}(c_0,c_1,c_2)$ contains an open neighborhood of of $\bar{c}$ in $\Gamma (c_0,c_1,c_2)$ with respect to the $S$-adelic topology. 

\end{lemma}

\begin{corollary}\label{cor:uff_1} Under Setting \ref{nota:subspace},  let $\alpha \in \Gamma$ be a non-identity element. If
$	\gcl_{\Gamma}(\alpha)^{M}\Lambda=	\gcl_{\Gamma}(\alpha)^{M+1}\Lambda$, then 
	 $\gcl_{\Gamma}(\alpha)^{M}\Lambda$  is  a congruence subgroup of $\Gamma$. 
\end{corollary}
\begin {proof}  The assumption implies that $	\gcl_{\Gamma}(\alpha)^{M}\Lambda=\langle	\gcl_{\Gamma}(\alpha)\rangle\Lambda$.  Fix non-central $\beta \in \langle	\gcl_{\Gamma}(\alpha)\rangle$.   Lemma \ref{lemma:uff} implies that there exists $N$ and $0 \ne \mathfrak{q} \lhd A$ such that $$\gcl_{\Gamma}(\alpha)^{M}\Lambda=\langle	\gcl_{\Gamma}(\alpha)\rangle\Lambda \supseteq \gcl_{\Gamma}(\beta)^{N}\Lambda \supseteq \Gamma[\mathfrak{q}].$$
\end{proof}

\begin{corollary}\label{cor:uff_2} Under Setting \ref{nota:subspace}, for every $\epsilon>0$ there exists $N=N(\epsilon,\Gamma)$ such every for every $\epsilon$-separated  $\alpha \in \Gamma$, $$\gcl_{\Gamma}(\alpha)^{N}\Lambda =\gcl_{\Gamma}(\alpha)^{N+1}\Lambda =\langle \gcl_{\Gamma}(\alpha)\rangle \Lambda.$$

\end{corollary}
\begin{proof} 
	Lemma \ref{lemma:uff} implies that there are $N_1$ and $\mathfrak{q}\ne 0$ such that $\ccl_{\Gamma}(\alpha)^{N_1}\bar{c}$ contains the $\mathfrak{q}$-congruence neighborhood of of $\bar{c}$ in $\Gamma \bar{c}$.
Let $N_2$ be the constant given by Proposition  \ref{cor:bound_width_cong}. For $N=N_1+N_2$, we have $\gcl_{\Gamma}(\alpha)^{N}\Lambda = \langle \gcl_{\Gamma}(\alpha)\rangle\Lambda.$

\end{proof}

\begin{proof}[Proof of Propostion \ref{prop:unifrom_definable}]
We use Setting \ref{nota:subspace}. Let $N$ be the natural number obtained by applying Corollary  \ref{cor:uff_2} with $\epsilon =\frac{1}{2}$.
Then for every $\gamma \in \Gamma$,  $\gamma\Lambda\gamma^{-1}$ is the subgroup of $\Gamma$ which acts on $\gamma C$ as $\pm 1$. 
Moreover, for every $\alpha,\gamma \in \Gamma$ and every $M$, $\gcl_\Gamma(\alpha)^M(\gamma\Lambda\gamma^{-1})=\gamma(\gcl_\Gamma(\alpha)^M\Lambda)\gamma^{-1}$. 

%Thus, $	\gcl_{\Gamma}(\alpha)^{N}\Lambda=	\gcl_{\Gamma}(\alpha)^{N+1}\Lambda$ if and only if $	\gcl_{\Gamma}(\alpha)^{N}(\gamma\Lambda\gamma^{-1})=	\gcl_{\Gamma}(\alpha)^{N+1}(\gamma\Lambda\gamma^{-1})$

Choose $\gamma_1,\ldots,\gamma_{n+1}\in \Spin_q(A)$ such that $\gamma_1c_1,\ldots,\gamma_{n+1}c_1$ are in general position in $K^{n+1}$ (this means that every $n$ of them are linearly independent).   Lemma \ref{lemma:def_stab} implies that $Y:=\{\alpha \in \Gamma \mid \gcl_{\Gamma}(\alpha)^{N}\Lambda=	\gcl_{\Gamma}(\alpha)^{N+1}\Lambda\}$ is a definable subset of $\Gamma$. Let $X \subseteq \Gamma \times \Gamma$ be the definable subset 
$$X:=\{(\alpha,\beta)\in \Gamma\times \Gamma \mid \alpha \in Y\text{ and } \beta \in \cap_{1 \le i \le n+1}\gamma_i(\gcl_\Gamma(\alpha)^N\Lambda)\gamma_i^{-1}\}. $$

We will show that $\{X_\alpha \mid \alpha \in Y\}$ is a uniformly definable collection of subgroups of $\Gamma$ which is a basis of neighborhoods of identity under the projective congruence topology. 

Corollary \ref{cor:uff_1}  implies that, for every $\alpha \in Y$, $X_\alpha$ is a congruence subgroup. Let $0 \ne \mathfrak{q} \lhd A$. We want to show that there exists $\alpha \in Y$ such that $X_\alpha\subseteq \Gamma^*[\mathfrak{q}]$. 
Let $\mathfrak{p}$ be a prime ideal of odd residue characteristic such that the reductions of $a_i$ modulo $\mathfrak{p}$ are in general position in $(A/\mathfrak{p})^n$. Let $m \ge 1$ be such that $mA^n \subseteq \Span_A\{\gamma_i c_1 \mid 1 \le i \le n\}$. Choose a $\frac{1}{2}$-separated $\alpha \in \Gamma[m\mathfrak{p}\mathfrak{q}]$. Corollary \ref{cor:uff_2} implies that  $\alpha \in Y$. We will show that $X_\alpha \subseteq \Gamma^*[\mathfrak{q}]$. 

Let $\beta \in X_\alpha$. For every $1 \le i \le n+1$ there exists $\epsilon_i \in \{\pm 1\}$ such that ${\beta} \gamma_ic_1 =\epsilon_i \gamma_ic_1(\textrm{mod } m\mathfrak{p}\mathfrak{q})$. By the choice of $\mathfrak{p}$, $\epsilon_1=\epsilon_2=\cdots=\epsilon_{m+1}=\pm1$. Since $mA^n \subseteq \Span_A\{\gamma_ic_1 \mid 1 \le i \le n\}$, ${\beta}\in \Gamma^*[\mathfrak{p}\mathfrak{q}]$. \end{proof}

\section{Interpretation of $\mathbb{Z}$}\label{sec:inter}

\begin{theorem}\label{thm:interpretation} Let $\Gamma$ be as in Setting \ref{nota:Gamma}.   Then, there is an element $\alpha\in \Gamma$ of infinite order such that \begin{enumerate}
\item\label{item:def_cyc} $\langle \alpha \rangle$ is definable.
\item\label{item:def_prod} The map $(\alpha^r,\alpha^s) \mapsto \alpha^{rs}$ is definable.
\end{enumerate}
In particular, $\Gamma$ interprets $\mathbb{Z}$.
\end{theorem}

The proof of Theorem \ref{thm:interpretation} is based on the following two propositions: 

\begin{proposition}\label{prop:good_subgroup} Under Setting \ref{nota:Gamma}, there are a definable subgroup $\Lambda$ of $\Gamma$,  a regular quadratic form $f$ on $K^3$ such that $i_f(K^3_w)=1$ and a homomorphism $\rho :\Spin_f \rightarrow \Gr$ which is an isogeny over its image such that $\rho(\Spin_q) \cap \Gamma$ has finite index in $\Lambda$. 
\end{proposition} 

\begin{proposition}\label{prop:newnew} Under Setting \ref{nota:Gamma}, denote $\mathcal{P}=\{\mathfrak{p}^k \mid \mathfrak{p}\lhd A \text{ is prime and } k \ge 1\}$. Let $f$, $\rho$ and $\Lambda$ be as in Proposition \ref{prop:good_subgroup}. For every infinite order semisimple $\alpha \in \Lambda$, there exist $d,e \ge 1$ such that, for every cofinite $\mathcal{R} \subseteq \mathcal{P}$, the set
	 $$
	\{\gamma \in Z(\Cent_{\Lambda}(\alpha)) \mid (\forall 1 \le i \le d\ \forall \mathfrak{r} \in \mathcal{R})\ (\gamma\alpha^{-i})^{e}\notin \Gamma^*[\mathfrak{r}]\}
	 $$
	 is finite. 
\end{proposition}

\subsection{Proof of Proposition \ref{prop:good_subgroup} for non-uniform $\Gamma$}

We will need the following straightforward extension of the notions of definable sets, imaginaries, and interpretations from a single structure to a sequence of structures. 

\begin{definition} Let $L$ be a first order language, let $(M_n)_{n\in \mathbb{N}}$ be a sequence of $L$-structures, and let $k\in \mathbb{N}$. We say that a sequence of subsets $A_n \subset M_n^k$ is definable if there is an $L$-formula $F(x_1,\ldots,x_k,y_1,\ldots,y_m)$ and, for each $n$, an $m$-tuple $(c_1^n,\ldots,c_m^n)\in M_n^m$ such that $A_n=\left\{ (x_1,\ldots,x_k)\in M_n^k \mid F(x_1,\ldots,x_k,c_1^n,\ldots,c_m^n) \right\}$, for every $n$. In this case, we also say that the sequence $(A_n)$ is a definable sequence of sets in $(M_n)$. In a similar manner, define the notions of definable sequence of functions between definable sequences of sets in $(M_n)$, and the notion of a sequence of imaginary sets in $(M_n)$.
\end{definition} 

\begin{definition} Let $L,L'$ be first order languages, let $(M_n)_{n\in \mathbb{N}}$ be a sequence of $L$-structures, and let $(M_n')_{n\in \mathbb{N}}$ be a sequence of $L'$-structures. \begin{enumerate} 
\item An interpretation of $(M_n')$ in $(M_n)$ is a pair $(F,f)$, where $F=(F_n)$ is a sequence of imaginaries in $(M_n)$ and $f=(f_n)$ is a sequence of bijections $f_n:F_n \rightarrow M_n'$ such that \begin{enumerate}
%\item For each constant symbol $c$ of $L_2$, the sequence $(f_n ^{-1} (c^{M'_n}))$ is a definable sequence of elements of $(F_n)$.
\item For each $k$-ary relation symbol $r$ of $L_2$, the sequence of imaginaries $(f_n ^{-1} (r^{M'_n}))$ is a definable sequence of subsets of $F_n^k$.
\item For every function symbol $g$ of $L_2$, say of arity $(r,s)$, the sequence of functions $(f_n ^{-1} \circ g^{M'_n} \circ f_n)$ is definable.
\end{enumerate}
\item An interpretation $(F,(f_n))$ of $(M'_n)$ in $(M_n)$ is called trivial if the sequence of functions $(f_n)$ is definable.
\item A pair $(\mathscr{F}_{1,2},\mathscr{F}_{2,1})$ consisting of an interpretation $\mathscr{F}_{1,2}$ of $(M_n)$ in $(M'_n)$ and an interpretation $\mathscr{F}_{2,1}$ of $(M'_n)$ in $(M_n)$ is called a bi-interpretation if the compositions $\mathscr{F}_{1,2} \circ \mathscr{F}_{2,1}$ and $\mathscr{F}_{2,1} \circ \mathscr{F}_{1,2}$ are trivial.
\end{enumerate} 
\end{definition} 

\begin{proposition} \label{prop:PSL_2.biinterpretation} Let $Q$ be the set of prime powers. The sequence $(\PSL_2(\mathbb{F}_q))_{q\in Q}$ is in bi-interpretation with the sequence $(\mathbb{F}_q)_{q\in Q}$.
\end{proposition} 

\begin{proof} It is clearly enough to restrict the sequence to $q>3$, which we will do in the rest of the proof. We first construct an interpretation $\mathscr{F}$ of $(\mathbb{F}_q)_q$ in $(\PSL_2(\mathbb{F}_q))_q$.  For every $q>3$, let $u_q=\left(\begin{matrix} 1 & 1 \\ 0 & 1 \end{matrix}\right) \in \PSL_2(\mathbb{F}_q)$ and choose $t_q=\left(\begin{matrix} \epsilon & 0 \\ 0 & \epsilon ^{-1}\end{matrix}\right)\in \PSL_2(\mathbb{F}_q)$, for some $\epsilon \in \mathbb{F}_q \smallsetminus \left\{ 0,1,-1 \right\}$. The sequences $U_q=\Cent_{\PSL_2(\mathbb{F}_q)}(u_q)$ and $T_q=\Cent_{\PSL_2(\mathbb{F}_q)}(t_q)$ are definable, as well as the sequence of functions $U_q \times U_q \rightarrow U_q$ taking $\left( \left(\begin{matrix} 1 & x \\ & 1\end{matrix}\right), \left(\begin{matrix} 1 & y \\ & 1\end{matrix}\right)\right)$ to $\left(\begin{matrix} 1 & x+y \\ & 1\end{matrix}\right)$. For every $q$ and every $a:=\left(\begin{matrix} 1 & x \\ & 1 \end{matrix}\right), b:=\left(\begin{matrix} 1 & y \\ & 1 \end{matrix}\right) \in U_q \smallsetminus \left\{ 1 \right\}$, there are $s_1,s_2\in T_q$ such that $(s_1 ^{-1} u_q s_1) (s_2 ^{-1} u_q s_2) =a$; for every such $s_1,s_2$, we have $(s_1 ^{-1}  b s_1) (s_2 ^{-1} b s_2) =\left(\begin{matrix} 1 & xy \\ & 1 \end{matrix}\right)$. This shows that the bijection $U_q \rightarrow \mathbb{F}_q$ given by $\left(\begin{matrix} 1 & x \\ & 1 \end{matrix}\right) \mapsto x$ is an interpretation.

In the other direction, let $\mathscr{G}$ be the interpretation of $\PSL_2(\mathbb{F}_q)$ in $\mathbb{F}_q$ whose imaginary is the set of 4-tuples $(x,y,z,w)\in \mathbb{F}_q$ satisfying the equation $xw-yz=1$ (modulo $\pm1$), and whose bijection is $(x,y,z,w) \mapsto \left(\begin{matrix} x & y \\ z & w \end{matrix}\right)$.

The inverse of composition $\mathscr{F} \circ \mathscr{G}$ is the function $x \mapsto (1,x,0,1)$ from $\mathbb{F}_q$ to $\mathbb{F}_q^4$, which is clearly definable.

Finally, the inverse of the composition $\mathscr{G}\circ \mathscr{F}$ is the sequence of functions $h_q: \PSL_2(\mathbb{F}_q) \rightarrow \PSL_2(\mathbb{F}_q)^4$ given by 
\[
h_q\left(\begin{matrix} a & b \\ c & d\end{matrix}\right) = \left( \left(\begin{matrix} 1 & a \\  & 1\end{matrix}\right) , \left(\begin{matrix} 1 & b \\  & 1\end{matrix}\right) , \left(\begin{matrix} 1 & c \\  & 1\end{matrix}\right) , \left(\begin{matrix} 1 & d \\  & 1\end{matrix}\right) \right) .
\]
We need to show that $h_q$ is definable. Let $v_q=\left(\begin{matrix} 1 & 0 \\ 1 & 1 \end{matrix}\right)\in \PSL_2(\mathbb{F}_q)$, and let $V_q=\Cent_{\PSL_2(\mathbb{F}_q)}(v_q)$. The restriction of $h_q$ to $U_q$ is definable, as well as its restriction to $V_q$. Using the definability of addition and multiplication operations in $U_q$, we get that the restriction of $h_q$ to $U_qV_qU_qV_q$ is definable, but $U_qV_qU_qV_q=\PSL_2(\mathbb{F}_q)$.
\end{proof} 

\begin{proposition} \label{prop:uniform.biint.finite} Let $Q$ be the set of prime powers. Let $G$ be a connected, simply connected and split simple group scheme over $\mathbb{Z}$. Then the sequence $(G(\mathbb{F}_q)/Z(G(\mathbb{F}_q)))_{q\in Q}$ is bi-interpretable with the sequence $(\mathbb{F}_q)_{q\in Q}$.
\end{proposition} 

\begin{proof} Let $r$ be the rank of $G$. It is enough to restrict the claim to the subsequence $q>r+1$. Choose a maximal split torus $T$, and, for every $q>r+1$, choose a regular element $t_q\in T(\mathbb{F}_q)$. The sequence $T(\mathbb{F}_q)=\Cent_{G(\mathbb{F}_q)}(t_q)$ is definable. Let $\alpha$ be a root of $(G,T)$, let $U_\alpha \cong \mathbb{G}_a$ be the root subgroup, and choose $u_{q,\alpha}\in U_\alpha (\mathbb{F}_q)$. Since $\alpha$ is a non-trivial character, there is a constant $k$ such that $\alpha(T(\mathbb{F}_q))$ contains the collection of all $k$th powers in $\mathbb{F}_q ^ \times$. It follows that there is a constant $C$ (independent of $q$) such that $\underbrace{\alpha(T(\mathbb{F}_q))+\cdots+\alpha(T(\mathbb{F}_q))}_{\text{$C$ times}}=\mathbb{F}_q$. It follows that every element in $U_\alpha(\mathbb{F}_q)$ is a product of $C$ conjugates of $u_{q,\alpha}$ by elements of $T(\mathbb{F}_q)$. This implies that the sequence $U_\alpha(\mathbb{F}_q)$ is definable. The proof now continues in the same way as in Proposition \ref{prop:PSL_2.biinterpretation}.
\end{proof} 

\begin{definition} Let $d\in \mathbb{N}$, let $R$ be a domain whose characteristic is bigger than $d$, and let $u\in \GL_d(R)$ be a unipotent element. Denote the fraction field of $R$ by $Frac(R)$. We define $u^{R}$ to be the set $\exp \left( R \log(u) \right) \subseteq \GL_d(Frac(R))$. Note that $u^{R}$ is a group.
\end{definition} 

\begin{corollary} \label{cor:uniformly.definable.unipotent.subgroup} Let $G$ be a simply connected Chevalley group scheme over $\mathbb{Z}$. There is an integer $d$ and a first order formula $F(x,y)$ in the language of groups such that, for every finite field $\mathbb{F}_q$ of characteristic larger than $d$, every unipotent element $u\in G(\mathbb{F}_q)/Z(G(\mathbb{F}_q))$, and every $g\in G(\mathbb{F}_q)/Z(\mathbb{F}_q)$, we have $G(\mathbb{F}_q)/Z(\mathbb{F}_q)$ satisfies $F(g,u)$ if and only if $g\in u^{\mathbb{F}_q}$.
\end{corollary} 

\begin{proof} Using the bi-interpretation of $G(\mathbb{F}_q)/Z(\mathbb{F}_q)$ and $\mathbb{F}_q$, the sequence of Lie algebras $\mathfrak{g}(\mathbb{F}_q)$, as well as the exponential and logarithm maps, are definable.
\end{proof} 

The following is well known (the first claim follows from generic flatness and \cite[Expose XIX, Proposition 3.8]{SGA3}; the second claim follows from Strong Approximation and the construction of the finite simple groups of Lie type): 
\begin{lemma} \label{lem:simple.quotients} Under Setting \ref{nota:Gamma}, \begin{enumerate} 
\item For all but finitely many prime ideals $\mathfrak{p} \lhd A$, $G_{A/\mathfrak{p}}$ is a simple and connected algebraic group.
\item For all but finitely many prime ideals $\mathfrak{p} \lhd A$, we have $\Gamma / \Gamma ^*[\mathfrak{p}]=G(A/\mathfrak{p})/Z(G(A/\mathfrak{p}))$ is a simple group of the Lie type of $G$.
\end{enumerate} 
In particular, the sets 
\[
\left\{ \Delta \subseteq \Gamma \mid \text{$\Delta$ is a maximal normal congruence subgroup} \right\} 
\]
and
\[
\left\{ \Gamma ^*[\mathfrak{p}] \mid \text{$\mathfrak{p} \lhd A$ is a prime ideal} \right\} 
\]
are commensurable.
\end{lemma} 

\begin{lemma} \label{lem:silly} Let $n,C$ be natural numbers greater than 1. If $F$ is a field and $x,y\in \GL_n(F)$ satisfy $x ^{-1} y x = y^C$, then $y^{C^n!}$ is a unipotent.
\end{lemma} 

\begin{proof} Let $\lambda_1,\ldots,\lambda_n$ be the eigenvalues of $y$. Since $\left\{ \lambda_i \right\} =\left\{ \lambda_i^C\right\}$, all $\lambda_i$ are roots of unity of order at most $C^n$, and the claim follows.
\end{proof} 

\begin{corollary} \label{cor:uniformly.definable.congruence} Under Setting \ref{nota:Gamma}, there is an infinite set $\mathcal{Q}$ of primes of $A$ such that
\begin{enumerate} 
\item For every $\mathfrak{q} \in \mathcal{Q}$, $G_{A/\mathfrak{q}}$ is split.
\item The collection $\left\{ \Gamma ^*[\mathfrak{q}] \mid \mathfrak{q} \in \mathcal{Q} \right\}$ is uniformly definable.
\end{enumerate} 
\end{corollary} 

%\nir{write a proof using bi-interpretation}

\begin{proof}  By Theorem \ref{thm:principal.definable} there is a collection $\mathcal{F}_1$ of normal congruence subgroups of $\Gamma$ that contains all subgroups of the form $\Gamma ^*[\mathfrak{q}]$. Taking the elements of $\mathcal{F}_1$ which are maximal, we get a uniformly definable collection $\mathcal{F}_2$ which, by Lemma \ref{lem:simple.quotients}, is commensurable with $\left\{ \Gamma ^*[\mathfrak{q}] \mid \mathfrak{q} \triangleleft A \right\}$. By imposing a lower bound on the index of the subgroup, we get a uniformly definable collection $\mathcal{F}_3$ consisting of almost all subgroups of the form $\Gamma ^*[\mathfrak{q}]$.

%Let $\Phi$ be the absolute root system of $G$ (i.e. the root system of $G_\mathbb{C}$) and let $G_\Phi$ be the Chevalley group corresponding to $\Phi$. Fix a presentation $G_\Phi(A)=\langle x_1,\ldots,x_n \mid r_1,\ldots,r_m \rangle$, and let $\mathcal{F}_4$ be the collection of subgroups $\Delta \in \mathfrak{F}_3$ such that there are elements $a_1,\ldots,a_n\in \Gamma / \Delta$ satisfying the relations $r_1,\ldots,r_m$. \nir{write about the finite quotients of $G_\Phi$ and of $\Gamma$} If $\Delta \in \mathcal{F}_4$, then $\Gamma / \Delta$ is a quotient of $G_\Phi(A)$ \nir{right now, just contains a quotient}, so, by Margulis's Normal Subgroup Theorem, is isomorphic to $G_\Phi(A/\mathfrak{q})/Z(G_\Phi(A/\mathfrak{q}))$.

Let $n$ be such that there is an embedding $G \hookrightarrow \GL_n$. Let $r$ be the rank of $G$ and let $\Phi \subseteq X^*(\mathbb{G}_m^r)$ be the absolute root system of $G$. Choose a basis $\beta_1,\ldots,\beta_r$ to $X_*(\mathbb{G}_m^r)$ such that $\alpha(\beta_i) \geq 0$, for all $\alpha \in \Phi ^+$, and denote $C=\max \left\{ 2^{\alpha(\beta_i)} \mid \alpha \in \Phi, i=1,\ldots,r \right\}$. By Chebotarev Density Theorem, there are infinitely many prime ideals $\mathfrak{p} \lhd A$ such that $\mathfrak{p} \not| C^n!$, $G_{A/\mathfrak{p}}$ is split, and $A/\mathfrak{p}$ contains a primitive $(r+1)$ root of unity, which we denote by $\zeta_\mathfrak{p}$. In this case, let $\mathbb{G}_m^r \cong T \subseteq G_{A/\mathfrak{p}}$ be a split torus defined over $A/\mathfrak{p}$ and let $t\in T(A/\mathfrak{p})$ be the element corresponding to $(1,\zeta_\mathfrak{p},\ldots,\zeta_\mathfrak{p} ^r)$. For each $\alpha \in \Phi^+$, choose a non-trivial element $u_\alpha$ in the root subgroup of $\alpha$ and, for each $i=1,\ldots,r$, let $t_i=\beta_i(2)\in T(A/\mathfrak{p})$. Then, the following hold:
\begin{enumerate}[(1)] 
\item \label{cond:split.def.1} $t^{r+1}=1$ and $\Cent_{G(A/\mathfrak{p})}(t)$ is abelian.
\item \label{cond:split.def.2} $t_i^{-1} u_\alpha t_i=u_\alpha^{2^{\alpha(\beta_i)}}$.
\item \label{cond:split.def.3} $u_\alpha^{C^n!} \neq 1$.
\end{enumerate} 
Now assume that $\mathfrak{p} \lhd A$ is such that the characteristic of $A/\mathfrak{p}$ is greater than $\max \left\{ r+1, n,C^n!,D, 2^{\alpha(\beta_i)} \mid \alpha \in \Phi, i=1,\ldots,r \right\}$ and there are elements $t,t_i,u_\alpha \in \Gamma / \Gamma^*[\mathfrak{p}]$, for $i=1,\ldots,r$ and $\alpha \in \Phi^+$ satisfying the conditions \ref{cond:split.def.1},\ref{cond:split.def.2},\ref{cond:split.def.3}. By Condition \ref{cond:split.def.1}, $t$ is regular and semisimple, so $S=\Cent_{G_{A/\mathfrak{p}}}(t)$ is a torus defined over $A/\mathfrak{p}$. By Conditions \ref{cond:split.def.2}, \ref{cond:split.def.3}, and Lemma \ref{lem:silly}, the elements $u_\alpha^{C^n!}$ are non-trivial unipotents. Every element of $S$ acts on the line $\overline{A/\mathfrak{p}} \cdot \log(u_\alpha)$ by scalar multiplication, so we get a map $f:S \rightarrow \mathbb{G}_m^{| \Phi |}$. Finally, Condition \ref{cond:split.def.2} implies that $f$ is an embedding. Hence, $S$ is split. Letting $\mathcal{F}_4$ be the collection of all subgroups $\Delta \in \mathcal{F}_3$ for which there are elements in $\Gamma / \Delta$ satisfying Conditions \ref{cond:split.def.1}, \ref{cond:split.def.2}, and \ref{cond:split.def.3}, we get the claim of the Corollary.
\end{proof} 

\begin{proof}[Proof of Proposition \ref{prop:good_subgroup} for non-uniform $\Gamma$] We will show that there is a homomorphism $\rho : \SL_2 \rightarrow G$ which is an isogeny over its image and such that $\rho (\SL_2) \cap \Gamma$ is definable. Since $\SL_2$ is isomorphic to the spin group of the form $x^2+y^2-z^2$, this will prove the claim.

Choose $u\in \Gamma$ unipotent. We have that $u^A \cap \Gamma$ is a subgroup of finite index in $u^A$, so, after replacing $u$ by some integral power of itself, we can assume that $u^A \subset \Gamma$. Let $X=\log(u)\in \mathfrak{g}(K)$. By Jacobson--Morozov, there is $Y\in \mathfrak{g}(K)$ such that $(X,Y)$ is an $\mathfrak{sl}_2$-pair. There is a natural number $m$ such that $\exp(mAY) \subset \Gamma$. Let $v=\exp(mY)$. We have that $u^A,v^A \subset \Gamma$ and the Zariski closure of the subgroup generated by $u^A,v^A$, which we denote by $S$, is isogeneous to $\SL_2$. 

It remains to show that $S\cap \Gamma$ is definable. For any prime $\mathfrak{q}$ of $A$, let $S_\mathfrak{q}$ be the image of $S(A/\mathfrak{q})$ in $\Gamma / \Gamma ^*[\mathfrak{q}]$, and let $u_\mathfrak{q},v_\mathfrak{q}$ be the images of $u,v$ in $\Gamma / \Gamma ^*[\mathfrak{q}]$. Let $\mathcal{Q}$ be the set of primes given by Corollary \ref{cor:uniformly.definable.congruence}. For all but finitely many primes $\mathfrak{q}$, $S_\mathfrak{q}=u_{\mathfrak{q}}^{A/\mathfrak{q}} v_{\mathfrak{q}}^{A/\mathfrak{q}}u_{\mathfrak{q}}^{A/\mathfrak{q}} v_{\mathfrak{q}}^{A/\mathfrak{q}}$. Using this and Corollary \ref{cor:uniformly.definable.unipotent.subgroup}, the sequence $(S_\mathfrak{q})_{\mathfrak{q} \in \mathcal{Q}}$ is a definable sequence of subsets of $(\Gamma/\Gamma^*[\mathfrak{q}])_{\mathfrak{q} \in \mathcal{Q}}$. It follows that there is a first order formula $F$ such that $F(g)$ holds if and only if $g \Gamma ^*[\mathfrak{q}] \in S_\mathfrak{q}$, for every $\mathfrak{q} \in \mathcal{Q}$. If $g\in \Gamma \smallsetminus S$, then, for almost all primes $\mathfrak{p}$, the reduction of $g$ modulo $\mathfrak{p}$ is not in $S_\mathfrak{p}$. Thus, $F(g)$ holds if and only if $g\in S \cap \Gamma$.
\end{proof}

\begin{proof}[Proof of Proposition \ref{prop:good_subgroup} for $G=\Spin$] 
Choose a regular 3-dimensional subset $U$ of $K^n$ such that $i_q(U_w)\ge 1$. We view $ \Spin_{q\restriction_U}(K)$ as a subgroup of $\Spin_q(K)$. Denote $f=q\restriction_U$.  There is an isomorphism $\rho:\Spin_f(K)\rightarrow  \Spin_{q\restriction_U}(K)$.  Let $\Lambda$ be the subgroup of $\Gamma$ consisting of the elements which act on $U^\perp$ as $\pm 1$. Then $\rho(\Spin_f(K))\cap \Gamma$ is of finite index in $\Lambda$. The  proof of Lemma \ref{lemma:def_stab} shows that $\Lambda$ is definable. \end{proof}

\subsection{Proof of Proposition \ref{prop:newnew} }

In the first few lemmas, we will use the following setting:

\begin{setting}\label{nota:Noskov} 
\begin{enumerate}
	\item[]
	\item  $A$ is the ring of $S$-integers in a number field $K$ and $\mathcal{P}=\{\mathfrak{p}^k \mid \mathfrak{p}\lhd A \text{ is prime and } k \ge 1\}$. 
	\item $f$ is a quadratic from on $K^3$,  $\alpha \in \SO_f(A)$ is an infinite order semisimple  element, $\Delta$ is a subgroup of $\Cent_{\SO_f(A)}(\alpha)$ and, for every ideal $\mathfrak{q}$, $\Delta[\mathfrak{q}]:=\Delta\cap \SO_f(A;\mathfrak{q})$. 
	\item $L$ is the spliting field of the characteristic polynomial of $\alpha$, $T$ is the set of places of $L$ that lie above $S$, and $B$ is the ring of $T$-integers in $L$. 
	\item $\beta\mapsto \lambda_\beta$ is a non-trivial homomorphism  from $\Delta$ to $B^\times$ such that, for every $\beta$, $\lambda_\beta$ is an eigenvalue of $\beta$. It follows that for every $\beta \in \Delta$ the eigenvalues of $\beta$ are $\{\lambda_\beta,\lambda_\beta^{-1},1\}$. 
\end{enumerate}
\end{setting}

The following is Theorem 2.0 of \cite{Nos}:

\begin{theorem}[Noskov]\label{thm:Noskov}Let $B$ be a finitely generated integral domain. There exists a number $d$ such that, for every distinct elements $c_1,\ldots,c_d \in B$ and every $0 \ne a \in B$, the set $\{b \in B \mid (\forall 1 \le i \le d)\ b-c_i | a\}$ is finite. 
\end{theorem}

The following Lemma is clear. 

\begin{lemma}\label{lemma:ramified2} Under Setting \ref{nota:Noskov}, assume that $K=L$. For every non-zero $a \in A$ and $\lambda \in A^\times$ such that $\lambda-1$ does not divides $a$, there exist a prime ideal $\mathfrak{p} \lhd A$ and a natural number $m \ge 1$ such that $a \notin B \mathfrak{p}^m$ and $\lambda-1 \in B\mathfrak{p}^m$.	
\end{lemma}

\begin{lemma}\label{lemma:ramified} Under Setting \ref{nota:Noskov}, assume that $K \ne L$. Let $c  \in A$ be a non-zero element that belongs to every prime ideal of $A$ that is ramified in $L$. 
	For every non-zero $a \in A$ and $\lambda\in U$ such that $\lambda-1$ does not divide $ac$,  there exist a prime ideal $\mathfrak{p} \lhd A$ and a natural number $m \ge 1$ such that $a \notin B\mathfrak{p}^m$ and $\lambda-1 \in B\mathfrak{p}^m$.
\end{lemma}
\begin{proof}
Since $\lambda-1$ does not divide $ac$, there exist a prime ideal $\mathfrak{q}$ of $B$ and $m \ge 1$ such that $\lambda-1 \in \mathfrak{q}^m$ and $ac \notin \mathfrak{q}^m$. Denote $\mathfrak{p}:=\mathfrak{q} \cap A$. We divide the proof into four cases:
\begin{enumerate}
	\item Assume that $\mathfrak{p}$ is inert in $L$. Then $\mathfrak{p}^m=\mathfrak{q}^m\cap A$ and $B\mathfrak{p}^m=\mathfrak{q}^m$ so  $a\notin  B\mathfrak{p} ^m$ and  $\lambda-1\in B\mathfrak{p}^m$.
	\item  Assume that $\mathfrak{p}$ splits in $L$ and let $\sigma$ be the non-identity element of  $\mathrm{Gal}(L/K)$. Then  $\mathfrak{p}^m=\mathfrak{q}^m\cap A$ and $B\mathfrak{p}^m \cap A =\mathfrak{p}^m$ so $a \notin B\mathfrak{p}^m$. Since $\sigma(\lambda)=\lambda^{-1}$, $\lambda-1=\sigma(-\lambda^{-1}(\lambda-1))\in \sigma(\mathfrak{q})^m$. It follows that  $\lambda-1\in \mathfrak{q}^m\cap \sigma(\mathfrak{q})^m =(\mathfrak{q}\sigma(\mathfrak{q}))^m=B\mathfrak{p}^m$. 
	\item Assume that $\mathfrak{p}$ ramifies in $L$ and $m=2l$.  Then $\mathfrak{p}^l=\mathfrak{q}^m \cap A $ and $B\mathfrak{p}^l=\mathfrak{q}^m$ so  $a\notin  B\mathfrak{p} ^l$ and  $\lambda-1\in B\mathfrak{p}^l$.
	\item Assume that $\mathfrak{p}$ ramifies in $L$ and $m=2l+1$. Then  $\lambda-1 \in B\mathfrak{p}^l=\mathfrak{q}^{2l}$. Since $ \mathfrak{p}^{l+1}=\mathfrak{q}^m \cap A $, $ac \notin  \mathfrak{p}^{l+1}$. Since $c \in \mathfrak{p}$ and $B\mathfrak{p}^l\cap A=\mathfrak{p}^l$, $a \notin B\mathfrak{p}^l$.  \end{enumerate}
\end{proof}

\begin{lemma}\label{lemma:not_diag} Under Setting \ref{nota:Noskov}, let $\mathfrak{p} \lhd A$ be a prime ideal and define $n\ge 0$ to be minimal such that  $\lambda^2_\alpha \not\equiv_{B\mathfrak{p}^{n+1}}1$. If $\beta \in \Delta$ and, for some $m \ge 2n+1$,  $\lambda_\beta \equiv_{B\mathfrak{p}^{m}}1$, then ${\beta}\in \Delta[\mathfrak{p}^{m-2n}]$.
\end{lemma}
\begin{proof} Since all the eigenvalues of $\alpha$ belong to $B$, it follows from a variant of the structure theorem of finitely generated modules over principal ideal domains (see \cite[Lemma 3.2]{Cas}) that there exists $\gamma \in \SL_3(B_{\mathfrak{p}})$ such that $\gamma \alpha \gamma^{-1}$ is an upper triangular matrix.  Since $\Delta$ is abelian and all the eigenvalues of $\alpha$ are distinct,   $\gamma \Delta \gamma^{-1}$ consists of upper triangular matrices. We can assume that $(\gamma\alpha\gamma^{-1})_{1,1}=\lambda_\alpha$, $(\gamma\alpha\gamma^{-1})_{2,2}=\lambda^{-1}_\alpha$ and $(\gamma\alpha\gamma^{-1})_{3,3}=1$.

Let $\beta \in \Delta$ and $m \ge 2n+1$ be such that $\lambda_\beta \equiv_{B\mathfrak{p}^{m}}1$. For every $1 \le i < j \le 3$, let $b_{i,j}$ be the $(i,j)$-entry of $\gamma \beta \gamma^{-1}$. Since $\alpha$ and $\beta$ commute, $b_{1,2} \equiv_{B\mathfrak{p}^m}\lambda_\alpha^2b_{1,2}$ and $b_{2,3} \equiv_{B\mathfrak{p}^m}\lambda_\alpha^{-1} b_{2,3}$. Since $\lambda_\alpha^2 \not\equiv_{\mathfrak{p}^{n+1}B} 1$, $b_{1,2},b_{2,3} \in B\mathfrak{p}^{m-n}$.  By the same argument, $b_{1,3} \equiv_{B\mathfrak{p}^{m-n}}\lambda_{\alpha}b_{i,j}$.  Since $\lambda_\alpha^2 \not\equiv_{\mathfrak{p}^{n+1}B} 1$, $b_{1,3} \in B\mathfrak{p}^{m-2n}$. 
	Thus, $\gamma\beta\gamma^{-1} \in \SL_3(B;B\mathfrak{p}^{m-2n})$ and $\beta \in \Delta[\mathfrak{p}^{m-2n}]$. 
\end{proof}

\begin{lemma}\label{lemma:main_SL_2} Under Setting \ref{nota:Noskov}, let $d$ be as in Theorem \ref{thm:Noskov}. Let $ \mathcal{Q} $ be a cofinite subset of $\mathcal{P}$.  Then $$
Q:=\{\beta \in \Delta \mid   (\forall \mathfrak{q} \in \mathcal{Q}\ \forall 1 \le i\le d)\ \beta\alpha^{-i} \notin \Delta[\mathfrak{q}]  \}
$$		
\end{lemma}
is finite.
\begin{proof}  
For every prime $\mathfrak{p}\lhd A$, let $m_\mathfrak{p}\ge 0$ be minimal such that $\mathfrak{p}^{m_{\mathfrak{p}}+k}\in \mathcal{Q}$, for every $k \geq 1$. For every prime $\mathfrak{p} \lhd A$, let $n_\mathfrak{p} \ge 0$ be minimal such that  $\lambda^2_\alpha \not\equiv_{B\mathfrak{p}^{n_{\mathfrak{p}}+1}}1$. There exists a finite set $P$ of prime ideals of $A$ such that for every  $\mathfrak{p}\notin P$, $m_{\mathfrak{p}}=n_{\mathfrak{p}}=0$. Choose a non-zero element $a \in \prod_{\mathfrak{p}\in P}\mathfrak{p}^{m_{\mathfrak{p}}+2n_{\mathfrak{p}}}$. 

We claim that if $\beta \in \Delta$ and there are a prime ideal $\mathfrak{p}$, an integer $k \ge 1$, and an integer $1 \le i \le d$ such that $a \notin \mathfrak{p}^k$ and $\lambda_{\beta\alpha^{-i}}\equiv_{B\mathfrak{p}^{k}}1$, then $\beta\notin Q$.  Indeed, since $a \in \prod_{\mathfrak{p}\in P}\mathfrak{p}^{m_{\mathfrak{p}}+2n_{\mathfrak{p}}}$,  $k \ge m_{\mathfrak{p}}+ 2n_{\mathfrak{p}}+1$. Lemma  \ref{lemma:not_diag} implies that  $\beta\alpha^{-i}\in\Delta[\mathfrak{p}^{k-2\mathfrak{n_\mathfrak{p}}}]$. Since $ \mathfrak{p}^{k-2\mathfrak{n_\mathfrak{p}}}\in  \mathcal{Q}$, $\beta\notin Q$.  
	
If $K=L$, define $c=1$ and, if $K \ne L$, let $c$ be as in Lemma \ref{lemma:ramified}. We will show that $$Q \subseteq \{\beta\in \Delta \mid \forall 1 \le i \le d,\ \lambda_\beta-\lambda_{\alpha^i}|ac\},$$ so Theorem \ref{thm:Noskov}  implies that  $Q$ is finite. Indeed, let $\beta \in \Delta$ and $1 \le i \le d$ be such that $\lambda_\beta-\lambda_{\alpha^i}=\lambda_{\alpha^i}(\lambda_{\beta\alpha^{-i}}-1)$ does not divide $ac$. Lemmas \ref{lemma:ramified2} and \ref{lemma:ramified} imply that there exist a prime ideal $\mathfrak{p}\lhd A$ and $k \ge 1$ such that $\lambda_{\beta\alpha^{-i}}\equiv_{B\mathfrak{p}^k}1$ and $a \notin \mathfrak{p}^k$. The second paragraph implies that $\beta\notin Q$.
\end{proof}

\begin{lemma}\label{lemma:isog} Let $K$ be a number field, $S$ a finite set of places containing all archimedean ones, $A$  the ring of $S$-integers in $K$ and $\mathcal{P}=\{\mathfrak{p}^k \mid \mathfrak{p}\lhd A \text{ is prime and } k \ge 1\}$.
Let $G_1,G_2 \subset (\SL_n)_A$ be group schemes such that $(G_1)_K$ and $(G_2)_K$ are semisimple, and let $\rho : (G_1)_K\rightarrow (G_2)_K$ be an isogeny (of algebraic groups over $K$). Then: \begin{enumerate} 
\item\label{item:isog_1} $\rho ^{-1} (G_2(A))$ is commensurable with $G_1(A)$.
\item \label{item:isog_2} For almost all prime ideals $\mathfrak{p} \lhd A$ and for every $k \geq 1$, $G_2(A_\mathfrak{p};\mathfrak{p} ^k) \subseteq \rho (G_1(A_\mathfrak{p} ; \mathfrak{p} ^k))$.
\item\label{item:isog_3} For every prime ideal $\mathfrak{p} \lhd A$ there exists $m_\mathfrak{p}\ge 0$ such that for every $k \ge 1$, $G_2(A_\mathfrak{p};\mathfrak{p}^{m_\mathfrak{p}+k}) \subseteq \rho(G_1(A_\mathfrak{p},\mathfrak{p}^k))$. 
\end{enumerate} 
\end{lemma} 

\begin{proof} \begin{enumerate} 
\item\label{item:nnnn} Let $\tau : (\Mat_n)_A \rightarrow (\Mat_n)_A$ be the map $\tau(X)=X+I$. The map $\tau ^{-1} \circ \rho \circ \tau: \tau ^{-1} (G_1) \rightarrow \tau ^{-1} (G_2)$ is given by polynomials with coefficients in $K$. Since $\tau ^{-1} \circ \rho \circ \tau (0)=0$, all these polynomials vanish at $0$. Let $N\in A$ be the product of all denominators of all coefficients of all polynomials in $\tau ^{-1} \circ \rho \circ \tau$. For any ideal $\mathfrak{q} \triangleleft A$, we have $\tau ^{-1} \circ \rho \circ \tau (N \mathfrak{q} A^{n^2}) \subseteq q A^{n^2}$, which implies that $\rho (G_1(A;N \mathfrak{q})) \subseteq G_2(A;\mathfrak{q})$. In particular, taking $\mathfrak{q} =A$, we get that $\rho (G_1(A;N)) \subseteq G_2(A)$. Since $\rho ^{-1} (G_2(A))$ is discrete and contains a lattice, it is a lattice. Hence, $[\rho ^{-1} (G_2(A)):G_1(A;N)]< \infty$. Since $[G_1(A):G_1(A;N)]<\infty$, they are commensurable.

\item\label{item:nnnn2}  Let $\mathfrak{g}_1$ (respectively, $\mathfrak{g}_2$) be the Lie ring of $G_1$ (respectively, $G_2$). We show that the claim holds for all prime ideals $\mathfrak{p} \lhd A$ for which both $G_1$ and $G_2$ have good reductions modulo $\mathfrak{p}$ and the map $d \rho |_{1}:\mathfrak{g}_1(A_\mathfrak{p}) \rightarrow \mathfrak{g}_2(A_\mathfrak{p})$ is an isomorphism. Let $k \geq 1$ and let $g\in G_1(A_\mathfrak{p})$ be such that $\rho(g)\in G_2(A_\mathfrak{p};\mathfrak{p} ^k)$. Since, by assumption, $d \rho |_{g}$ is an isomorphism, Hensel's lemma implies that there is $h\in G_1(A_\mathfrak{p};\mathfrak{p} ^k)$ such that $\rho(gh)=1$, so $gh \in \ker \rho$ and the claim follows.

\item\label{item:nnnn3} There is a natural number $a$ such that the power series $\log(x)$ and $\exp(x)$ converge on $G_1(A_\mathfrak{p},\mathfrak{p} ^a)$ and $\mathfrak{g}_1(K_\mathfrak{p}) \cap \mathfrak{p} ^a \mathfrak{sl}_n(A_\mathfrak{p})$ and define inverse bijections between the two sets. Fix $\varpi\in \mathfrak{p} A_\mathfrak{p} \smallsetminus \mathfrak{p}^2A_\mathfrak{p}$ and, for each natural number $t$, let $\delta_t:G_1(A_\mathfrak{p};\mathfrak{p} ^a) \rightarrow G_1(A_\mathfrak{p};\mathfrak{p} ^a)$ be the dilation map $\delta_t(g)=\exp \left( \varpi ^t \log(g) \right)$. For any $k \geq a$, we have that $\delta_t(G_1(A_\mathfrak{p};\mathfrak{p} ^k))=G_1(A_\mathfrak{p};\mathfrak{p} ^{k+t})$. There is a natural number $b$ such that $\varphi:=\rho \circ \delta_b$ is equal to a convergent power series with coefficients in $A_\mathfrak{p}$. Let $c$ be the constant obtained by applying Lemma \ref{lem:open.mapping.quantitative} with $R=A_\mathfrak{p}$, $X=G_1(A_\mathfrak{p};\mathfrak{p} ^a)$, and $Y=G_2(A_\mathfrak{p})$. Finally, since $d \varphi |_1= \varpi ^b d \rho|_1$, there is a natural number $d$ such that $d \varphi |_1\left(\mathfrak{g}_1(K_\mathfrak{p}) \cap \mathfrak{sl}_n(A_\mathfrak{p}) \right)  \supseteq \mathfrak{p} ^d \left(\mathfrak{g}_2(K_\mathfrak{p}) \cap \mathfrak{sl}_n(A_\mathfrak{p}) \right)$. By Lemma \ref{lem:open.mapping.quantitative}, for every $k \geq a+b$, 
\[
\rho \left( G_1(A_\mathfrak{p};\mathfrak{p} ^{k}) \right) = \varphi \left( G_1(A_\mathfrak{p};\mathfrak{p} ^{k-b}) \right) \supseteq G_2(A_\mathfrak{p};\mathfrak{p} ^{d+k-b+c} ),
\]
and the result follows.

\end{enumerate} 
\end{proof}

\begin{proof}[Proof of Proposition \ref{prop:newnew}] Denote $H:=\rho(\Spin_f)$. Since $H \subseteq G$, we have $\rho(\Spin_f) \cap G(A) = H(A)$. The only finite and non-trivial normal subgroup of $\Spin_f$ is the center $Z(\Spin_f)$ and this center has order two. We get that $\rho:\Spin_f \rightarrow H$ is either isomorphism  or $\ker \rho=Z(\Spin_f)$. In any case, we have an isogeny $\psi:H \rightarrow \SO_f$ of algebraic groups over $K$ and $\ker \psi $ is either trvial or central of over $2$. Lemma \ref{lemma:isog} implies that there exists a finite index normal subgroup $\Lambda^* \le \Lambda$ such that $\Lambda^* \le \Lambda \cap H(A)$ and  $\psi(\Lambda^*)$ is contained in $\SO_f(A)$. By Lemma \ref{lemma:isog}, for every prime ideal $\mathfrak{p} \lhd A$, there exists $m_{\mathfrak{p}} \ge 0 $ such that for every $k \ge 1$, $\psi(H(A_\mathfrak{p},\mathfrak{p}^{k}))$ contains $\SO_f(A_\mathfrak{p},\mathfrak{p}^{k+m_\mathfrak{p}})$. We can further assume that $m_\mathfrak{p}=0$ for all but finitely many prime ideals. 

Let $\beta \in \Lambda^*$. We claim that, for every prime ideal $\mathfrak{p}$ and every $k \ge 1$, if $\beta^2 \notin \Gamma^*[\mathfrak{p}^{k}]$, then $\psi(\beta)\notin \SO_f(A;\mathfrak{p}^{k+m_\mathfrak{p}})$. Assume otherwise. Lemma \ref{lemma:isog} implies that $\beta \in H(A) \cap ((\ker \rho)\cdot H(A_\mathfrak{p},\mathfrak{p}^k))$ so $\beta^2 \in H(A) \cap H(A_\mathfrak{p},\mathfrak{p}^k)=H(A;\mathfrak{p}^k) \subseteq  \Gamma^*[\mathfrak{p}^{k}]$, a contradiction.

Let $\alpha \in \Lambda^*$ be an infinite order semisimple element. Then $\Cent_{\Lambda^*}(\alpha)$ is an abelian subgroup whose torsion subgroup is finite. Let $d$ be the constant given in  \ref{lemma:main_SL_2} with respect to $\psi(\alpha)$ and $\Delta:=\psi(\Cent_{\Lambda^*}(\alpha))$. Let $\mathcal{R} \subseteq \mathcal{P}$ be a cofinite subset and denote $ \mathcal{Q}:=\{\mathfrak{p}^{k+\mathfrak{m}_p} \mid \text{$\mathfrak{p}$ prime and } \mathfrak{p}^k \in \mathcal{R}\}$. Note that $\mathcal{Q}$ is cofinite in $\mathcal{P}$. 
We claim that 
$$
D:=	\{\gamma \in \Cent_{\Lambda^*}(\alpha) \mid (\forall 1 \le i \le d\ \forall \mathfrak{r} \in \mathcal{R})\ (\gamma\alpha^{-i})^2\notin  \Gamma^*[\mathfrak{r}]\} $$ is finite. Let $\gamma \in D$. The previous paragraph implies that for every $\mathfrak{q} \in \mathcal{Q}$ and every $1 \le i\le d$, $\psi(\gamma\alpha^{-i})\notin \SO_f(A;\mathfrak{q})$. Since $\mathcal{Q}$ is cofinite in $\mathcal{P}$ and $\ker\psi$ is finite, Lemma \ref{lemma:main_SL_2} implies that $D$ is finite. 

Denote $e=[\Lambda:\Lambda^*]$. Let $\alpha \in \Lambda$ be an infinite order semisimple element. Then $\alpha^e \in \Lambda^*$ is an infinite order semisimple element. Let $d$ be the constant given in Lemma \ref{lemma:main_SL_2} with respect to $\psi(\alpha^e)$. In order to finish the proof it suffices to show that the set
 $$
E:=	\{\gamma \in Z(\Cent_{\Lambda}(\alpha)) \mid (\forall 1 \le i \le d\ \forall \mathfrak{r} \in \mathcal{R})\ (\gamma\alpha^{-i})^{2e}\notin \Gamma^*[\mathfrak{r}]\}
	 $$
is finite. If $\gamma \in Z(\Cent_{\Lambda}(\alpha))$ then $\gamma^e \in \Cent_{\Lambda^*}(\alpha^e)$ so the previous paragraph implies that $\{\gamma^e \mid \gamma \in E\}$ is finite.  Since $Z(\Cent_{\Lambda}(\alpha))$ is an abelian group whose torsion subgroup is finite, then map $x \mapsto x^e$ has finite fibers so $E$ is finite. 
\end{proof}

\subsection{Proof of Theorem \ref{thm:interpretation}}

\begin{proof}[Proof of Theorem \ref{thm:interpretation}]
Let $f$ and $\rho$ and $\Lambda$ be as in Proposition \ref{prop:good_subgroup}.
Fix an infinite order semisimple   $\alpha \in \Lambda$ and denote $\Theta:=Z(\Cent_\Lambda(\alpha))$.   Robinson \cite{Rob1} proved that $(\Z,+,\times)$ is definable in $(\Z,+,|)$ where $|$ is  the divisibility relation. For every non-zero $r,s \in \Z$, $r|s$ if and only if $\alpha^s \in \langle \alpha^r\rangle$. Thus, in order to prove Theorem \ref{thm:interpretation},  it is enough to show that there exists a definable subset $C \subseteq \Theta \times \Theta$ such that for every $\beta \in \Theta$ of infinite order, $C_\beta=\langle\beta\rangle$.  

Theorem  \ref{thm:principal.definable} implies that there exists a uniform definable collection  $\mathcal{F}$ of normal congruence subgroups in $\Gamma$ which contains $\{\Gamma^*[\mathfrak{q}] \mid \mathfrak{q} \lhd A\}$. Let $d,e \ge 1$ be as in Proposition \ref{prop:newnew}. Denote $\Psi:=\{\beta \in \Theta \mid \exists (1 \le i \le d)\ (\beta\alpha^{-i})^e=1\}$. Since the torsion subgroup of $\Theta$ is finite, $\Psi$ is a finite. Let $D \subseteq \Gamma \times \Theta$ be the definable subset 

$$D:=\{(\gamma,\beta)\in\Gamma \times \Theta \mid (\forall \Delta \in \mathcal{F}\ \forall 1 \le i\le d)\ \gamma \notin \Delta \rightarrow \left((\beta \in \Psi) \lor (\beta\alpha^{-i})^{e} \notin \Delta\right)  \}.$$

Then for every non-identity $\gamma \in \Gamma$, $D_\gamma$ is finite. 

\begin{claim}
	 Let $\Phi  \subseteq \Theta$ be finite. There exists a non-identity $\gamma \in \Gamma$ such that $\Phi \subseteq D_\gamma$.
\end{claim}
\begin{proof}
	There exists a finite set $
	\mathcal{C} \subseteq \mathcal{F}$ such there for every $\Delta \in \mathcal{F}\setminus \mathcal{C}$, $\Delta \cap \{(\phi\alpha^{-i})^{e} \mid \phi \in \Phi \text { and }1 \le i \le d\}=1$. For every non-trivial $\gamma \in \cap_{\Delta \in \mathcal{C}}\Delta$, $\Phi \subseteq D_\gamma$. 
\end{proof}

  Let $E \subseteq \Gamma^3 \times \Theta$ be the definable subset  $$\{((\gamma,\delta_1,\delta_2),\beta)\in \Gamma^3 \times \Theta \mid \left(\beta \in D_\gamma\right) \land \left((\forall \Delta\in \mathcal{F})\ \delta_2 \notin \Delta\rightarrow \ \beta\notin \delta_1\Delta\right)\}.$$

\begin{claim}\label{lemma:finite_def} Let $\Phi \subseteq \Theta$ be finite. There exist non-identity  $\gamma,\delta_1,\delta_2 \in \Gamma$ such that $\Phi=E_{(\gamma,\delta_1,\delta_2)}$.	
	\end{claim}
\begin{proof}
Choose non-identity $\gamma \in \Gamma$ such that $\Phi \subseteq D_\gamma$. Assume that $D_\gamma=\{\beta_i \mid 1\le i\le r+s\}$ and $\Phi=\{\beta_i \mid 1 \le i \le r\}$. Choose distinct prime ideals $\mathfrak{p}_1, \ldots, \mathfrak{p}_s$  such that for every $1 \le j \le s$, $\Gamma/\Gamma^*[\mathfrak{p}_j]$ is non-abelian and simple and the map $\beta_i \mapsto \beta_i\Gamma^*[\mathfrak{p}_j]$ is injective on $D_\gamma$. 

 By the strong approximation theorem, $\Gamma$ projects onto $\prod_{1 \le j \le s}\Gamma/\Gamma^*[\mathfrak{p}_j]$. Hence, 
 there exists $ \delta_1 \in \Gamma$ such that for every $1 \le j \le s$, $\beta_{r+j}\Gamma^*[\mathfrak{p}_j]=\delta_1\Gamma^*[\mathfrak{p}_j]$. Let $\mathcal{C}\subseteq \mathcal{F}$ consists of the $\Delta \in \mathcal{F}$ for which there exits  $1 \le i \le r$ such that $\delta_1\Delta=\beta_i\Delta$. Then $\mathcal{C}$ is finite and every $\Delta \in \mathcal{C}$ is not contained in any of the subgroups  $\Gamma^*[{\mathfrak{p}_1}],\ldots, \Gamma^*[{\mathfrak{p}_s}]$. Since  for every $1 \le j \le s$,  $\Gamma/\Gamma^*[\mathfrak{p}_j]$ is simple and every $\Delta \in \mathcal{C}$ is normal in $\Gamma$, $\cap_{\Delta \in \mathcal{C}}\Delta$ projects onto $\prod_{1 \le j \le s} \Gamma/\Gamma^*[\mathfrak{p}_j]$. Thus, there exists $\delta_2 \in \cap_{\Delta \in \mathcal{C}}\Delta$ such that for every $1 \le j \le s$, $\delta_2 \notin \Gamma^*[\mathfrak{p}_j]$. It follows that $\Phi=E_{(\gamma,\delta_1,\delta_2)}$.	
\end{proof}

Claim \ref{lemma:finite_def} implies that the collection $\mathcal{E}$ of finite subsets of $\Gamma$ is uniformly definable. Let $C$ be the subset of $\Theta \times \Theta$ such that for every $\beta,\gamma \in \Theta$, $(\beta,\gamma)\in C$ if and only if there exists $\Phi \in \mathcal{E}$ for which the following two conditions hold:
\begin{enumerate}[a)]
	\item\label{item:rob1} $\beta \in \Phi$.
	\item\label{item:rob2} If $\delta \in \Phi$ then either $\delta=\gamma^{\pm 1}$ or $\beta\delta\in \Phi $. 
\end{enumerate}

We claim that  $C$ is the desired definable subset. Let $\beta \in \Theta$ be of infinite order. For every $r \in \N$, the set $\Phi:=\{\beta^i \mid 0 \le |i| \le |r|\} $ satisfies items \ref{item:rob1} and \ref{item:rob2}  so $(\beta,\beta^r)\in C$. On the other hand, if  $\gamma \notin \langle \beta\rangle$ then every set which satisfies items \ref{item:rob1} and \ref{item:rob2} contains all positive powers of $\beta$ and thus is not finite. The proof of Theorem \ref{thm:interpretation} is now complete. 
\end{proof}

\section{Bi-interpretation}\label{sec:bi}

The goal of this section is to prove Theorem \ref{thm:main.biint}. The following Lemma follows from Corollary 2.8  of \cite{AKNS}:

\begin{lemma}\label{lemma:self _interpretation} Every self interpretation of $\mathbb{Z}$ is trivial.
\end{lemma}

Robinson \cite{Rob2} proved that $\Z$ is a definable subset in the ring of integers of any number field. Using the fact that every such ring is a free $\Z$-module, the following lemma can be easily proved for rings of integers. A similar argument works for rings of $S$-integers.  Alternatively, it follows from the main Theorem of  \cite{AKNS} that every finitely generated infinite integral domain is bi-interpretable with $\Z$.

\begin{lemma}\label{lemma:ring_int_bi} Every ring of $S$-integers of a number field is in bi-interpretation with $\mathbb{Z}$.
\end{lemma}

The following is a well known theorem of G\"{o}del:

\begin{theorem}\label{thm:Godel} Every recursive function $\N \rightarrow \Z^m$ and every  recursively enumerable subset $B \subseteq \Z^m$ are definable in $\Z$. 
\end{theorem}

\begin{lemma}\label{lemma:Godel2}
 Let $A$, $\Gr$ and $\Gamma$ be as in Setting \ref{nota:Gamma}. If $G=\Spin_q$, assume further that $n \geq 9$. There exist an interpretation  $\mathscr{R}=(R,r)$ of $\mathbb{Z}$ in $\Gamma$, an interpretation $\mathscr{E}=(E,e)$  of $\Gamma $ in $\Z$ and an infinite order element $\alpha \in \Gamma$ such that for  $\mathscr{H}=(H,h):=\mathscr{E}  \circ \mathscr{R}$ the restriction of $h^{-1}$ to $\langle \alpha\rangle$ is definable. 
\end{lemma}
\begin{proof}
Theorem \ref{thm:interpretation} implies that there is  an infinite order element element $\alpha \in \Gamma$ 
 and an interpretation $\mathscr{R}=(R,r)$ of $\mathbb{Z}$ in $\Gamma$ such that $R=\langle \alpha \rangle$ and $r(\alpha^m)=m$. 
 
 Lemma \ref{lemma:ring_int_bi} gives an interpretation $\mathscr{B}=(B,b)$ of $A$ in $\Z$. Let $\mathscr{C}:=(C,c)$ be the standard interpretation of $\Gr(A)$ in $A$. Then $\mathscr{D}=(D,d):=\mathscr{C} \circ \mathscr{B}$ is an interpretation of $\Gr(A)$ in $\Z$.  Since $\Gamma$ is finitely generated, $E:=d^{-1}(\Gamma)$ is  recursively enumerable subset of $D$. Denote $e:=d\restriction_E$. Theorem \ref{thm:Godel} implies that:
 \begin{enumerate}[a)]
 	\item  $E$ is definable so $\mathscr{E}=(E,e)$ is an interpretation of $\Gamma$ in $\Z$.
 	\item\label{item:power_def} The map $p:\Z \rightarrow E$ given by $p(m) = e^{-1}(\alpha^m)$ is definable in $\Z$. 
 \end{enumerate}
 Item \ref{item:power_def} implies that the map $\mathscr{R}^*p:R \rightarrow \mathscr{R}^*E$ is definable in $\Gamma$. Denote $\mathscr{H}=(H,h):=\mathscr{E} \circ \mathscr{R}$. Then $H=\mathscr{R}^*E$  and the restriction of $h^{-1}$ to $\alpha$ is $\mathscr{R}^*p$.
\end{proof}

 The following is Theorem 2 of \cite{PySz}. It can also be deduced from Theorem 2.3 of \cite{BGT} under the assumption that $X$ is symmetric.

\begin{theorem}[\cite{PySz}]\label{thm:BGT}
	Let $G$ be a finite simple group of Lie type of rank $r$ and $X$ a generating set of L. Then either $X^3 =G$ or
$|X^3| > |A|^{1+\epsilon}$ where $\epsilon$ depends only on $r$.
\end{theorem} 

\begin{lemma}\label{lemma:pro_boinded_gen} Let $\Gamma$ be as in Setting \ref{nota:Gamma} and let $\alpha \in \Gamma$ be of infinite order. There exist infinitely many prime ideals  $\mathfrak{p} \lhd A$ such that the order of $\alpha$ in $\Gamma/\Gamma[\mathfrak{p}]$ is at least $|A/\mathfrak{p}|^{\frac{1}{3[K:\Q]}}$.
\end{lemma}
\begin{proof} 
	Assume first that $\alpha$ is virtually-unipotent, Then there exists $m>0$ such that  $\alpha^m$ is unipotent. If $\alpha^m \notin  \Gamma[\mathfrak{p}]$ then the order of the image of $\alpha^m$ in $\Gamma/\Gamma[\mathfrak{p}]$ is at least $p=\mathrm{char}(A/\mathfrak{p})$. The claim follows since $p^{[K:\Q]} \ge |A/\mathfrak{p}|$. 

	For every rational prime $p$ there exists a prime ideal of $\mathfrak{p} \lhd A$ for which  $ \mathrm{char}(A/\mathfrak{p})=p$. Moreover, if $\mathfrak{p} \lhd A$  and $ \mathrm{char}(A/\mathfrak{p})=p$ then $|A/\mathfrak{p}| \le p^{[K:\Q]}$. It follows from the prime number theorem that for a large enough $m$, the number of prime ideals  $\mathfrak{p} \lhd A$ for which $|A/\mathfrak{p}| \le m^{3[K:\Q]}$  is at least $\frac{m^{3}}{6\ln m}$. Therefore, in order to prove the lemma it is enough to show that if $\alpha$ is not virtually-unipotent, then the number of prime ideals $\mathfrak{p} \lhd A$ for which the image of $\alpha$  in  $\Gamma/\Gamma[\mathfrak{p}]$ has order at most $m$, is bounded by a quadratic function of $m$. In order to show this it is enough to show that the number of prime ideals $\mathfrak{p} \lhd A$ for which the image of $\alpha$  in  $\Gamma/\Gamma[\mathfrak{p}]$ has order exactly $m$, is bounded by a linear function of $m$.

	Assume that  $\alpha$ is not virtually-unipotent. Then $\alpha$ has an eigenvalue $\lambda$  which is not a root of unity.
	Let $E$ be a finite extension of $K$ which contains $\lambda$ and let $B$ be the ring of integers of  $E$. Then $B$ contains $\lambda$.  If $\mathfrak{p} \lhd A$ is a prime ideal and the order of the image of $\alpha$ in $\Gamma/\Gamma[\mathfrak{p}]$ is $m$, then $\lambda^m -1 \in \mathfrak{p}B$. If follows that $|B/ \mathfrak{p}B|$ divides $|B/(\lambda^m -1)B|=N_{E/\Q}(\lambda^m-1)$. In particular,  $\mathrm{char}(A/\mathfrak{p})$ divides $N_{E/\Q}(\lambda^m-1)$. The number of distinct prime divisors of $N_{E/\Q}(\lambda^m-1)$ is at most $\log_2(N_{E/\Q}(\lambda^m-1))$ so it is bounded by a linear function in $m$. The result follows since for every prime $p$, there exists at most $[K:\Q]$ prime ideals $\mathfrak{p} \lhd A$ such that  $\mathrm{char}(A/\mathfrak{p})=p$.

\end{proof}

\begin{corollary}\label{cor:set_proj} Let $\Gamma$ be as in Setting \ref{nota:Gamma} and let $\alpha \in \Gamma$ be of infinite order.  There are $\beta_1,\ldots,\beta_d \in \Gamma$ such that the set $\prod_{1 \le i \le d}\langle \beta_{i}\alpha\beta_i^{-1}\rangle$ projects onto $\Gamma/\Gamma[\mathfrak{p}]$, for infinitely many prime ideals  $\mathfrak{p} \lhd A$.
\end{corollary}

\begin{proof} By Margulis's Normal Subgroup Theorem, $[\Gamma:\langle \gcl_\Gamma(\alpha)\rangle]<\infty$, so $\langle \gcl_\Gamma(\alpha)\rangle$ is generated by finitely many conjugates of $\alpha$. By the strong approximation theorem, for all but finitely many prime ideals $\mathfrak{p}$, the normal subgroup generated by $\alpha$  projects onto $\Gamma/\Gamma[\mathfrak{p}]$. The result follows from  Theorem \ref{thm:BGT}, Lemma \ref{lem:simple.quotients}, and Lemma \ref{lemma:pro_boinded_gen}.
\end{proof}

\begin{lemma}\label{lemma:fibres} Let $G$ be a group and let $L,M$ be normal subgroups of $G$. Let $\Phi \subseteq G$ be a symmetric generating subset which contains the identity and  projects onto $G/L$ and $G/M$. If the quotients maps $\Phi^2 \rightarrow  G/L$ and $\Phi^2 \rightarrow  G/M$ have the same fibers then $L=M$.

\end{lemma}
\begin{proof}
Assume that the fibers are the same. Define a map $\rho:G/L\rightarrow G/M$ by setting $\rho(\phi L):=\phi M$, for every $\phi \in \Phi$.  Every $g\in G$ is a product of elements in $\Phi\cup \Phi^{-1}$ and induction on the length of this product shows that  $\rho(gL)=gM$. It follows that $\rho$ is an isomorphism. In particular, $g \in L$ if and only if $gL \in \ker\rho$ if and only if $g \in M$.
\end{proof}

%\begin{theorem}\chen{Probably not needed} Let $A$ be the ring of $S$-integers in a number field $K$ and let $q$ be a quadratic form on $A^n$. The group $\Delta:=\Gamma$ is biinterpretable with $\mathbb{Z}$ if and only if there is an interpretation $\mathscr{A}=(\mathcal{A},a)$ of $A$ in $\Delta$ such that the composition of the inclusion $\Delta \hookrightarrow A^{n^2}/\pm \Id $ and $a^{-1}: A^{n^2}/\pm \Id \rightarrow \left( \mathcal{A}^{n^2}/\pm \Id \right) (\Delta)$ is $\Th_\Delta$-definable.
%\end{theorem}

 \begin{proof}[Proof of Theorem \ref{thm:main.biint}]  
 
Lemma \ref{lemma:Godel2} implies that there exist an interpretation  $\mathscr{R}=(R,r)$ of $\mathbb{Z}$ in $\Gamma$, an interpretation $\mathscr{E}=(E,e)$  of $\Gamma $ in $\Z$ and an infinite order element $\alpha \in \Gamma$ such that for  $\mathscr{H}=(H,h):=\mathscr{E}  \circ \mathscr{R}$ the restriction of $h^{-1}$ to $\langle \alpha\rangle$ is definable. Lemma \ref{lemma:self _interpretation} states that every self interpretation of $\Z$ is trivial. Therefore,  in order to show that $\Gamma$ is bi-interpretable with $\Z$, it is enough to prove that the isomorphism $h^{-1}: \Gamma \rightarrow H$ is definable. Since $h ^{-1}$ is a homomorphism, if $D_1,D_2$ are definable subsets of $\Gamma$ and the restrictions of $h^{-1}$ to each $D_i$ is definable, then the restriction of $h^{-1}$ to $D_1^{-1}$ and $D_1 D_2$ are definable. 

Margulis's Normal Subgroup Theorem implies that non-trivial normal subgroups of $\Gamma$ have finite index. Since finite index subgroups of $\Gamma$ are finitely generated, we can choose $\beta_1,\ldots,\beta_d \in \Gamma$ such that $\Lambda:=\langle \beta_i\alpha\beta_i^{-1} \mid 1 \le i \le d\rangle$ is a normal finite index subgroup of $\Gamma$. Corollary \ref{cor:set_proj} allows us to further assume that   $D_1:= \langle \beta_1\alpha\beta_1^{-1}\rangle\langle \beta_2\alpha\beta_2^{-1}\rangle\cdots \langle \beta_{d}\alpha\beta_{d}^{-1}\rangle$ projects onto $\Gamma/\Gamma[\mathfrak{p}]$ for infinitely many prime ideals $\mathfrak{p} \lhd A$. Choose a finite representative set $D_2$ for $\Gamma/\Lambda$ which contains the identity and denote $D:=D_1\cup D_1^{-1}\cup D_2 \cup D_2^{-1}$. Since  the restriction of $h^{-1}$ to $\langle \alpha\rangle$ is definable, the previous paragraph implies that the restriction of $h^{-1}$ to $D^2$ is also definable.

Every ideal of $A$ is generated by two elements. Therefore, there exist  definable sets $I$ and $X \subseteq I \times H$  such that $\mathcal{H}:=\{h^{-1}(\Gamma^*[\mathfrak{q}]) \mid \mathfrak{q} \lhd A\}=\{X_i \mid i \in I\}$. Theorem \ref{thm:principal.definable} implies that there exists a uniformly definable collection $\mathcal{F}$ of normal congruence subgroups of $\Gamma$ which contains $\{\Gamma^*[\mathfrak{q}] \mid \mathfrak{q} \lhd A\}$. 
Let $\tilde{J}$ and $Y \subseteq \tilde{J} \times \Gamma$ be definable sets  such that $\mathcal{F}=\{Y_j \mid j \in \tilde{J}\}$. For every $i \in I $ and $j \in \tilde{J}$, denote $H[i]=X_i$ and $\Gamma^*[j]=Y_j$. 

 Let $J \subseteq \tilde{J}$ be the definable subset such that $j \in J$ if and only if the following condition holds:
\begin{enumerate}[a)]
	\item\label{item:bi_int_0}  $D\Gamma^*[j]/\Gamma^*[j]=\Gamma/\Gamma^*[j]$. 	
	 \setcounter{my_counter}{\value{enumi}}
	\end{enumerate}
By the construction of $D$ we get that
	\begin{enumerate}[a)]
	\setcounter{enumi}{\value{my_counter}}
		\item\label{item:bi_int_1} The set $\{\Gamma^*[j] \mid j \in  J \}$ is infinite.
	 \setcounter{my_counter}{\value{enumi}}
	\end{enumerate}

	We claim that the set $W:=\{(i,j) \in I \times J \mid H[{i}]=h^{-1}(\Gamma^*[j])\}$ is definable.  We first show that if the claim is true then  $h^{-1}$ is definable. Since $\Gamma$ is centerless, two elements of $\Gamma$ are equal if and only if they are equal modulo infinitely many $\Gamma^*[\mathfrak{q}]$. Thus, for every $\gamma\in \Gamma$ and $\eta \in H$, $h^{-1}(\gamma)=\eta$ if and only if, for every $(i,j)\in W$, there exists $\delta \in D$ such that $\delta\Gamma^*[j]=\gamma\Gamma^*[j]$ and $h^{-1}(\delta)H[i]=\eta H[i]$. Since $h^{-1}\restriction_D$ is definable, the later statement can be expressed as a first order statement. 

It remains to show that  $W$ is definable.  Let $U \subseteq I \times J $ be the set such that $(i,j) \in U$ if and only if the following two conditions hold:
\begin{enumerate}[a)]
\setcounter{enumi}{\value{my_counter}}
	\item\label{item:bi_int_5.5} $h^{-1}(D)H[i]/H[i]=H/H[i]$. 
	\item\label{item:bi_int_6} For every $\delta_1,\delta_2 \in D^2$, $\delta_1\Gamma^*[j]=\delta_2\Gamma^*[j]$ if and only if   $h^{-1}(\delta_1)H[i]=h^{-1}(\delta_2)H[i]$. \end{enumerate}

Since $h^{-1}$ is definable on $D^2$, conditions \ref{item:bi_int_5.5} and \ref{item:bi_int_6} are first order conditions. Thus, $U$ is definable. Clearly, $W \subseteq U$ so in order to complete the proof it is enough to show that $U \subseteq W$. 
Item \ref{item:bi_int_6} implies that, for  every $(i,j) \in U$, $h^{-1}$ induces a bijection between the fibers of the reduction  map $D^2 \rightarrow \Gamma/\Gamma^*[j]$ and the fibers of the reduction map $h^{-1}(D^2)\rightarrow H/H[i]$. Moreover, items \ref{item:bi_int_0} and \ref{item:bi_int_5.5} imply that, for every $(i,j) \in U$, $D$ and  $h^{-1}(D)$ project onto $\Gamma/\Gamma^*[j]$ and $H/H[i]$, respectively. Thus,  Lemma \ref{lemma:fibres} implies that $U \subseteq W$.
\end{proof}

\section{Width of squares and bi-interpretation}\label{sec:width}

The following Lemma is well known, we include a proof for the convenient of the reader.

\begin{lemma}\label{lemma:Z_int_fp}
	The ring $\Z$ interprets every finitely presented group. 
\end{lemma}
\begin{proof}
	We first claim that for every $d \ge 1$,  the ring $\Z$ interprets a finitely generated free group of rank at least $d$. This could be proved directly using G\"{o}del's encoding or in the following way: For every $d$, there exists $\mathfrak{p}\lhd \Z$ such that  $\SL_2(\Z;\mathfrak{p})$ is a free group of rank at least $d$. Clearly, $\Z$ interprets  $\SL_2(\Z;\mathfrak{p})$.
	
	Let $\Gamma$ be a finitely presented group and let $\rho:F \rightarrow \Gamma$ be an epimorphism where $F$ is a free group of finite rank. Since $\Gamma$ is finitely presented,  $\ker \rho$ is a recursively enumerable subset of $F$. The result follows from Theorem \ref{thm:Godel}. 	 
\end{proof}

\begin{lemma}\label{lemma:with_to_int}
	Let $\Gamma$ be a finitely presented group which is bi-interpretable with the ring $\Z$. Let $k \ge 2$ and assume that the word $x^k[y,z]$ has finite width in $\Gamma$. If $\Delta$ is a finite central extension of $\Gamma$ by a group of size $k$ then $\Delta$ is bi-interpretable with $\Z$.
\end{lemma}
\begin{proof} Since bi-interpretability is an equivalence relation, it is enough to show that $\Gamma$ and $\Delta$ are bi-interpretable. Identifying $\Gamma$ with a quotient of $\Delta$ by a central subgroup of size $k$, we can view $\Gamma$ as an imaginary in $\Delta$.  Then $\mathscr{C}:=(\Gamma,\id_\Gamma)$ is an interpretation of $\Gamma$ in $\Delta$. Lemma \ref{lemma:Z_int_fp}  implies that there exists an interpretation $\mathscr{D}:=(D,d)$ of $\Delta$ in $\Gamma$. Since $\Gamma$ is an imaginary in $\Delta$, $D$ is also an imaginary in $\Delta$.

We want to show that $\mathscr{D}\circ \mathscr{C}=\mathscr{D}$ and  $\mathscr{C}\circ \mathscr{D}=[\mathscr{D}^*\Gamma,d_\Gamma]$ are trivial.  Since $\Gamma$ is bi-interpretable with $\Z$, every self interpretation, in particular $\mathscr{C}\circ \mathscr{D}$, is trivial. Thus, $d_\Gamma$ is definable in $\Gamma$. Since we view $\Gamma$, and thus also $\mathscr{D}^*\Gamma$, as imaginaries of $\Delta$, $d_\Gamma$ is also definable as a function between two imaginaries in $\Delta$. Let $\rho:\Delta \rightarrow \Gamma$ be the quotient map.  We have a commutative square: 

$$\begin{CD}
	D @>d>> \Delta\\
	@VV\mathscr{D}^*\rho V @VV\rho V\\
	\mathscr{D}^*\Gamma @>d_\Gamma>> \Gamma
\end{CD} 
$$
Since $\rho$ is definable in $\Delta$, $\mathscr{D}^*\rho$ is definable as a function between imaginaries of  $\Gamma$ and thus also  as a function between imaginaries of  $\Delta$. It follows that $\rho \circ d= \mathscr{D}^*\rho \circ d_\Gamma$ is definable in $\Delta$. 

Recall that the bijection  $d:D \rightarrow \Delta$ induces a group structure on $D$ and that the induced multiplication $D \times D \rightarrow D$ is definable in $\Delta$.   Denote $w=x^k[y,z]$. For every $\delta_1,\delta_2,\delta_3,\delta'_1,\delta'_2,\delta'_3 \in D$ satisfying $\rho \circ d(\delta_i)=\rho \circ d(\delta'_i)$, for $1 \leq i \leq 3$, we have
$$d\left(w(\delta_1,\delta_2,\delta_3)\right)=w(d(\delta_1),d(\delta_2),d(\delta_3))=w(d(\delta'_1),d(\delta'_2),d(\delta'_3))=d\left(w(\delta'_1,\delta'_2,\delta'_3)\right).$$ 
It follows that if $\alpha=w(\alpha_1,\alpha_2,\alpha_3) \in w(\Delta)$, $\delta=w(\delta_1,\delta_2,\delta_3) \in w(D)$ and, for every $1 \le i \le 3$,  $\rho\circ d(\delta_i)=\rho(\alpha_i)$, then 
 $d(\delta)=\alpha$. Thus, the restriction of $d^{-1}$ to ${w(\Delta)}$ is definable. Since $w$ has finite width in $\Gamma$ and, thus, in $\Delta$, the restriction of  $d^{-1}$ to  ${\langle w(\Delta)\rangle}$ is definable. Since $\Delta$ is finitely generated, $[\Delta:\langle w(\Delta) \rangle ]<\infty$ so $d^{-1}$ is definable. 
\end{proof}

%\begin{corollary}\label{cor:bi_chen}
%	Let $\Gamma$ be a finitely presented group which contains an infinite order element. Then $\Gamma$ is bi-interpretable with the ring $\Z$ if and only if every recursively  enumerable  subset of $\Gamma$ is definable. 
%\end{corollary}
%\begin{proof}
%	Lemma \ref{lemma:Z_int_fp} implies that there exists an interpretation $\mathscr{C}=(C,c)$ of $\Gamma$ in $\Z$ . Fix $\alpha \in \Gamma$ of infinite order. The set $\{(\alpha^r,\alpha^s) \mid r \text{ divides } s \}$ is recursively  enumerable and thus is definable. The first paragraph of the proof of Theorem \ref{thm:interpretation} implies that there is an interpretation $\mathscr{D}=(D,d)$ of $\Z$ in $\Gamma$. We have to show that $\mathscr{D}\circ \mathscr{C}$ and $(E,e):=\mathscr{C}\circ \mathscr{D}$ are trivial.  Corollary 2.8  of \cite{AKNS} implies that every self interpretation of $\Z$ is trivial so $\mathscr{D}\circ \mathscr{C}$ is trivial. It is clear that the graph of $e$ is  recursively enumerable, so $e$ is definable. 
%\end{proof}

\begin{definition}
Suppose that $L$ is a first order language and that $M$ is an $L$-structure. We denote by $\Aut_L(M)$ the group of automorphisms of $M$ as an $L$-structure. In particular, the elements of $\Aut_L(M)$ point-wise fix the constants.

For every $\varphi \in \Aut_L(M)$ and every imaginary $I$ in $M$, we also denote  by $\varphi_I$ the automorphism that $\varphi$ induces on $I$. Note that, if $F
 :I\rightarrow J$ is a definable function between imaginaries, then, for every $x \in I$, $F(\varphi_I(x))=\varphi_J({F}(x))$.

\end{definition}

\begin{definition}
	Suppose that $L$ is a first order language, $M$ is an $L$-structure and $I$ is a definable subset of $M$. Denote the $L$-theory of $M$ by $\Th_M$ and let $\phi$ be an $L_M$-formula such that $I=\phi(M)$. Then for every $\Th_M$ model $M'$, $I':=\phi(M')$ is a definable subset of $M'$ which is independent of the choice of the formula $\phi$. Thus, there is no ambiguity in denoting the set $I'$ by $I(M')$. 
	
	We use similar definition and notation in the case where $I$ is an imaginary. 
\end{definition}

\begin{lemma}\label{lemma:nec_def}
	Suppose that $L$ is a first order language,  $M$ is an $L$-structure and $I$ is an imaginary. Let $\Th_M$ be the $L$-theory of $M$.
	Then a necessary condition for the existence of a definable surjective function form $I$ onto $M$ is that for every $\Th_M$ model $M'$ and every automorphism $\varphi\in \Aut_{L}(M')$, if $\varphi_{I(M')}=\id_{I(M')}$ then $\varphi=\id$. 	
\end{lemma}
\begin{proof}
	Assume that $F:I \rightarrow M$ is a definable surjective function. Let $M'$ be a $\Th_M$ model  and $\varphi\in \Aut_{L}(M')$ be such that $\varphi_{I(M')}=\id_{I(M')}$. Then  for every $x \in I(M')$, $F(M')(x)= F(M')(\varphi_I(x))=\varphi({F(M')}(x))$. Since $F(M'):I(M')\rightarrow M'$ is surjective, $\varphi=\id$. 
\end{proof}

\begin{proof}[Proof of Theorem \ref{thm:width}]
	The if part is Lemma \ref{lemma:with_to_int}. For the only if part, assume that the word $w=x^d[y,z]$ has infinite width in $\Gamma$ and, thus, also in $\Delta$. View $\Gamma$ as the  quotient of $\Delta$ by  a central subgroup $\Lambda$ of size $d$. In particular, $\Gamma$ is an imaginary in $\Delta$. By Lemma \ref{lemma:Z_int_fp}, there exists an interpretation $\mathscr{C}=(C,c)$ of $\Delta$ in $\Gamma$. Since $\Gamma$ is an imaginary of $\Delta$, $\mathscr{C}$ is also an interpretation of $\Delta$ in itself. If $\Delta$ is bi-interpretable with $\Z$ then  Lemma \ref{lemma:self _interpretation} implies that $\mathscr{C}$  is trivial so $c:C \rightarrow \Delta$ is definable. Lemma \ref{lemma:nec_def} implies that, in order to show that $c$ is not definable, it is enough to show that there exists a $\Th_\Delta$-model $\Delta'$ and a non-identity automorphism $\varphi \in \Aut_{L_\Delta}(\Delta')$ such that $\varphi\restriction_{C(\Delta')}=\id_{C(\Delta')}$. Note that $\Aut_{L_\Delta}(\Delta')$ is the subgroup of $\Aut(\Delta')$ consisting of the group automorphisms  which fix every element of $\Delta$ and that, if $\varphi\restriction_{\Delta'/\Lambda}=\id_{\Delta'/\Lambda}$, then $\varphi\restriction_{C(\Delta')}=\id_{C(\Delta')}$.
	
	Let $U$ be a non-principal ultrafilter on $\N$ and denote $\Delta':=\prod_{n \in \N}\Delta/U$.  Identify $\Delta$ as a subgroup of $\Delta'$ via the diagonal embedding so $\Delta'$ is a $\Th_\Delta$-model. Since the width of $w$ in $\Delta$ is infinite, $[\Delta':\langle w(\Delta')\rangle]=\infty$ and $\Delta' / \langle w(\Delta') \rangle$ is an uncountable abelian group of exponent $d$. Since $\Delta$ is finitely generated, there exists a non-trivial homomorphism $\rho:\Delta'\rightarrow \Lambda$ such that $\rho\restriction_\Delta=\id_\Delta$. The automorphism  $\varphi \in \Aut_{L_\Delta}(\Delta')$ defined by $\varphi(x)=x\rho(x)$ is the desired automorphism. 	
\end{proof}

\section{Proof of Theorem \ref{thm:effective_kneser} and Lemma \ref{lemma:uff}}\label{sec:orbit} 
\begin{setting} \label{asum:open} $K$ is a number field, $S$ is a finite set of places containing all archimedean ones, $A$ is the ring of $S$-integers,  $\Theta_q=\Spin_q$, where $q$ is a regular integral quadratic form on $A^n$ and  $w\in S$ is a place such that $i_q(K_w^n) \ge 1$. Finally,  $\Gamma$ is a congruence subgroup of $G(A)$. %and $\tilde{\Gamma}$ is a  abstract group which contains $\Gamma$ as a finite index subgroup and does not have non-trivial normal abelian subgroups. 

For every subspace $C$ of  $K^n$, we view $\Theta_{q\restriction_C}(K):=\Spin_{q\restriction_C}(K)$ as a subgroup of $\Theta_q(K)$. 

For every place $v$, let  $K_v$ be the $v$-completion of $K$. 
	 %and let $|\cdot|_v:K_v\rightarrow \mathbb{R}$ be the $v$-adic valuation. 
	 For a subset $C \subseteq K^n$ let $C_v$ be its closure in $(K_v)^n$. In particular, $A_v=K_v$ for $v\in S$. For $v \not \in S$, let $k_v$ be the residue field of $A_v$ and $p_v:A_v \rightarrow k_v$ be the residue map. The kernel of $p_v$ is denote by $\mathfrak{p}_v$. 
\end{setting}

\begin{remark}
	In some places we assume the stronger condition $i_q(K_w^n) \ge 2$. In particular, in the proofs of Theorem \ref{thm:effective_kneser} and Lemma \ref{lemma:uff} we assume that  $i_q(K_w^n) \ge 2$. 
\end{remark}

\subsection{Proof of Theorem \ref{thm:effective_kneser}}

The following definition is essential in what follows.

\begin{definition}\label{def:good_trip} Under Assumption \ref{asum:open},  let $a_1\in A^n$ be non-isotropic and $a_2,a_3\in \Gamma a_1$. We say that $(a_1,a_2,a_3)$ is $A$-good if there exist $a_4\in A^n$ and $\sigma,\tau\in \Gamma$ such that $\sigma(a_1,a_3)=(a_1,a_4)$ and $\tau(a_1,a_3)=(a_2,a_4)$. Similarly, for any place $v$, if we replace $A$ by $A_v$ and $\Gamma$ by $\Gamma_v$, we get the notion of an $A_v$-good triple.
\end{definition}

The following lemma is the motivation for  Definition \ref{def:good_trip}.

\begin{lemma}\label{lem:main_idea}  Under Setting \ref{asum:open},	let $a_1,a_2 \in  A^n$ and let $M$ be a symmetric normal subset of $\Gamma$. If there are $\beta,\gamma \in M$ such that $(a_1,\beta (a_2),\gamma( a_1))$ is  $A$-good, then $a_2 \in M^3 a_1$.
\end{lemma}
\begin{proof} If  $\sigma(a_1,\gamma(a_1))=(a_1,a_4)$ and $\tau(a_1,\gamma(a_1))=(\beta(a_2),a_4)$ then
	$\delta(a_1)=a_2$ where $\delta:=\beta^{-1}\tau \gamma ^{-1} \tau ^{-1}\sigma \gamma\sigma^{-1} \in M^3$.
\end{proof}

 Lemma \ref{lem:main_idea} implies that, in order to prove Theorem \ref{thm:effective_kneser}, it is enough to show that there exists $N$ such that, for every non-isotropic $a_1$, there exists an $S$-adelic neighborhood $V$ of $a_1$ in $\Gamma a_1$ such that, for every $a_2 \in V$, there exist $\beta,\gamma \in \ccl_\gamma(\alpha)^N$ for which  $(a_1,\beta(a_2),\gamma(a_1))$ is $A$-good. \\

 We start by stating a local-to-global condition for being $A$-good. Recall that if $B$ is a commutative ring then a matrix $M \in M_k(B)$ is said to be in general position if for every  non-empty $I \subseteq \{1,\ldots,k\}$, $\det M_I\ne 0$ where $M_I$ is the principal $I$-minor of $M$. The following lemma is a reformulation of Lemma 8.1 of \cite{Kne}.

\begin{lemma}[Local-to-global principle for $A$-good triplets, {\cite[Lemma 8.1]{Kne}}] \label{lem:kne_8.1} Under Setting \ref{asum:open}, assume that   $n \ge 6 $ and $i_q(K_w^n) \ge 2$. Let $a_1,a_2,a_3 \in A^n$ be non-isotropic vectors such that $i_q\left( (K_wa_1+K_w a_2)^\perp \right)  \ge 1$, $i_q\left( (K_w a_1+K_w a_3) ^\perp \right) \ge 1$, and the matrix
\begin{equation}\label{eq:matrix}
M(a_1,a_2,a_3)=\left(\begin{matrix}
q(a_1,a_1) & q(a_1,a_2) & q(a_1,a_3) \\
q(a_2,a_1) & q(a_2,a_2) & q(a_1,a_3) \\
q(a_3,a_1) & q(a_3,a_1) & q(a_3,a_3) \end{matrix}\right) 
\end{equation}
is in general position (note that the (2,3) and (3,2) entries of the above matrix are equal to $q(a_1,a_3)$ and not to $q(a_2,a_3)$). Assume that, for every place $v$, $(a_1,a_2,a_3)$ is $A_v$-good. Then $(a_1,a_2,a_3)$ is $A$-good.
\end{lemma}

The next task is to find local conditions for being $A_v$-good.  The following Lemma is a reformulation of Lemmas 7.1, 7.2, 7.3 and 7.4 of \cite{Kne}.
\begin{lemma}[Local conditions for being $A_v$-good]\label{lem:local_conditions}Under Assumption \ref{asum:open}, assume that $n \ge 5$. Let $M(a_1,a_2,a_3)$ be the matrix defined in Equation \eqref{eq:matrix}. Let $v$ be a place.
\begin{enumerate}
\item\label{item:local_1}  Let $a_1,a_2,a_3\in A_v^n$  be non-isotropic vectors that belong to  the same $\Gamma_v$-orbit. If $v\in S$, the matrix $M(a_1,a_2,a_3)$ is in general position, and $i_q(\left( K_va_1+K_v a_3 \right)^\perp)\geq 1$, then $(a_1,a_2,a_3)$ is $A_v$-good.
\item\label{item:local_2} Let $a_1\in A_v^n$ be non-isotropic, and let $\Delta_v$ be an open subgroup of $\Gamma_v$. Then there are open sets $U_v^2,U_v^3 \subset A_v^n$ such that $U_v^2 \cap \Delta_v a_1 \neq \emptyset, U_v^3 \cap \Delta_v a_1 \neq \emptyset$ and, for every $a_2\in U_v^2 \cap \Delta_va$ and $a_3\in U_v^3 \cap \Delta_v a_1$, the matrix $M(a_1,a_2,a_3)$ is in general position and  $(a_1,a_2,a_3)$ is $A_v$-good.
\item\label{item:local_3} Assume that $v\notin S$, that $v$ is not dyadic and that $q$ is regular on $A_v^n$. Let $a_1,a_2,a_3 \in A_v^n$ be in the same $\Gamma_v$ orbit such that  $M(a_1,a_2,a_3)$ is in general position and $q(a_1) \in A_v^\times$. Then $(a_1,a_2,a_3)$ is $A_v$-good if  the following two conditions hold:
\begin{enumerate}[(i)]
\item At least one of $\disc_q(a_1,a_2)$ or $\disc_q(a_1,a_3)$ belongs to $A_v^\times$.
\item $p_v(a_1),p_v(a_2)$, as well as $p_v(a_1),p_v(a_3)$, are linearly independent over $k_v$.
\end{enumerate}
\end{enumerate}
\end{lemma}

Let $a_1 \in A^n$ be non-isotropic vector. We will apply part \ref{item:local_1} of Lemma \ref{lem:local_conditions} only for $v=w$. Part \ref{item:local_2} will be used for a finite set of places with bad properties, and part \ref{item:local_3} will be used for the remaining places. Note that, if $T \supseteq S$ is a finite set of places, then the set consisting of the vectors $a_2 \in A^n$ such that for every $v \not \in T$, $p_v(a_1),p_v(a_2)$ are linearly independent over $k_v$, is not open in the $S$-adelic topology. {Thus, if $a_1 \in A^n$ is non-isotropic, then the subset of $(\Gamma a_1)^3$ consisting of the triplets that satisfy the conditions of  Lemma \ref{lem:local_conditions}, is not open in the $S$-adelic topology and we cannot  directly  use  Lemmas \ref{lem:kne_8.1} and \ref{lem:local_conditions} in order to prove Theorem \ref{thm:effective_kneser}. Lemma \ref{lem:kneser_5.27} below allows us to overcome this issue. }

\begin{definition}Under Assumption \ref{asum:open},
\begin{enumerate}
\item Let $T_\Gamma$ be as in Definition \ref{def:T_Gamma} with respect to $\underline{\Gr}:=\Theta_q$.
\item For $\alpha \in \Gamma$, let $T_\alpha$ be the set of places $v\notin S$ for which $\langle \ccl_\Gamma(\alpha)\rangle \ne \Theta_q(A_v)$.
\item For $a \in A^n$, let $T_a$   be the set of places $v\notin S$ for which   $q(a) \notin A_v^\times$.
\item For $\alpha \in \Gamma$ and $a \in A^n$, denote $T_{\Gamma,a,\alpha}:=T_{\Gamma} \cup T_a \cup T_\alpha.$
\end{enumerate}
\end{definition}

The following lemma is an effective version of Lemma 5.2 of \cite{Kne}. 

\begin{lemma}[cf. Lemma 5.2 of \cite{Kne}]\label{lem:kneser_5.27} Under Setting \ref{nota:Gamma}, for every $n \ge 7$ and every $\epsilon>0$ there exists $N=N(n,\epsilon)$ such that  the following claim holds:

If $\alpha \in \Gamma$ is $\epsilon$-separated, then there exists an open neighborhood of the identity,  $W \subseteq \prod_{v \in T_\Gamma}\langle\ccl_\Gamma(\alpha)\rangle$, such that, for every $b_1,b_2\in A^n$  with $q(b_1)=q(b_2)\ne 0$ and every finite set of places $T\supseteq T_{\Gamma,b_1,\alpha}  \cup S$, the set of elements $\beta \in \ccl_\Gamma(\alpha)^N$ for which
\begin{enumerate}
\item\label{item:req1} $i_q \left( K_wb_1+K_w\beta b_2 \right) = 1$.
\item\label{item:req2}$p_v(b_1),p_v(\beta b_2)$ are linearly independent, for every $v\notin T$.
\end{enumerate}
contains a dense subset of $W \times \prod_{v \in T \setminus (T_\Gamma  \cup \{w\})}\langle \ccl_\Gamma(\alpha)\rangle$.
\end{lemma}

The proof of Lemma \ref{lem:kneser_5.27} is given in the next subsection. 

\begin{comment}
\begin{lemma}\label{lemma:reduction}
	Assume that Theorem \ref{thm:effective_kneser} is true under the additional assumption that $\Gamma$ is a congruence subgroup in $\Theta_q(A)$. Then Theorem \ref{thm:effective_kneser} is true in general.
\end{lemma}
\begin{proof}
	Let $\Lambda \subseteq \Theta_q(A)$ be the congruence closure of $\Gamma$ and let $\Delta \subseteq \Gamma$ be a finite index subgroup which is normal in $\Lambda$. By the congreunce subgroup orperty

	For every $v \in S_{def}$ and every non-central  $\alpha \in \Theta_q(K_v)$, there exists $\beta_v\in \Theta_q(K_v)$ such that $[\alpha_v,\beta_v]\notin Z(\Theta_q(K_v))$. A compactness argument implies that for for every $\epsilon>0$ there exists $\epsilon'>0$ such that for every $v \in S_{def}$ and every $\epsilon$-separated element  $\alpha \in \Theta_q(K_v)$, there exists $\beta_v\in \Theta_q(K_v)$ such that $[\alpha_v,\beta_v]$ is $\epsilon'$-separated. 
	
	By strong approximation we get that for every $\epsilon >0$ there exists $\epsilon'>0$ such that for  every $\epsilon$-separated element $\alpha \in \Gamma$ there exists $\beta \in \Delta$ such that $[\alpha,\beta]$ is $\epsilon'$-separated. Thus, we can assume that $\alpha \in \Delta$. 	
	 Let $\alpha \in $ 
\end{proof}
\end{comment}

\begin{proof}[Proof of Theorem \ref{thm:effective_kneser} ] The proof closely follows the proof of Theorem 6.1 of  \cite{Kne}. 

Denote $a_1:=a \in A^n$, $\Delta = \langle \ccl_\Gamma(\alpha) \rangle$, $T:=T_{\Gamma,a_1,\alpha} \cup S$. Lemma  \ref{lem:kneser_5.27} implies that there are a constant $N$ and  an open neighborhood of the identity  $W \subseteq \prod_{v \in T_\Gamma}\langle\ccl_\Gamma(\alpha)\rangle$ such that, for any $b_1,b_2\in A^n$ such that $q(b_1)=q(b_2)=q(a)$, and any finite $T'$ containing $T$, %the set of $\beta \in \ccl_\Gamma(\alpha)^N$ for which $i_q \left( K_wb_1+K_w\beta b_2 \right) = 1$ and   $p_v(b_1),p_v(\beta b_2)$ are linearly independent, for every $v\notin T'$, is dense in $W \times \prod_{v \in T' \setminus (T_\Gamma  \cup \{w\})}\langle \ccl_\Gamma(\alpha)\rangle$.

\begin{equation} \label{cond:5.27}
\parbox{10cm}{the set of $\beta \in \ccl_\Gamma(\alpha)^N$ for which $i_q \left( K_wb_1+K_w\beta b_2 \right) = 1$ and  $p_v(b_1),p_v(\beta b_2)$ are linearly independent, for every $v\notin T'$, contains a dense subset of  $W_{T'}:=W \times \prod_{v \in T' \setminus (T_\Gamma  \cup \{w\})}\langle \ccl_\Gamma(\alpha)\rangle$.}
\end{equation}

For every $v \in T_\Gamma$, choose an open normal subgroup $\Delta^*_v \subseteq \Delta_v$ such that $\prod_{v \in T_\Gamma}\Delta^*_v \subseteq W$. For every  $v \in T \setminus T_\Gamma$, denote $\Delta^*_v =\Delta_v$. Item \ref{item:local_2} of Lemma \ref{lem:local_conditions} implies that for every $v \in T\setminus \{w\}$,  there are open sets  $U_v^2,U_v^3 \subset A_v^n$ such that $U_v^2 \cap \Delta^*_v a_1 \neq \emptyset$, $ U_v^3 \cap \Delta^*_v a_1 \neq \emptyset$ and, for every $a_2\in U_v^2 \cap \Delta^*_va$ and $a_3\in U_v^3 \cap \Delta^*_v a_1$, the matrix $M(a_1,a_2,a_3)$ is in general position and  $(a_1,a_2,a_3)$ is $A_v$-good.

Let $V=\Gamma a_1 \cap \bigcap_{v\in T \smallsetminus S}U_v^2$. We will show that $\ccl(\alpha)^{3N} a_1 \supseteq V$, which implies that $\ccl(\alpha)^{6N}a_1$ contains an $S$-adelic neighborhood of $a$ in $\Gamma  a$.

Let $a_2 \in V$. Lemma \ref{lem:main_idea} implies that it is enough to find  $\beta,\gamma \in \ccl_\Gamma(\alpha)^N$ such that $(a_1,\beta a_2,\gamma a_3)$ is $A$-good. We start by finding $\beta$. Applying \eqref{cond:5.27} with $b_1=a_1$, $b_2=a_2$, and $T'=T$, and since $\left\{ \beta \in \Gamma \cap W_T \mid \beta a_2 \in U_v^2 \right\}$ is non-empty and open, we can find $\beta \in \ccl_\Gamma(\alpha)^{N}$  such that:
\begin{enumerate}[(a)]
\setcounter{enumi}{\value{my_counter}}
\item\label{item:h2} For every $v \in T \setminus \{w\}$, $\beta \in \Delta^*_v$ and $\beta a_2 \in  U^2_v$.
\item\label{item:h1} $i_q (K_wa_1 + K_w\beta a_2) = 1$ so $i_q((K_wa_1+K_w\beta a_2)^\perp)\ge 1$.
\item\label{item:h3} $p_v(a_1),p_v(\beta a_2)$ are linearly independent, for every $v\notin T$.
\setcounter{my_counter}{\value{enumi}}
\end{enumerate}

Item \ref{item:h2} and the choice of $U_v^2,U_v^3$ imply that, for every $v \in T \setminus \{w\}$ and $a_3 \in U_v^3 \cap \Delta^*_v a_1$, the triple $(a_1,\beta a_2,a_3)$ is $v$-good and $M(a_1,\beta a_2,a_3)$ is in general position. 

We now find $\gamma$. One of the requirements on $\gamma$ will be that $\gamma a_1 \in U_v^3 \cap \Delta ^*_v a_1$, for every $v \in T \smallsetminus \left\{ w \right\}$. It then follows that the triple $(a_1,\beta a_2, \gamma a_1)$ is $v$-good, for every $v\in T \smallsetminus \left\{ w \right\}$.

Since $M(a_1,\beta a_2,a_3)$ is in general position, $\disc(a_1,\beta a_2) \ne 0$. It follows that the set $T(\beta)$ consisting  of the places $v\not \in T$ for which  $\disc(a_1,\beta a_2) \not \in A^\times_v$ is finite. For every $v\in T(\beta)$, $q(a_1)\in A_v^\times$ and $\Delta_v=\Theta_q(A_v)$. Thus, for every $v\in T(\beta)$, the set of $\gamma_v\in  \Delta_v$ such that $\disc(a_1,\gamma_v a_1)\in A_v^ \times$ is non-empty and open. Applying \eqref{cond:5.27} with $b_1=b_2=a_1$ and $T'=T \cup T(\beta)$, there is $\gamma \in \ccl_\Gamma(\alpha)^N $ such that

\begin{enumerate}[(a)]
\setcounter{enumi}{\value{my_counter}}
\item\label{item:n4} For every $v \in T \setminus \{w\}$, $\gamma \in \Delta^*_v$ and $\gamma a_1 \in U_v^3$.

\item\label{item:n2} $\disc (a_1, \gamma a_1) \in  A_v^\times$  for $v \in T(\beta)$.

\item\label{item:n1} $i_q (K_wa_1 + K_w\gamma a_1) = 1$ so $i_q((K_wa_1+K_w\gamma a_1)^\perp)\ge 1$.

\item\label{item:n3} $p_v(a_1),p_v(\gamma a_1)$ are linearly independent, for every $v\notin T \cup T(\beta)$. It follows from item \ref{item:n2} that $p_v(a_1),p_v(\gamma a_1)$ are linearly independent also for $v\in T(\beta)$.
\setcounter{my_counter}{\value{enumi}}
\end{enumerate}

Since  $\gamma a_1 \in U_v^3 \cap \Delta ^*_v a_1$, $M(a_1,\beta a_2,\gamma a_1)$ is in general position and  $(a_1,\beta a_2,\gamma a_1)$ is $A_v$-good for every $v \in T \smallsetminus \{w\}$. Item \ref{item:local_1} of Lemma \ref{lem:local_conditions} and item \ref{item:n1}  imply that $(a_1,\beta a_2,\gamma a_1)$ is $w$-good. Item \ref{item:local_3} of Lemma \ref{lem:local_conditions}, items \ref{item:h3}, \ref{item:n2}, \ref{item:n3}, and the definition of $T(\beta)$ imply that $(a_1,\beta a_2,\gamma a_1)$ is $A_v$-good for every $v \not \in T$. We conclude that  $(a_1,\beta a_2,\gamma a_1)$   is $A_v$-good for every $v$. Lemma \ref{lem:kne_8.1} and items \ref{item:h1} and \ref{item:n1} imply that $(a_1,\beta a_2,\gamma a_1)$ is $A$-good. Lemma \ref{lem:main_idea} shows that $a_2 \in \ccl_\Gamma(\alpha)^{3N}$.
\end{proof}

\subsection{Proof of Lemma \ref{lem:kneser_5.27}}

\begin{lemma}[Lemmas 4.2, 4.3, 4.4, 4.5 and 4.6 of \cite{Kne}]\label{lem:kne_4} Under Setting \ref{asum:open}, assume that 
  $n =3 $ or $n \ge 5$. Let $\alpha \in\Gamma \setminus Z(\Gamma)$ and denote $\Delta:=\langle \ccl_\Gamma(\alpha) \rangle$. Then:
		\begin{enumerate}
		\item\label{item:4.2} $\Delta$ is dense in $\prod_{v \ne w}\Delta_v$.
		\item\label{item:4.3}  Let $v \in S$. Then, $\Delta_v=\Gamma_v=\Theta_q(K_v)$.
		\item\label{item:4.4}  For every $v \not \in S$, $\Delta_v$ is open in $\Theta_q(K_v)$.
		\item\label{item:4.5} For all but finitely many $v \not \in S$, $\Delta_v=\Gamma_v=\Theta_q(A_v)$.
 		\item\label{item:4.6} If $v \not \in S \cup T_{\Gamma}$, then $p_v(\Theta_q(A_v))= \Theta_q (k_v)$. 
 		\end{enumerate}
\end{lemma}

\begin{lemma}[Lemma 4.7 of \cite{Kne}]\label{lem:kne_4.7} Under Assumption \ref{asum:open}, assume that 
  $n =3 $ or $n \ge 5$. Let  $b_1,b_2 \in K^n$ be non-zero vectors and let $U$ be a non-empty open subset of $\Theta_q(\mathbb{A}^{\{w\}})$. For every $r>0$, there is a non-empty open subset $W \subseteq U$ such that every element $\alpha \in  \Theta_q(K) \cap W$, $|q(b_1,\alpha b_2)|_w>r$. In particular, if $r$ is large enough then  $i_q(K_w b_1  + K_w\alpha b_2) = 1$.
\end{lemma}

\begin{definition} For an element $\alpha \in \Theta_q(K)$, the support of $\alpha$ is the subspace $\supp(\alpha):=\left\{ a\in K^n \mid \alpha(a)=a \right\} ^{\perp}$. We say that $\alpha$ is strongly $\epsilon$-separated if $\alpha \restriction_{\supp(\gamma)}$ is $\epsilon$-separated in $\Theta_{q\restriction_{\supp(\gamma)}}(K)$.
\end{definition}

For example, any reflection is $\sqrt{2}$-separated, but not strongly $\epsilon$-separated, for any $\epsilon >0$. The following lemma is an effective version of Lemma 4.8 of \cite{Kne}.

\begin{lemma}[cf. Lemma 4.8 of \cite{Kne}]\label{lem:kne_4.8}
For every $n \ge 5$ and every $\epsilon>0$ there exists $N=N(n,\epsilon)$ such that the following claim holds:

If $\alpha \in \Gamma$ is $\epsilon$-separated then  the set $\ccl_\Gamma(\alpha)^N$ contains a strongly 1-separated element $\beta$ such that $\supp(\beta)$ is a 5-dimensional regular subspace and $i_q(\supp(\beta)_w) \ge 1$.
\end{lemma}

\begin{proof} 
The proof is by induction on $n$. Assume that $n=5$. Let $N=N(5,\epsilon)$ be the constant in Proposition \ref{cor:bounded_prod}. For every $v \in S_{def}$, choose $\beta_v \in \Theta_v(K_v)$ such that
$\dist_v(\beta_v,Z(\Theta_v))>1$ and $\supp(\beta_v)=K_v^5$.  The definition of $N$ implies that there exists $\beta \in \ccl_\Gamma(\alpha)^{N}$ such that $\beta$ is arbitrary close to $\beta_v$, for every $v \in S_{def}$. If the approximation is good enough, then $\beta$ has the required properties.

Assume $n>5$ and let $M=M(n,\epsilon)$  be the constant in Proposition \ref{cor:bounded_prod}.
For every $v\in S \setminus   \{w\}$, choose  $a_v \in A_v^n$ and  $\beta_v \in \langle\ccl_\Gamma(\alpha)\rangle_v=\Theta_q(K_v)$ such that, for $c_{v,i}:=\beta_v^i a_v$, the following hold:
\begin{enumerate} [(a)]
\item For every $v \in S_{def}$,  $c_{v,0},c_{v,1},c_{v,2},c_{v,3},c_{v,4} $ is an orthonormal basis to a regular 5-dimensional subspace.
 \item For every $v \in S \setminus (S_{def}\cup \{w\})$, $\Span_{K_{v}}\{c_{v,i} \mid 0\le i \le 4\}$ is a regular  isotropic subspace.
 \setcounter{my_counter}{\value{enumi}}
\end{enumerate}
A straightforward computation shows that, for every $v \in S_{def}$ and every $\gamma_v \in \Theta_q(K_v)$, if $\supp(\gamma_v)=\Span_{K_v}\{c_{v,0},c_{v,1}\}$,
$\gamma_v(c_{v,0})=c_{v,1}$ and $\gamma_v(c_{v,1})=-c_{v,0}$ then, for $\delta_v=\beta_v\gamma_v\beta_v^{-1}\gamma_v^{-1}$, we have $\dist_v(\delta_v \restriction_{\supp(\delta_v)} ,Z(\Theta_{q\restriction_{\supp(\delta_v)}}(K_v)))) = \sqrt{2}$.
Lemma \ref{lem:kne_4.7} implies that we can choose $a \in A_v^n$ and $\beta \in \ccl_\Gamma(\alpha)^{M}$ which are  arbitrary close to $a_v$ and $\beta_v$, for every $v \in S\setminus \{w\}$, such that $i_q(K_w a +K_w\beta a) \ge 1$. If the approximation is good enough, then for $c_i:=\beta^i a$, the following hold:
\begin{enumerate} [(a)]
\setcounter{enumi}{\value{my_counter}}
\item  $C:=\Span_{K}\{c_i \mid 0 \le i \le 2\}$ and $D:=\Span_{K}\{c_i \mid 0 \le i \le 4\}$ are regular.
\item\label{item:good_gamma} For every $v \in S_{def}$ and every $\gamma_v \in \Theta_q(K_v)$ satisfying $\supp(\gamma_v)=\Span_{K_v}\{c_{0},c_{1}\}$, $\gamma_v(c_{0})=c_{1}$, and $\gamma_v(c_{1})=-c_{0}$, we have $\supp(\beta\gamma_v\beta^{-1}\gamma_v^{-1}) \subseteq C$ (because $\supp(\beta\gamma_v\beta^{-1})=\Span_{K_v}\left\{ c_1,c_2 \right\}$) and $\dist_v(\beta\gamma_v\beta^{-1}\gamma_v^{-1}\restriction_{C},Z(\Theta_{q\restriction_C}(K_v))) > 1$.
 \item For every $v \in S \setminus (S_{def}\cup \{w\})$, $D_v$ is an isotropic subspace.
\end{enumerate}

The group $\Gamma \cap \Theta_{q\restriction_C}(K)$  is a congruence subgroup in $\Theta_{q\restriction_C}(K)$ and $i_q(C_w) \ge 1$. Hence, the strong approximation theorem and item \ref{item:good_gamma} imply that there exists $\gamma \in\Gamma \cap \Theta_{q\restriction_C}(K)$ such that  $\delta:=\beta\gamma\beta^{-1}\gamma^{-1} \in \ccl_\Gamma(\alpha)^{2M}\cap \Theta_{q \restriction_D}(K) $ is $1$-separated in $\Theta_{q\restriction_D}(K)$. Since  $\Gamma \cap \Theta_{q\restriction_D}(K)$  is a congruence subgroup in $\Theta_{q\restriction_D}(K)$ and $i_q(D_w) \ge 1$, the induction basis implies that  $\ccl_{\Theta_{q \restriction_D}(K)\cap \Gamma}(\delta)^{N} \subseteq \ccl_{\Gamma}(\alpha)^{2NM}$ contains the required element.
\end{proof}

%\begin{remark} \nir{change} The considerations that one has elements with certain properties(as $\beta$ in the case $\dim X = 3$, $a$ and $\beta$ for  $\dim X \ge 5$) first constructed over $A_v$, and then by approximation  carries over $A$ will occur several times in this paper.  We will not repeat the argument but say that the elements in question could be constructed by approximation.
%\end{remark}

\begin{proof}[Proof of Lemma \ref{lem:kneser_5.27}]  The proof closely follows the proof of Theorem 5.2 of \cite{Kne}. 
Lemma \ref{lem:kne_4.8} implies that there is $N_1$ for which there exists an $\beta \in \ccl_\Gamma(\alpha)^{N_1}$ such that:
\begin{enumerate}[(a)]
	\item $C:=\supp(\beta)$ is a regular 5-dimensional plane.
	\item $\beta$ is strongly 1-separated.
	\item $i_q(C_w)=1$.
\setcounter{my_counter}{\value{enumi}}
\end{enumerate}

By replacing $\beta$ with a conjugate element, we can assume that $b_1 \not \in C \cup C^\perp$. Denote $\Lambda:=\Gamma \cap \Theta_{q \restriction_C}(K)$, then $\Lambda$ is a congruence subgroup of $\Theta_{q \restriction_C}(K)$ and $\beta \in \Lambda$. By choosing  a free $A$-lattice $M \subseteq C$, {we get a form $\Theta_{M,q \restriction_C}$ of $\Theta_{q\restriction_C}$ defined over $A$.}
 There is a finite set $T'$ of places, disjoint form $T$, a constant $N_2$ and a neighborhood of the identity  $W' \subseteq \prod_{v \in (T \cup T')\setminus \{w\} }\Theta_{q \restriction C}(K_v)$ such that the following items hold:
\begin{enumerate}[(a)]
\setcounter{enumi}{\value{my_counter}}
	\item\label{item:5.2_0} For every $v \not \in T \cup T'$, $A_v^n \cap C_v=M_v$ and $q$ is regular on $p_v(M_v)$. In particular,  $k_v^n$ is an orthogonal sum of $p_v(M_v)$ and $p_v(M_v)^\perp$
	\item\label{item:5.2_2}  For every $v \not \in T \cup T'$,  $p_{v}(b_1) \notin p_{v}(M_v)^\perp$.
	\item\label{item:5.2_2.5}  For every $v \not \in T \cup T'$,  $\langle \ccl_\Lambda(\beta) \rangle_v=\Theta_{M,q \restriction_C}(A_v)$ and $\pi_{M,v}(\Theta_{M,q \restriction_C}(A_v))=\Theta_{M,q\restriction_C}(k_v)$.
	\item\label{item:5.2_3}  $\ccl_\Lambda(\beta)^{N_1}$  contains a dense subset of $W' \times \prod_{v \notin T \cup T' \cup\{w\}}\langle \ccl_\Lambda(\beta) \rangle_v$.
\setcounter{my_counter}{\value{enumi}}
\end{enumerate}

Indeed, item \ref{item:5.2_0} follows from the fact that for all but finitely many places $v$, $A_v^n \cap C_v=M_v$ and $q$ is regular on $M_v$.  Item \ref{item:5.2_2} follows from the assumption that $b_1 \not \in C^\perp$.  Item \ref{item:5.2_2.5} follows from Lemma \ref{lem:kne_4} applied to $\beta$ and $\Lambda$. The existence of $N_2$ and $W'$ for which Item \ref{item:5.2_3} holds follows from  Proposition \ref{cor:bounded_prod} applied to $\beta$ and $\Lambda$.

Proposition \ref{cor:bounded_prod} applied to $\Gamma$ and $\alpha$  shows that there  exit a constant $N_3$ and an open neighborhood of the identity $W \subseteq \prod_{v\in T_\Gamma} \Theta_q(K_v)$ such that $\ccl_\Gamma(\alpha)^{N_3}$ contains a dense subset of $W \times \prod_{v \notin T_{\Gamma}\cup \left\{ w \right\} } \langle \ccl_\Gamma(\alpha) \rangle_v$. Let  $U$ be a non-empty open subset of $W \times \prod_{v \in T \setminus (T_{\Gamma}\cup \left\{ w \right\} )} \langle \ccl_\Gamma(\alpha) \rangle_v$. We will show that  the intersection of  $\ccl_{\Gamma}(\alpha)^{N_1N_2+N_3}$ with $U$ contains an element which satisfies the desired properties.

Fix some place $u \notin T \cup T'$. {Then $q(b_1)=q(b_2)\in A_u^\times$ so $p_v(b_2) \ne 0$.} 
Since $\langle\ccl_\Gamma(\alpha) \rangle_u=\Theta_u$ and $\dim C^\perp \ge 2$, there exist $\gamma_u \in \langle\ccl_\Gamma(\alpha) \rangle_u$ such that $b_1+C_u$ and $\gamma_u b_2+C_u$ are linearly independent in $K_u^n/C_u$. For every $v \in T'$, $p_v(\langle \ccl_\Gamma(\alpha)\rangle_v )=\Theta_q(k_v)$, $q$ is regular on $k_v^n$ and $q(b_1)=q(b_2)\in A_v^\times$, thus $p_v(b_2) \ne 0$ and there exists $\gamma_v \in \langle \ccl_\Gamma(\alpha)\rangle_v$ for which $p_v(b_1)$ and  $p_v(\gamma_v b_2)$ are linearly independent over $k_v$. Approximation at the places in $T \cup T' \cup \{u\} $ implies that there is $\gamma \in \ccl_\Gamma(\alpha)^{N_3}$ with the following properties:
\begin{enumerate}[(a)]
\setcounter{enumi}{\value{my_counter}}
\item\label{item:5.2_3.5} $\gamma \in U$.
\item\label{item:5.2_5} $b_1+C$ and $\gamma b_2+C$ are linearly independent in $K^n/C$.
\item\label{item:5.2_6} $p_v(b_1)$ and  $p_v(\gamma b_2)$ are linearly independent over $k_v$ for $v \in  T'$.
\setcounter{my_counter}{\value{enumi}}
\end{enumerate}

Item \ref{item:5.2_5} implies that there is a finite set  of places $T''$ which is disjoint from $T \cup T'$ such that:
\begin{enumerate}[(a)]
\setcounter{enumi}{\value{my_counter}}
    \item\label{item:5.2_8}  For every $v \not \in T \cup T' \cup T''$, the images $p_v(b_1)+p_v(M_v)$ and  $p_v(\gamma b_2)+p_v(M_v)$ are linearly independent in $k_v^n/p_v(M_v)$.
\setcounter{my_counter}{\value{enumi}}
\end{enumerate}

For every $v \in T''$, $q(b_1)=q(b_2)\in A_v^\times$ so $p_v(b_2) \ne 0$. If  $v \in T''$ and $p_v(\gamma b_2) \in p_{v}(M)^\perp$,  denote $\gamma'_v=\id $. Item \ref{item:5.2_2} implies that $p_v(b_1)$ and $p_v(\gamma_v'\gamma b_2)$ are linearly independent over $k_v$. If $v \in T''$ and $p_v(\gamma b_2) \notin p_v(M)^\perp $, then items \ref{item:5.2_0}, \ref{item:5.2_2} and   \ref{item:5.2_2.5} imply that there is  $\gamma_v' \in\Theta_{M,q \restriction_C}(A_v)$ such that $p_v(b_1)$ and $p_v(\gamma_v'\gamma b_2)$ are linearly independent over $k_v$. Let $V$ be an open subset of $\Theta_{q\restriction_C}(\mathbb{A}^{\{w\}})$ such that for every $\gamma' \in \Theta_{q\restriction_C}(K) \cap V $:

\begin{enumerate}[(a)]
\setcounter{enumi}{\value{my_counter}}
	\item\label{item:5.2_9}   For every $v \in T  \setminus  \{w\}$, $\gamma'$ is  so close to $1$ such that $\gamma'\gamma \in U$.
	\item\label{item:5.2_10}  For every  $v \in T'$, $\gamma'$ is  so close to $1$ such that  $p_v(b_1)$ and  $p_v(\gamma'\gamma b_2)$ are linearly independent over $k_v$.
	\item\label{item:5.2_11}  For $v \in T''$, $\gamma'$ is so close to $\gamma'_v$ such that  $p_v(b_1)$ and  $p_v(\gamma'\gamma b_2)$ are linearly independent over $k_v$.
\end{enumerate}

Let $c_1$ and $c_2$ be the orthogonal projections of $b_1$ and $\gamma(b_2)$ to $C$.  Items \ref{item:5.2_2.5} and \ref{item:5.2_3} and Lemma \ref{lem:kne_4.7}  applied to  $\beta$ and $\Lambda$, imply that there exists $\gamma' \in \ccl_\Lambda(\beta)^{N_2} \cap V$ such that  $q(c_1,\gamma'c_2)$  is arbitrary large. If $q(c_1,\gamma'c_2)$ is large enough then  $i_q(K_w b_1 + K_w \gamma'\gamma b_2)=1$.  Denote $\delta:=\gamma'\gamma \in \ccl_\Lambda(\beta)^{N_2}\ccl_\Gamma(\alpha)^{N_3} \subseteq   \ccl(\alpha)^{N_1N_2+N_3}$. Item \ref{item:5.2_9} implies that $\delta \in U$. The linear independence of $p_v(b_1)$ and $ p_v(\delta b_2)$ follows for $v  \in  T'$ from item \ref{item:5.2_10} ,  for $v \in T''$ from item \ref{item:5.2_11} and for $ v \not \in T \cup T' \cup T''$ from item \ref{item:5.2_8}.
\end{proof}

\subsection{Proof of Lemma \ref{lemma:uff}}\label{subsection:uff}

\begin{lemma}[cf. Lemma 4.10 of \cite{Kne}]\label{lem:kne_4.10} Under Setting \ref{asum:open}, assume that $i_q(K_w) \ge 2$ and 
$n \ge 6$. For every $\epsilon>0$ there exists a constant $N=N(n,\epsilon)$ (in particular, $N$ does not depend on $q$ nor on $\Gamma$) such that the following claim hold:

If $\alpha$ is $\epsilon$-separated and  $a \in K^n$ is non-isotropic, then $\ccl_\Gamma(\alpha)^N$ contains an element which fixes $a$ and is strongly $1$-separated.
\end{lemma}

\begin{proof}
Let $N:=N(n,\epsilon)$ be as in Proposition \ref{cor:bounded_prod}.  Lemma \ref{lem:kne_4.7} and approximation imply that there exist $\beta \in \ccl_\Gamma(\alpha)^N$ and $\gamma_v \in  \Theta_{q\restriction_{(K_va+K_v \beta a)^\perp}}(K_v)$, for every $v \in S_{def}$, such that:
\begin{enumerate}[(a)]
\item  $Ka+K\beta a + K \beta^2 a$ is a regular 3-dimensional subspace. 
\item  $i_q (K_wa+K_w\beta a) = 1$.
\item For every $v \in S_{def}$,  $\dist_v((\beta^{-1}\gamma_v^{-1}\beta\gamma_v)\restriction_{(K_va)^\perp},Z(\Theta_{q\restriction_{(K_va)^\perp}}(K_v)))>1$.
\end{enumerate}
Since $i_q((K_wa+K_w\beta a)^\perp) \ge 1$ and $\Lambda:=\Gamma \cap \Theta_{q\restriction_{(Ka+K\beta a)^\perp}}(K)$ is a congruence subgroup of  $\Theta_{q\restriction_{(Ka+K\beta a)^\perp}}(K)$, the strong approximation theorem implies that there exists  $\gamma \in \Lambda$ which is arbitrary close to $\gamma_v$, for every $v \in S_{def}$. If the approximation is good enough, then $\delta:=\beta^{-1}\gamma^{-1}\beta\gamma \in  \ccl_\Gamma(\alpha)^{2N}$ fixes
$a$ and the restriction of $\delta$ to $(Ka)^\perp$  is 1-separated in $\Theta_{q\restriction_{(Ka)^\perp}}(K)$. Lemma \ref{lem:kne_4.8} implies that there exists $M=M(n)$ such that $\ccl_{\Gamma \cap \Theta_{q\restriction_{(Ka)^\perp}}(K)}(\delta)^{M}\subseteq \ccl_\Gamma(\alpha)^{2NM}$ contains the required element.\end{proof}

\begin{corollary}\label{cor:simplified} Under Setting \ref{asum:open}, assume that $i_q(K_w) \ge 2$ and $n \ge 6$.
	For every $\epsilon>0$ there exists a constant $N=N(n,\epsilon)$ such that the following claim hold:

Let $d \le n-6$ and let $a_0,a_1,\ldots,a_d \in K^n$ be non-isotropic orthogonal vectors such that $i_q((K_wa_0+\ldots K_w a_{d-1})^\perp) \ge 2 $. If $\alpha \in \Gamma$ is $\epsilon$-separated and $\mathfrak{q} \lhd A$, then there exists  $\beta \in \ccl_\Gamma(\alpha)^{N} \cap  \Theta_q(A;\mathfrak{q})$ which is strongly $1$-separated and fixes $a_0,\ldots,a_d$. 
\end{corollary} 
 \begin{proof}   Denote  $\epsilon'=\min(1,\epsilon)$. For every $0 \le k \le d$, denote $U_k:=(\Span_K\{a_i \mid 0 \le i \le k-1 \})^\perp$ and $\Gamma_k:=\Gamma\cap \Theta_{q\restriction_{U_k}}(K)$. For every $6 \le n' \le n$ let $N(n',\epsilon')$ be the constant given by  Lemma \ref{lem:kne_4.10} with respect to $n'$ and $\epsilon'$ . Denote $N:=\max\{N(n',\epsilon') \mid 6  \le n' \le n\}$. 
 
We will prove by induction on $1 \le k \le d$ that  there exists a strongly $1$-separated $\alpha_k \in \ccl_\Gamma(\alpha)^{N^k}\cap \Theta_{q\restriction_{U_k}}(K)$. The case  $k=1$ follows from the definition of $N$. Assume that the claim is true for some $1 \le k < d$.
Since $\Gamma_k$ is a congruence subgroup in  $\Theta_{q\restriction_{U_k}}(K)$, by the definition of $N$, there exists a strongly $1$-separated $$\alpha_{k+1} \in \ccl_{\Gamma_k}(\alpha_k)^N \cap \Theta_{q\restriction_{U_{k+1}}}(K) \subseteq  \ccl_{\Gamma}(\alpha)^{N^{k+1}} \cap \Theta_{q\restriction_{U_{k+1}}}(K).$$

For every $5 \le n' \le n$ let $M(n',1)$ be the constant given by Proposition \ref{cor:bounded_prod} and denote $M=\max_{n'}M(n',1)$. By the definition of $M$, there exists a strongly $1$-separated 
$$\beta \in \ccl_{\Gamma_d}(\alpha_d)^M  \cap \Theta_q(A;\mathfrak{q})\subseteq \ccl_{\Gamma}(\alpha)^{N^dM} \cap \Theta_{q\restriction_{U_{d}}}(K)\cap  \Theta_q(A;\mathfrak{q}).$$ 

\end{proof}

\begin{proof}[Proof of Lemma \ref{lemma:uff}] Denote  $\epsilon'=\min(\epsilon,1)$ and let $N:=N(n,\epsilon')$ be the the constant given by Corollary \ref{cor:simplified}. For every $n-2 \le n'\le n$, let $M(n',\epsilon')$ be the constant given by Theorem \ref{thm:effective_kneser} and denote $M:=\max_{n'}M(n',\epsilon')$. For every $0 \le r \le 2$, let  $U_r:=(\Span_K\{c_i \mid 0 \le i \le r-1\})^\perp$ and $\Gamma_r:=\Gamma \cap \Theta_{q\restriction_{U_r}}(K)$. Note that $\Gamma_r$ is a congruence subgroup in  $\Theta_{q\restriction_{U_r}}(K)$, that $U_0=K^n$, and that $\Gamma_0=\Gamma$.

By the definition of $N$, there exists a strongly 1-separated element $\alpha_2 \in \ccl_\Gamma(\alpha)^N \cap \Theta_{q\restriction_{U_2}}(K)$. By the definition of $M$, there exists  $0 \ne \mathfrak{q}_2 \lhd A$ such that $\ccl_{\Gamma_2}(\alpha_2)^Mc_2$ contains the $\mathfrak{q}_2$-th neighborhood of $c_2$ in $\Gamma_2c_2$.

By the definition of $N$, there exists a strongly 1-separated element $\alpha_1 \in \ccl_\Gamma(\alpha)^N \cap \Theta_{q\restriction_{U_1}}(K)\cap \Theta_q(A;\mathfrak{q}_2)$. By the definition of $M$, there exists  $0 \ne \mathfrak{q}_1 \lhd A$ such that $ \mathfrak{q}_1 \subseteq \mathfrak{q}_2 $ and $\ccl_{\Gamma_1}(\alpha_1)^Mc_1$ contains the $\mathfrak{q}_1$-th neighborhood of $c_1$ in $\Gamma_1c_1$.

By the definition of $N$, there exists a strongly 1-separated element $\alpha_0 \in \ccl_\Gamma(\alpha)^N \cap \Theta_{q\restriction_{U_0}}(K)\cap \Theta_q(A;\mathfrak{q}_1)$. By the definition of $M$, there exists  $0 \ne \mathfrak{q}_0 \lhd A$ such that $ \mathfrak{q}_0 \subseteq \mathfrak{q}_1 $ and $\ccl_{\Gamma_0}(\alpha_0)^Mc_0$ contains the $\mathfrak{q}_0$-th neighborhood of $c_0$ in $\Gamma_0c_0$. 

We claim that the $\mathfrak{q}_0$-th  neighborhood $J$ of $(c_0,c_1,c_2)$ in $\Gamma(c_0,c_1,c_2)$ is contained in $\ccl_\Gamma(\alpha)^{3MN}(c_0,c_1,c_2)$. For every $0 \le r \le 2$, denote the $\mathfrak{q}_r$-th neighborhood of $c_r$ in $\Gamma_rc_r$ by $J_r$. 
Let $(b_0,b_1,b_2)\in J$. There exists $\beta_0 \in \ccl_{\Gamma_0}(\alpha_0)^M$ such that $\beta_0c_0=b_0$. Since 
$\alpha_0 \in  \Theta_q(A;\mathfrak{q}_1)$, $\beta_0^{-1}b_1\in J_1$. Thus, there exists $\beta_1 \in \ccl_{\Gamma_1}(\alpha_1)^M$ such that $\beta_1c_1=\beta_0^{-1}b_1$. Since $\alpha_0 ,\alpha_1 \in  \Theta_q(A;\mathfrak{q}_2)$ ,  $\beta_1^{-1}\beta_0^{-1}b_2 \in J_2$. Thus, there exists $\beta_2 \in \ccl_{\Gamma_2}(\alpha_2)^M$ such that $\beta_2c_2=\beta_1^{-1}\beta_0^{-1}b_2$. It follows that $\beta_0\beta_1\beta_2(c_0,c_1,c_2)=(b_0,b_1,b_2)$ and $\beta_0\beta_1\beta_2\in \ccl_\Gamma(\alpha)^{3MN}$.
\end{proof} 

\appendix

\section{Bi-interpretability of $A$ and $\PSL_n(A)$}

\begin{setting}\label{nota:appen}
$n \ge 3$ is an integer, $A$ is an infinite integral domain of Krull dimension $d <\infty$ which has trivial Jacobson radical and $\PSL_n(A):=\SL_n(A)/Z(\SL_n(A))$.
\begin{enumerate}
	\item 	For every $\mathfrak{q} \lhd A$, $\rho_{\mathfrak{q}}:\SL_n(A)\rightarrow \SL_n(A/\mathfrak{q})$ is the  quotient map,
	 $\SL_n(A;\mathfrak{q}):=\ker \rho_{\mathfrak{q}}$ is the $\mathfrak{q}$-th congruence subgroup and  $\SL^*_n(A;\mathfrak{q}):=\rho_\mathfrak{q}^{-1}(Z(\SL_n(A/\mathfrak{q})))$.
	\item For every $1 \le i \ne j \le n$, $a \in A$ and $\mathfrak{q} \lhd A$, ${e}_{i,j}(a) \in \SL_n(A)$ is the matrix with $1$ on the diagonal, $a$ in the $(i,j)$-entry and zero elsewhere, $e_{i,j}:=e_{i,j}(1)$, $E_{i,j}(A):=\{e_{i,j}(b)\mid b \in A\}$, $U(A):=E_{1,2}(A)E_{1,3}(A)\cdots E_{1,n}(A)$,
 $E_{i,j}(A;\mathfrak{q}):=E_{i,j}(A) \cap \SL_n(R;\mathfrak{q})$ and $U(A;\mathfrak{q}):=U(A) \cap \SL_n(A;\mathfrak{q})$. Finally, for every  $1 \le i \ne j \le n$, let $p_{i,j} \in \SL_n(A)$ be a permutation matrix such that for every $a \in A$, $p_{i,j}e_{1,n}(a)p_{i,j}^{-1}=e_{i,j}(a)$. 

	\item $\PSL_n(A;\mathfrak{q})$ and $\PSL_n^*(A;\mathfrak{q})$ are the images in $\PSL_n(A)$ of $\SL_n(A;\mathfrak{q})$  and $\SL^*_n(A;\mathfrak{q})$ respectively.  By abuse of notation, we denote by $e_{i,j}(a)$, $e_{i,j}$, $p_{i,j}$, $E_{i,j}(A)$, $U(A)$, $E_{i,j}(A;\mathfrak{q})$ and $U(A;\mathfrak{q})$ the images of $e_{i,j}(a)$, $e_{i,j}$, $p_{i,j}$, $E_{i,j}(A)$, $U(A)$, $E_{i,j}(A;\mathfrak{q})$ and $U(A;\mathfrak{q})$ in $\PSL_n(A)$.
\end{enumerate}	
 \end{setting}
	 
\begin{lemma}\label{lemma:ad_good}
	Under Setting \ref{nota:appen}, let $\mathfrak{q} \lhd A$ be a maximal ideal and let $e_1,\ldots,e_n$ be the standard basis of $A^n$. Let $\beta \in \SL_n(A)$ be such that $\beta $ is equivalent  to $e_{2,1}$ modulo $\mathfrak{q}$. Then there exists $\gamma \in \SL_n(A)$ which fixes the vectors  $e_1$ and $\beta e_1$ and is equivalent  to $e_{1,3}$ modulo $\mathfrak{q}$.
\end{lemma}	 
\begin{proof} Let $K$ be the field of fractions of $A$. Denote $e_2^*=\beta e_1$ and for every $1 \le i \ne 2 \le n$, $e_i^*:=e_i$. Then $e^*_1, ,\ldots,e^*_n$ is a basis of $K^n$. Let $\delta$ be the matrix such that, for every $1 \le i \le n$, the $i$th column of $\delta$ is $e_i^*$. Then $\delta \in M_n(A) \cap \GL_n(K)$ and $\delta$ is equivalent to $e_{1,2}$ modulo $q$.  Denote the $(2,2)$-coordinate of $\delta$ by $a$.  Then $a$ is equivalent to 1 modulo $\mathfrak{q}$ and  $\gamma:=\delta  e_{1,3}(a) \delta^{-1}\in \SL_n(A)$ is the required matrix. 
	
\end{proof}
	 
\begin{lemma}\label{lemma:app_br}  Under Setting \ref{nota:appen}, 
if $\mathfrak{q}\lhd A$ is a maximal ideal and $\alpha \notin \SL^*_n(A;\mathfrak{q})$ then $\gcl_{\SL_n(A)}(\alpha)^{32} \cap U(A)$ is not contained in $U(A;\mathfrak{q})$. 	
\end{lemma}
\begin{proof}

By Gauss elimination, $\SL_n(A)$ projects onto $\SL_n(A/\mathfrak{q})$. It is easy to see that, if $F$ is a field and $g\in \SL_n(F)$ in not-central, then $e_{2,1}\in \gcl_{\SL_n(F)}(g)^8$ (cf. the proof of Lemma 2.9 in \cite{ALM}). Choose $\beta \in \gcl_{\SL_n(A)}(\alpha)^8$ such that  $\beta $ is equivalent to $e_{2,1}$ modulo $\mathfrak{q}$. Lemma \ref{lemma:ad_good} implies that there exists  $\gamma \in \SL_n(A)$ that fixes $e_1$ and $\beta e_1$ and is equivalent  to $e_{1,3}$ modulo $\mathfrak{q}$. Then $\eta:=[\beta,\gamma]:=\beta^{-1}\gamma^{-1}\beta \gamma \in \gcl_{\SL_n(A)}(\alpha)^{16} $ fixes $e_1$ and is equivalent to $e_{2,3}$ modulo $q$. It follows that $[e_{1,2},\eta] \in  \gcl_{\SL_n(A)}(\alpha)^{32} \cap U(A)$ is equal to $e_{1,3}$ modulo $\mathfrak{q}$. 
\end{proof}

\begin{lemma}\label{lemma:app_int}  Under Setting \ref{nota:appen}, 
\begin{enumerate}
	\item $Z(\Cent_{\PSL_n(A)}(e_{1,n}))=E_{1,n}(A)$.
	\item For every $a,b\in A$, $e_{1,n}(a)e_{1,n}(b)=e_{1,n}(a+b)$.
	\item For every $a,b \in A$, $[p_{1,n-1}e_{1,n}(a)p_{1,n-1}^{-1},p_{n-1,n}e_{1,n}(b)p_{n-1,n}^{-1}]=e_{1,n}(ab)$.
\end{enumerate}
	In particular, $\PSL_n(A)$ interprets the ring $A$. 
\end{lemma}
\begin{proof}
	The proof consists of simple computations which are omitted. 
\end{proof}

\begin{lemma}\label{lemma:app_uni} Under Setting \ref{nota:appen}, The collections $\{U(A;\mathfrak{q}) \mid \mathfrak{q}\lhd A\text{ is maximal} \}$
	and  $\{\PSL^*_n(A;\mathfrak{q}) \mid \mathfrak{q}\lhd A\text{ is maximal} \}$ are uniformly definable.
\end{lemma}
\begin{proof}
	Lemma 1.5 of \cite{AKNS} implies that $\{\mathfrak{q} \lhd A \mid \mathfrak{q}\text{ is maximal}\}$ is uniformly definable in $A$. Lemma \ref{lemma:app_int} implies that  the family $\{E_{1,n}(A;\mathfrak{q}) \mid \mathfrak{q}\lhd A\text{ is maximal} \}$ is uniformly definable in $\PSL_n(A)$. For every ideal $\mathfrak{q} \lhd A$, $$U(A;\mathfrak{q})=\prod_{2 \le j \le n}E_{1,j}(A;\mathfrak{q})=\prod_{2 \le j \le n} p_{1,j}E_{1,n}(A;\mathfrak{q})p_{1,j}^{-1}$$
so $\{U(A;\mathfrak{q}) \mid \mathfrak{q}\lhd A\text{ is maximal} \}$ is uniformly definable.  Lemma \ref{lemma:app_br} implies that  for every maximal ideal $\mathfrak{q}\lhd A$,  $\alpha \in \PSL^*_n(A;\mathfrak{q})$ if and only if $\gcl_{\PSL_n(A)}(\alpha)^{32} \cap U(A) \subseteq U(A;\mathfrak{q})$. Thus, $\{\PSL^*_n(A;\mathfrak{q}) \mid \mathfrak{q}\lhd A\text{ is maximal} \}$ is uniformly definable. 
\end{proof}

\begin{theorem}\label{thm:app} Under Setting \ref{nota:appen}, $A$ and $\PSL_n(A)$ are bi-interpretable. 
\end{theorem}
\begin{proof}
	Denote $E=E_{1,n}(A)$ and let $\mathscr{E}=(E,e)$ be the interpretation of $A$ in $\PSL_n(A)$ given by Lemma \ref{lemma:app_int}. Let $\mathscr{S}=(S,s)$ be the standard interpretation of $\PSL_n(A)$ in $A$ (i.e. as $n \times n$ matrices with determinant 1, up to scalars). Viewing $\PSL_n(A)$ as an imaginary in $\SL_n(A)$, we get that $\mathscr{P}=(\PSL_n(A),\Id)$ is an interpretation of $\PSL_n(A)$ in $\SL_n(A)$ and $\mathscr{C}=(C,c):=\mathscr{P}\circ \mathscr{S}$ is an interpretation of $\PSL_n(A)$ in $A$. It is easy to see that $\mathscr{E}\circ \mathscr{C}$ is trivial. Therefore, in order to show that $A$ and $\PSL_n(A)$ are bi-interpretable, it is enough to show that $\mathscr{D}=(D,d):=\mathscr{C}\circ \mathscr{E}$ is trivial. 
	
By construction, the restriction of $d^{-1}$ to $E_{1,n}(A)$ is definable. Therefore, the restriction of $d^{-1}$ to 
$V:=\prod_{1 \le j \le n}\prod_{1 \le i \ne j \le n}p_{i,j}E_{1,n}(A)p_{i,j}^{-1}$ is definable. By Gauss elimination, there is a constant $C$ such that, for every maximal ideal $\mathfrak{q}$,  $V^C$ projects onto $\PSL_n(A)/\PSL^*_n(A;\mathfrak{q})\cong \PSL_n(A/\mathfrak{q})$. 
	
Lemma \ref{lemma:app_uni} implies that there are definable sets $I$, $J$, $X \subseteq I \times \PSL_n(A)$ and $Y \subseteq J \times D$ such that 
$\{X_i \mid i \in I\}=\{\PSL^*_n(A;\mathfrak{q}) \mid \mathfrak{q}\lhd A\text{ is maximal} \}$ and $\{Y_j \mid j \in J\}=\{d^{-1}(X_i) \mid i \in I\}$. We claim that there exists a definable  $Z \subseteq I \times J$ such that $(i,j)\in Z$ if and only if $Y_j=d^{-1}(X_i)$. Indeed, $d^{-1}$ is definable on $U(A) \subseteq V$ and $Y_j=d^{-1}(X_i)$ if and only if $Y_j \cap d^{-1}(U(A))=d^{-1}(X_i\cap U(A))$. 

Finally, $d^{-1}$ is definable since for every $\alpha \in \PSL_n(A)$ and $\delta \in D$, $d^{-1}(\alpha)=\delta$ if and only if for every $(i,j)\in Z$, there exists $\nu\in V$ such that  $\alpha X_i=\nu X_i$ and $\delta Y_j=d^{-1}(\nu)Y_j$. 
\end{proof}

\end{document}